\documentclass[12pt]{article}
\usepackage{graphicx}
\usepackage{natbib} 
\usepackage{url} 
\usepackage{amssymb,amsmath,amsthm}
\usepackage{booktabs} 
\usepackage{multirow}
\usepackage{enumitem}
\usepackage{hyperref}

\newtheorem{theorem}{Theorem}
\newtheorem{corollary}{Corollary}

\newtheorem{lemma}{Lemma}
\newtheorem{proposition}{Proposition}

\newcommand{\Exp}[1]{{\rm{E}}[ #1 ]}

\newcommand{\Cov}[1]{{\rm{Cov}}[ #1 ]}
\newcommand{\tr}{\text{\rm trace}}

\newcommand{\bl}[1]{{\mathbf #1}}
\newcommand{\bs}[1]{\boldsymbol #1}

\begin{document}

\def\spacingset#1{\renewcommand{\baselinestretch}%
{#1}\small\normalsize} \spacingset{1}

\title{\bf Subscedastic weighted least squares estimates}
\author{Jordan Bryan,
    Haibo Zhou, 
    Didong Li\\
    Department of Biostatistics, University of North Carolina at Chapel Hill,\\  Chapel Hill, North Carolina 27599, U.S.A.}
\date{}
\maketitle

\bigskip

\begin{abstract}
In the heteroscedastic linear model, the weighted least squares (WLS) estimate of the model coefficients is more efficient than the ordinary least squares (OLS) estimate. However, the practical application of WLS is challenging because it requires knowledge of the error variances. Feasible weighted least squares (FLS) estimates, which use approximations of the variances when they are unknown, may either be more or less efficient than the OLS estimate depending on the quality of the approximation. A direct comparison between FLS and OLS has significant implications for the application of regression analysis in varied fields, yet such a comparison remains an unresolved challenge. In this study, we address this challenge by identifying the conditions under which FLS estimates using fixed weights demonstrate greater efficiency than the OLS estimate. These conditions provide guidance for the design of feasible estimates using random weights. They also shed light on how certain robust regression estimates behave with respect to the linear model with normal errors of unequal variance.
\end{abstract}

\section{Introduction}\label{sec:intro}

Accounting for heteroscedasticity to improve the precision of regression estimates is an old, but not outdated practice. Indeed, modern statistical methods are still being adapted to incorporate information about heterogeneous variance in outcome variables. For instance, \cite{shah_stabilized_2023} develop a consistent estimate of the error variance in a model for individualized treatment rules in order to stabilize their parameter estimates. In a different setting, \cite{bryan_routine_2023} and \cite{bryan_linear_2024} apply principles of variance estimation to devise more efficient estimates of water quality using fluorescence spectroscopy data. The question of how to address heteroscedasticity has continued to inspire new methodological developments primarily because, while classical least squares theory provides optimal estimates when the error variances are known, optimal procedures are more difficult to identify when the error variances must be estimated.

Such challenges arise even in the context of the standard linear model with independent errors:
\begin{equation}\label{eq:smod}
    \bl{y} = \bl{X} \bs{\beta} + \bs{\varepsilon},
\end{equation}
where $\bl{y} \in \mathbb{R}^n$, $\bl{X} \in \mathbb{R}^{n \times p}$ is known and full-rank, and
\begin{equation*}
    \Exp{\bs{\varepsilon}} = \bl{0},~\Cov{\bs{\varepsilon}} = \Exp{\bs{\varepsilon} \bs{\varepsilon}^\top} = \bs{\Omega} = \mathrm{diag}(\omega_1, \dots, \omega_n).
\end{equation*}
In this setting, the \textit{weighted least squares estimate} (WLS) $(\bl{X}^\top \bs{\Omega}^{-1}\bl{X})^{-1} \bl{X}^\top \bs{\Omega}^{-1} \bl{y}$ has minimum variance among all linear unbiased estimates of $\bs{\beta}$. However, computing the WLS estimate requires knowledge of $\bs{\Omega}$, which is unknown in practice. By contrast, the \textit{ordinary least squares (OLS) estimate}  $(\bl{X}^\top \bl{X})^{-1} \bl{X}^\top \bl{y}$ can be computed in practice, since it is a function of $\bl{X}$ and $\bl{y}$ alone. OLS, though, can be significantly less efficient than WLS if there is a high degree of heteroscedasticity. 

A so-called \textit{feasible weighted least squares (FLS) estimate} of $\bs{\beta}$ is obtained by plugging a computable estimate of $\bs{\Omega}$, denoted by $\tilde{\bs{\Omega}}$, into the vector-valued function $b_{\bl{X}} : \mathcal{D}_+^{n} \rightarrow \mathbb{R}^p$, defined as
\begin{equation}\label{eq:gls_estimate}
    b_{\bl{X}}(\tilde{\bs{\Omega}}) = (\bl{X}^\top \tilde{\bl{\Omega}}^{-1} \bl{X})^{-1} \bl{X}^\top \tilde{\bl{\Omega}}^{-1} \bl{y},
\end{equation}
where $\mathcal{D}_+^n$ denotes the set of $n \times n$ diagonal positive definite matrices. This function yields the OLS estimate, $b_{\bl{X}}(\bl{I}_n)$, and the WLS estimate, $b_{\bl{X}}(\bs{\Omega})$, as special cases. FLS estimates have the benefit of being computable, and they have the potential to be more efficient than the OLS estimate. However, they also have the potential to be arbitrarily less precise than the OLS estimate if $\tilde{\bs{\Omega}}$, the feasible substitute for $\bs{\Omega}$, is far from the truth. 

{As only feasible and ordinary least squares estimates are available in practical settings, it is important to understand when the extra effort of designing feasible weights leads to a gain in efficiency relative to OLS. However, FLS estimates using random weights generally depend on $\bl{y}$ in a non-linear fashion, which makes explicit derivation of their covariance matrices difficult. Perhaps in part because of this difficulty, one approach to dealing with heteroscedastic errors in the linear model has been to focus on estimating standard errors for the OLS estimate, which are asymptotically valid even when the form of heteroscedasticity is unknown. \cite{white_heteroskedasticity-consistent_1980} devised such standard errors, and subsequent work from \cite{arellano_computing_1987} extended this approach to the case of temporal dependence between errors corresponding to repeated measures. \cite{driscoll_consistent_1998} and \cite{vogelsang_heteroskedasticity_2012} address consistent standard error estimation under both temporal and spatial dependence. While this line of research has provided several clever means of asymptotically valid inference, it generally does not address the issue of point estimation for the linear model coefficients, as pointed out by \cite{romano_resurrecting_2017}. One exception is \cite{liang_longitudinal_1986}, although they address estimation efficiency through simulation rather than theory.}

{On the other hand, most of the FLS literature related to point estimation uses the WLS estimate $b_{\bl{X}}(\bs{\Omega})$, not the OLS estimate $b_{\bl{X}}(\bl{I}_n)$, as reference. Early work on the finite-sample properties of FLS established upper bounds on the inefficiency of the OLS estimate \citep{anderson_theory_1948, watson_linear_1967, watson_prediction_1972, knott_minimum_1975} or fixed-weight FLS estimates \citep{khatri_extensions_1981, wang_kantorovich-type_1989} relative to the WLS estimate. \cite{kurata_least_1996} used the Loewner partial order to bound the covariance matrix of a limited class of FLS estimates between two scalar multiples of the optimal $\Cov{b_{\bl{X}}(\bs{\Omega})}$. Other authors, including \cite{fuller_estimation_1978}, \cite{carroll_transformation_1988}, and \cite{hansen_asymptotic_2007} examined the asymptotic properties of feasible weighted and generalized least squares estimates. These authors use the optimal $b_{\bl{X}}(\bs{\Omega})$ as the point of reference in the sense that they consistently estimate the true error covariance matrix, so that their estimates have the same asymptotic properties as $b_{\bl{X}}(\bs{\Omega})$.}

{In this article, we offer new insights into the comparison of FLS estimates to the OLS estimate in the context of point estimation. Our starting point is the finite-sample perspective initiated by \cite{szroeter_exact_1994}, in which we consider fixed-weight FLS estimates. We then use these results to draw conclusions about certain large-sample cases. Our primary contribution is to characterize a class of FLS estimates, which are guaranteed to be more efficient than the OLS estimate. We then demonstrate that certain robust regression estimates that are simple to compute satisfy a type of oracle efficiency property with respect to this class. This latter result puts our work in conversation with recent articles \citep{feng_optimal_2024, wiens_ignore_2024} that make explicit comparisons between OLS and alternatives; however, these articles do not discuss heteroscedasticity per-se.}

{In the case of a single regressor, we use variance as the measure of an estimate's efficiency. In the case of multiple regressors, we primarily work with the \textit{generalized variance}, although we show that our results also apply to the \textit{total variance} as well}. Following \cite{bloomfield_inefficiency_1975}, we define the generalized variance of a multivariate estimate as the determinant of its covariance matrix. For any FLS estimate that uses non-random weights $\tilde{\bs{\Omega}}$, the covariance matrix of $b_{\bl{X}}(\tilde{\bs{\Omega}})$ with respect to \eqref{eq:smod} will be an instance of the matrix-valued function $H_{\bl{X}} : \mathcal{D}_+^{n} \times \mathcal{D}_+^{n} \rightarrow \mathcal{S}_+^p$ defined as
\begin{equation}\label{eq:hform}
    H_{\bl{X}}(\tilde{\bs{\Omega}}, \bs{\Omega}) = (\bl{X}^\top \tilde{\bs{\Omega}}^{-1}\bl{X})^{-1} \bl{X}^\top \tilde{\bs{\Omega}}^{-1} \bs{\Omega} \tilde{\bs{\Omega}}^{-1} \bl{X} (\bl{X}^\top \tilde{\bs{\Omega}}^{-1}\bl{X})^{-1},
\end{equation}
where $\mathcal{S}_+^p$ refers to the set of $p \times p$ positive definite matrices. The FLS estimates forming our subclass of interest therefore take the form
\begin{equation*}
    b_{\bl{X}}(\tilde{\bs{\Omega}}),~ \tilde{\bs{\Omega}} \in \mathcal{C}^p_{\bs{\Omega}},
\end{equation*}
where
\begin{equation}\label{eq:cpdef}
    \mathcal{C}^p_{\bs{\Omega}} = \{ \tilde{\bs{\Omega}} \in \mathcal{D}_+^n : | H_{\bl{X}}(\bs{\Omega}, \bs{\Omega})| \leq | H_{\bl{X}}(\tilde{\bs{\Omega}}, \bs{\Omega}) | \leq | H_{\bl{X}}(\bl{I}_n, \bs{\Omega}) |,~\forall \bl{X} \in \mathbb{R}^{n \times p} \}.
\end{equation}
We call $\mathcal{C}^p_{\bs{\Omega}}$ a \textit{subscedastic set} because its elements are---in a sense made precise in Section \ref{sec:properties}---less dispersed than $\bs{\Omega}$. That $\mathcal{C}^p_{\bs{\Omega}}$ is non-empty is guaranteed by a matrix Cauchy inequality \citep{marshall_matrix_1990}, which says
\begin{equation*}
    \Cov{b_{\bl{X}}(\bl{\Omega})} = H_{\bl{X}}(\bs{\Omega}, \bs{\Omega}) \preceq H_{\bl{X}}(\bl{I}_n, \bs{\Omega}) = \Cov{b_{\bl{X}}(\bl{I}_n)},~\forall \bl{X} \in \mathbb{R}^{n \times p},
\end{equation*}
where $\preceq$ denotes the Loewner partial order on $\mathcal{S}_+^p$. Hence, both $\bs{\Omega}$ and $\bl{I}_n$ are in $\mathcal{C}^p_{\bs{\Omega}}$. In fact, a consequence of the Gauss-Markov Theorem \citep{aitken_ivleast_1936} is that $H_{\bl{X}}(\bs{\Omega}, \bs{\Omega}) \preceq H_{\bl{X}}(\tilde{\bs{\Omega}}, \bs{\Omega}),~\forall \bl{X} \in \mathbb{R}^{n \times p},~\forall \tilde{\bs{\Omega}} \in \mathcal{D}_n^+$, so that $\mathcal{C}_{\bs{\Omega}}^p$ may be equivalently defined using only the second inequality in \eqref{eq:cpdef}. 

In Section \ref{sec:properties} of this article, we examine the subscedastic set $\mathcal{C}^p_{\bs{\Omega}}$ in the case of a single regressor and provide a sufficient condition so that $\tilde{\bs{\Omega}} \in \mathcal{C}^1_{\bs{\Omega}}$. Building on this result, we then develop the necessary and sufficient condition so that $\tilde{\bs{\Omega}} \in \mathcal{C}^p_{\bs{\Omega}}$ for $1 \leq p < n$. In Section \ref{sec:impl} we discuss the implications of these results for estimation in the linear model. In particular, we show directly that a FLS estimate need not use consistent weights in order to outperform the OLS estimate. In Section \ref{sec:robust}, we provide a link between the results of Section 2 and the asymptotic variance of two robust regression estimates, one of which is derived from the $t$-distribution. We then conduct numerical experiments in Section \ref{sec:num} that demonstrate how the $t$ estimate behaves in the context of the normal linear model with heteroscedasticity. We see that it performs favorably relative to parametric FLS estimates, especially when the parametric form of heteroscedasticity is misspecified. Finally, in Section \ref{sec:disc}, we conclude with a discussion of possible extensions to this work. The proofs of all results may be found in Section \ref{asec:proofs} of the Appendix.

\section{Properties of subscedastic weights}\label{sec:properties}

The subscedastic set $\mathcal{C}^p_{\bs{\Omega}}$ is defined using the determinant inequality
\begin{equation}\label{eq:detineq}
|H_{\bl{X}}(\tilde{\bs{\Omega}}, \bs{\Omega})| \leq |H_{\bl{X}}(\bl{I}_n, \bs{\Omega})|.
\end{equation}
Because the determinant is invariant to multiplication of its matrix argument by any $p \times p$ orthogonal matrix, the inequality \eqref{eq:detineq} is unchanged when $\bl{U}$, the matrix whose columns are the left singular vectors of $\bl{X}$, is substituted for $\bl{X}$ itself. In this section, it will be convenient to work in terms of $\bl{U}$ rather than $\bl{X}$.

To build intuition for the properties of subscedastic sets, consider the case of a single regressor so that $p=1$. Let $\mathcal{V}^n$ denote the set of $n$-dimensional unit vectors, let $\bl{u} \in \mathcal{V}^n$ and define the functions
\begin{equation*}
    \begin{aligned}
        \mathrm{e}_{\bl{u}}(\bs{\Omega}) &:= \sum_{i=1}^n u_i^2 \omega_i, \\
        \mathrm{c}_{\bl{u}}(\tilde{\bs{\Omega}}, \bs{\Omega}) &:= \sum_{i=1}^n u_i^2(\tilde{\omega}_i - \mathrm{e}_{\bl{u}}(\tilde{\bs{\Omega}}))(\omega_i - \mathrm{e}_{\bl{u}}(\bs{\Omega})).
    \end{aligned}
\end{equation*}
Because $\bl{u}$ is a unit vector, the functions $\mathrm{e}_{\bl{u}}$ and $ \mathrm{c}_{\bl{u}}$ behave, respectively, like the expectation and covariance functions of discrete random variables with supports determined by the diagonal entries of $\tilde{\bs{\Omega}}, \bs{\Omega}$ and probability mass functions determined by the squared magnitude of the entries of $\bl{u}$. By rearranging terms in the inequality $H_{\bl{u}}(\tilde{\bs{\Omega}}, \bs{\Omega}) \leq H_{\bl{u}}(\bl{I}_n, \bs{\Omega})$, we can express the condition $\tilde{\bs{\Omega}} \in \mathcal{C}^1_{\bs{\Omega}}$ in terms of the functions $\mathrm{e}_{\bl{u}}$ and $\mathrm{c}_{\bl{u}}$ as follows:
\begin{equation}\label{eq:rearr}
        \tilde{\bs{\Omega}} \in \mathcal{C}^1_{\bs{\Omega}} \iff \mathrm{c}_{\bl{u}}(\tilde{\bs{\Phi}}, \tilde{\bs{\Phi}}\bs{\Omega}) + \mathrm{e}_{\bl{u}}(\tilde{\bs{\Phi}}) \mathrm{c}_{\bl{u}}(\tilde{\bs{\Phi}}, \bs{\Omega}) \leq 0,~\forall \bl{u} \in \mathcal{V}^n,
\end{equation}
where $\tilde{\bs{\Phi}} = \tilde{\bs{\Omega}}^{-1}$. This formulation is useful because it points to an intuitive sufficient condition so that $\tilde{\bs{\Omega}} \in \mathcal{C}^1_{\bs{\Omega}}$, namely that both $\mathrm{c}_{\bl{u}}(\tilde{\bs{\Phi}}, \tilde{\bs{\Phi}}\bs{\Omega})$ and $\mathrm{c}_{\bl{u}}(\tilde{\bs{\Phi}}, \bs{\Omega})$ are non-positive for all $\bl{u} \in \mathcal{V}^n$. Now, for any $\bl{u} \in \mathcal{V}^n$, $\mathrm{c}_{\bl{u}}(\tilde{\bs{\Phi}}, \bs{\Omega})$ will be non-positive if the diagonal entries of $\tilde{\bs{\Omega}}$ are in monotone non-decreasing relation with the diagonal entries of $\bs{\Omega}$ (\citealp{schmidt_inequalities_2014}, Corollary 3.1). The condition in the following proposition is sufficient so that $\mathrm{c}_{\bl{u}}(\tilde{\bs{\Phi}}, \tilde{\bs{\Phi}}\bs{\Omega})$ is also non-positive for any choice of $\bl{u}$, implying that $\tilde{\bs{\Omega}} \in \mathcal{C}^1_{\bs{\Omega}}$.
\begin{proposition}\label{prop:mono}
    Let $g : \mathbb{R}_+ \rightarrow \mathbb{R}_{+}$ be a function such that for each $\omega > 0$
    \begin{equation}\label{eq:grm}\tag{GRM}
        \begin{aligned} 
            1. && &g(\omega) \leq g(\omega'),~\forall \omega' \geq \omega \\
            2. && &g(\omega') / \omega' \leq  g(\omega) / \omega,~\forall \omega' \geq \omega.
        \end{aligned}
    \end{equation}
     Then if $\tilde{\omega}_i = g(\omega_i)$ for $i \in \{1, \dots, n\}$, $\tilde{\bs{\Omega}} \in \mathcal{C}^1_{\bs{\Omega}}$.
\end{proposition}
In other words, for any $\bl{u} \in \mathcal{V}^n$, it is guaranteed that $H_{\bl{u}}(\tilde{\bs{\Omega}}, \bs{\Omega}) \leq H_{\bl{u}}(\bl{I}_n, \bs{\Omega})$ if the ranks of the diagonal entries of $\tilde{\bs{\Omega}}$ and the ranks of the diagonal entries of $\tilde{\bs{\Omega}}^{-1} \bs{\Omega}$ agree with those of $\bs{\Omega}$. Another way of phrasing \eqref{eq:grm} is to say that the function $g(\omega)$ is non-decreasing, while the function $g(\omega)/\omega$ is non-increasing. Functions that satisfy \eqref{eq:grm} include, but are not limited to, fractional powers ($g(\omega) = \omega^{1/q},~q \geq 1$), translations by a positive constant ($g(\omega) = \omega + \lambda,~\lambda \in \mathbb{R}_+$), and functions of the form $g(\omega) = \log(\omega + \lambda),~\lambda > 1$. Some examples are displayed in Figure \ref{fig:exfns}. 

{Functions of this type appear in the literature on robust covariance estimation \citep{maronna_robust_1976, romanov_tylers_2023}. By \cite{rosenbaum_sub-additive_1950} Theorem 1.4.3, they constitute a subset of the class of \textit{subadditive} functions on $\mathbb{R}_+$. While they may be unbounded from above, a simple transformation can produce bounded versions of such functions.
\begin{proposition}\label{prop:bd}
    If $g : \mathbb{R}_+ \rightarrow \mathbb{R}_+$ is a function satisfying \eqref{eq:grm}, then for non-negative constants $\lambda, \gamma$ the function
    \begin{equation}\label{eq:boundedfn}
        f(\omega) = \frac{1}{1/g(\omega) + \lambda} + \gamma
    \end{equation}
    also satisfies it, where $f$ has the additional property of being bounded from above and below by $1/\lambda + \gamma$ and $\gamma$, respectively.
\end{proposition}
Functions satisfying \eqref{eq:grm} also have a connection with the theory of \textit{majorization}, which has many applications to matrix trace and determinant inequalities. Letting $\bl{a} \prec \bl{b}$ denote that $\bl{b} \in \mathbb{R}^n$ majorizes $\bl{a} \in \mathbb{R}^n$ (see \citealp{marshall_inequalities_2011} for a definition), we have the following result:
\begin{proposition}\label{prop:majorization}
    Let $g$ be a function satisfying \eqref{eq:grm}. Then
    \begin{equation*}
        \left(\frac{g(\omega_1)}{\sum_{i=1}^n g(\omega_i)}, \dots, \frac{g(\omega_n)}{\sum_{i=1}^n g(\omega_i)} \right) \prec \left( \frac{\omega_1}{\sum_{i=1}^n \omega_i}, \dots, \frac{\omega_n}{\sum_{i=1}^n \omega_i}\right).
    \end{equation*}
\end{proposition}}

\begin{figure}
    \centering
    \includegraphics[scale=0.725]{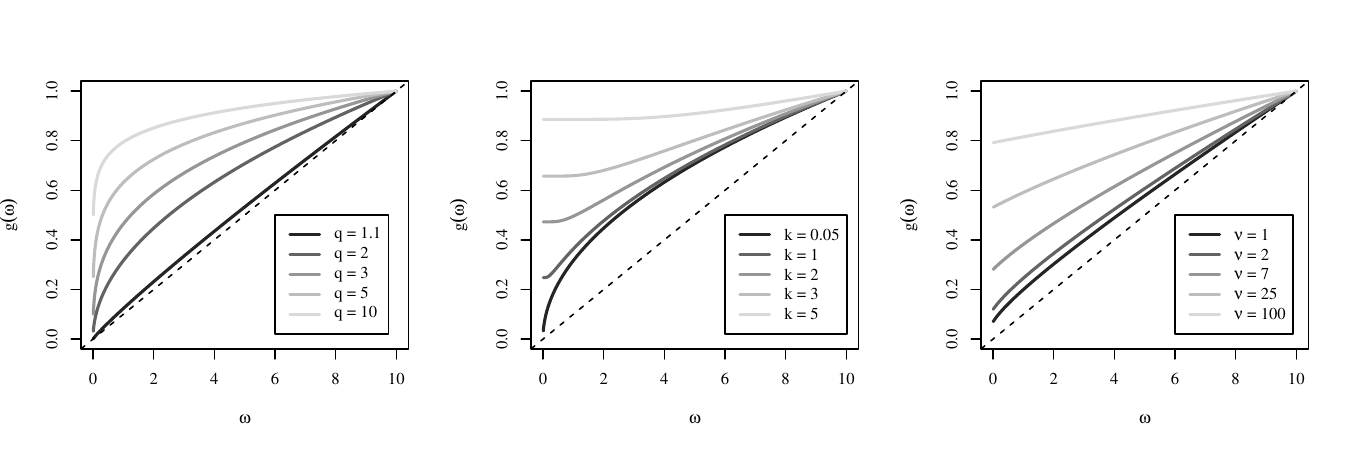}
    \caption{Three functions satisfying \eqref{eq:grm}. From left to right, the functions are $\omega^{1/q}$, $\sqrt{w}/\{\int_{-k}^k \mathrm{exp}(-z^2 / 2\omega) dz \}$, and $g_{1, \nu}(\omega)$, where this last function is defined in Theorem \ref{thm:asymvar}. For the purposes of visualization, the functions have been normalized to attain a maximum value of 1 at $\omega = 10$.}
    \label{fig:exfns}
\end{figure}

While \eqref{eq:grm} is an interesting property, it can only provide a partial characterization of $\mathcal{C}^1_{\bs{\Omega}}$. This is because $\mathrm{c}_{\bl{u}}(\tilde{\bs{\Phi}}, \tilde{\bs{\Phi}}\bs{\Omega}) \leq 0,~\forall \bl{u} \in \mathcal{V}^n$ is not necessary for the right hand side of \eqref{eq:rearr} to hold. On the other hand, the inequality
\begin{equation*}
\begin{aligned}
    \mathrm{c}_{\bl{u}}(\tilde{\bs{\Phi}}, \tilde{\bs{\Phi}}\bs{\Omega}) - \mathrm{e}_{\bl{u}}(\tilde{\bs{\Phi}}) \mathrm{c}_{\bl{u}}(\tilde{\bs{\Phi}}, \bs{\Omega}) &= \mathrm{e}_{\bl{u}}(\tilde{\bs{\Phi}}^2 \bs{\Omega}) - \mathrm{e}_{\bl{u}}(\tilde{\bs{\Phi}})^2 \mathrm{e}_{\bl{u}}(\bs{\Omega}) \\
    &= \sum_{i=1}^n u_i^2 \omega_i (\tilde{\phi}_i - \mathrm{e}_{\bl{u}}(\tilde{\bs{\Phi}}))^2 \\
    &\geq 0
\end{aligned}
\end{equation*}
demonstrates that $\mathrm{c}_{\bl{u}}(\tilde{\bs{\Phi}}, \bs{\Omega}) \leq 0,~\forall \bl{u} \in \mathcal{V}^n$ is necessary to ensure $\tilde{\bs{\Omega}} \in \mathcal{C}^1_{\bs{\Omega}}$. Hence, a necessary and sufficient condition for $\tilde{\bs{\Omega}} \in \mathcal{C}^1_{\bs{\Omega}}$ must include a monotonicity requirement on the diagonal entries of $\tilde{\bs{\Omega}}$ along with a weaker growth restriction than that of \eqref{eq:grm}. The following theorem shows that this relaxed restriction can be expressed in terms of pairs of diagonal entries of $\tilde{\bs{\Omega}}$ and $\bs{\Omega}$.
\begin{theorem}\label{thm:main}
Let $\bs{\Omega} = \mathrm{diag}(\omega_1, \dots, \omega_n) \in \mathcal{D}_+^n$ and let $\tilde{\bs{\Omega}} = \mathrm{diag}(\tilde{\omega}_1, \dots, \tilde{\omega}_n) \in \mathcal{D}_+^n$. Then $\tilde{\bs{\Omega}} \in \mathcal{C}^1_{\bs{\Omega}}$ if and only if
\begin{equation}\label{eq:main}
    1 \leq \frac{\tilde{\omega}_i}{\tilde{\omega}_j} \leq 2\frac{\omega_i}{\omega_j} - 1
\end{equation}
for all $\omega_i \geq \omega_j$, $i \in \{1, \dots, n\}, j \in \{1, \dots, n\}$.
\end{theorem}
The monotonicity requirement on the diagonal elements of $\tilde{\bs{\Omega}}$ is expressed through the first inequality in \eqref{eq:main}, while the weak growth restriction is expressed through the second.

{As it is stated, Theorem \ref{thm:main} only refers to the variance of a single regression coefficient. The next results show that \eqref{eq:main} is not unique to the univariate case, nor is it merely a peculiarity of the generalized variance criterion used to define $\mathcal{C}_{\bs{\Omega}}^p$.
\begin{corollary}\label{cor:main1}
Let $1 \leq p < n$. Then $\tilde{\bs{\Omega}} \in \mathcal{C}^p_{\bs{\Omega}}$ if and only if \eqref{eq:main} holds for all $\omega_i \geq \omega_j$, $i \in \{1, \dots, n\}, j \in \{1, \dots, n\}$.
\end{corollary}
\begin{corollary}\label{cor:main2}
    Let $1 \leq p < n$. Define the set
    \begin{equation*}
        \mathcal{K}^p_{\bs{\Omega}} = \{ \tilde{\bs{\Omega}} \in \mathcal{D}_+^n : \tr [H_{\bl{X}}(\tilde{\bs{\Omega}}, \bs{\Omega})] \leq \tr [H_{\bl{X}}(\bl{I}_n, \bs{\Omega}) ],~\forall \bl{X} \in \mathbb{R}^{n \times p} \}.
    \end{equation*}
    Then $\tilde{\bs{\Omega}} \in \mathcal{K}_{\bs{\Omega}}^p \iff \tilde{\bs{\Omega}} \in \mathcal{C}_{\bs{\Omega}}^p$.
\end{corollary}
Corollary \ref{cor:main1} says that the subscedastic set induced by a given $\bs{\Omega}$ does not depend on the number of regressors, provided that number is less than $n$. Corollary \ref{cor:main2} shows that the total variance of a FLS estimate using subscedastic weights is less than that of the OLS estimate.} {While it is not the focus of this article, the next result shows that the case of non-diagonal error covariance can also be partially addressed by \eqref{eq:main}.
\begin{corollary}\label{cor:main3}
    Let $\bs{\Omega}, \tilde{\bs{\Omega}} \in \mathcal{S}_+^n$ be simultaneously diagonalizable, and let $\{\lambda_1, \dots, \lambda_n\}$ and $\{\tilde{\lambda}_1, \dots, \tilde{\lambda}_n\}$ be the eigenvalues of $\bs{\Omega}$ and $\tilde{\bs{\Omega}}$, respectively. Then $\tilde{\bs{\Omega}} \in \mathcal{C}^p_{\bs{\Omega}}$ if and only if
    \begin{equation}
        1 \leq \frac{\tilde{\lambda}_i}{\tilde{\lambda}_j} \leq 2\frac{\lambda_i}{\lambda_j} - 1
    \end{equation}
    for all $\lambda_i \geq \lambda_j$, $i \in \{1, \dots, n\}, j \in \{1, \dots, n\}$. 
\end{corollary}}

We conclude this section by stating some additional properties of subscedastic weights that can be derived from \eqref{eq:main}. First, Corollary \ref{cor:main1} implies that the conclusion of Proposition \ref{prop:mono} also holds for $\mathcal{C}_{\bs{\Omega}}^p$ when $p>1$. This can be seen by the following reasoning: if the diagonal elements of $\tilde{\bs{\Omega}}$ satisfy \eqref{eq:grm}, then $\tilde{\omega}_i / \tilde{\omega}_j \geq 1$ for all $\omega_i \geq \omega_j$, and $\tilde{\omega}_i \omega_j / (\tilde{\omega}_j \omega_i) \leq 1 \leq 2-(\omega_j / \omega_i)$ for all $\omega_i \geq \omega_j$. Multiplying the latter inequalities by $\omega_i /\ \omega_j$ shows that \eqref{eq:grm} implies \eqref{eq:main}. Next, because \eqref{eq:main} depends only on pairwise ratios of diagonal elements, $\mathcal{C}^p_{\bs{\Omega}}$ has what \cite{bilodeau_choice_1990} and \cite{kariya_generalized_2004} call the ``symmetric inverse property," meaning
\begin{equation*}
	\tilde{\bs{\Omega}} \in \mathcal{C}^p_{\bs{\Omega}} \iff \tilde{\bs{\Omega}}^{-1} \in \mathcal{C}^p_{\bs{\Omega}^{-1}}.
\end{equation*}
Finally, $\mathcal{C}_{\bs{\Omega}}^p$ is a convex cone on $\mathcal{D}_+^n$. The cone property of $\mathcal{C}^p_{\bs{\Omega}}$ is clear from the fact that $H_{\bl{X}}(s\tilde{\bs{\Omega}}, \bs{\Omega}) = H_{\bl{X}}(\tilde{\bs{\Omega}}, \bs{\Omega}),~\forall s>0$. The convexity of $\mathcal{C}^p_{\bs{\Omega}}$ can be derived directly from \eqref{eq:main}: given $\tilde{\bs{\Omega}}, \tilde{\bs{\Psi}} \in \mathcal{C}^p_{\bs{\Omega}}$,
\begin{equation*}
    \tilde{\omega}_j \leq \tilde{\omega}_i \leq 2\frac{\omega_i}{\omega_j} \tilde{\omega}_j - \tilde{\omega}_j,~\text{ and }~\tilde{\psi}_j \leq \tilde{\psi}_i \leq 2\frac{\omega_i}{\omega_j} \tilde{\psi}_j - \tilde{\psi}_j,
\end{equation*}
for all $\omega_i \geq \omega_j$. For any $0 \leq t \leq 1$, this implies
\begin{equation*}
    t \tilde{\omega}_j + (1-t)\tilde{\psi}_j \leq t \tilde{\omega}_i + (1-t)\tilde{\psi}_i \leq 2\frac{\omega_i}{\omega_j} \{t \tilde{\omega}_j + (1-t)\tilde{\psi}_j\} - \{t \tilde{\omega}_j + (1-t)\tilde{\psi}_j\},
\end{equation*}
for all $\omega_i \geq \omega_j$, so $t \tilde{\bs{\Omega}} + (1-t) \tilde{\bs{\Psi}} \in \mathcal{C}^p_{\bs{\Omega}}$. Along with the cone property, convexity implies that 
\begin{equation}\label{eq:reg}
\tilde{\bs{\Omega}} \in \mathcal{C}^p_{\bs{\Omega}} \implies \tilde{\bs{\Omega}} + s\bl{I}_n \in \mathcal{C}^p_{\bs{\Omega}},~\forall s>0,
\end{equation}
so regularized subscedastic weights are also subscedastic weights.

\section{Estimation in the linear model}\label{sec:impl}

The matrix $H_{\bl{X}}(\tilde{\bs{\Omega}}, \bs{\Omega})$ used to define the notion of a subscedastic set is equal to the covariance matrix of the FLS estimate $b_{\bl{X}}(\tilde{\bs{\Omega}})$ under \eqref{eq:smod} when $\tilde{\bs{\Omega}}$ is any fixed matrix in $\mathcal{D}_n^+$. As seen in the previous section, the conditions so that $\tilde{\bs{\Omega}}$ is a member of $\mathcal{C}^p_{\bs{\Omega}}$ depend on the unknown $\bs{\Omega}$. A natural question is then: to what extent is $b_{\bl{X}}(\tilde{\bs{\Omega}})$ actually feasible? More broadly, what is the relevance of Section \ref{sec:properties} to estimation in practice if one must know $\bs{\Omega}$ to choose an appropriate $\tilde{\bs{\Omega}}$?

To begin, we observe that one does not need to know the values of the diagonal elements of $\bs{\Omega}$. One implication of Corollary \ref{cor:main1} is that knowing the ranks of $(\omega_1, \dots, \omega_n)$ along with a lower bound on the minimum ratio between consecutive ordered elements $\gamma < \min_{i = j+1}\{\omega_{(i)}/\omega_{(j)}\}$ would be sufficient to construct an $\tilde{\bs{\Omega}}$ and a corresponding $b_{\bl{X}}(\tilde{\bs{\Omega}})$ that is guaranteed to outperform the OLS estimate. By contrast, to reproduce $\bs{\Omega}$ itself up to a scale factor, it would be necessary to know the ranks of the diagonal elements of $\bs{\Omega}$ along with $n-1$ additional ratios, for instance the collection of all ratios between consecutive ordered elements. Thus, the task of finding an optimal estimate (the WLS estimate) depends on more unknowns than the task of finding an estimate that is at least better than the OLS estimate.

This latter, more modest goal, brings otherwise impossible tasks into the feasible realm in certain simple cases. For instance, consider the groupwise heteroscedastic linear model with error covariance matrix
\begin{equation*}
    \bs{\Omega} = \bs{\Omega}(\bs{\theta}) = \left[ \begin{array}{ccc}
         \theta_1^2 \bl{I}_{n_1} & & \bl 0 \\
         & \ddots & \\
         \bl 0 & & \theta_K^2 \bl{I}_{n_K}
    \end{array}\right],
\end{equation*}
where $\bs{\theta} \in \mathbb{R}_+^K$, $n = \sum_{k=1}^K n_k$. While it is implausible that one knows the exact values of $\bs{\theta}$ in advance, it is at least more plausible that, for small $K$, one knows the ordering of the elements of $\bs{\theta}$ and that no group's error variance is within some factor $\gamma \geq 1$ of another's. Let $\tilde{\theta}_{(1)}^2 = 1$, and set $\tilde{\theta}_{(i)}^2 := (2\gamma - 1)\tilde{\theta}_{(i-1)}^2$ for each $i > 1$. Then, using the notation above, $\tilde{\bs{\Omega}} = \bs{\Omega}(\tilde{\bs{\theta}})$ defines a matrix whose diagonal elements are subscedastic weights.

Alternatively, consider the linear model with error variances depending on a single covariate through a parameterized scedastic function
\begin{equation*}
    \omega_i = v_{\theta}(x_{i,1}),~i=1, \dots, n.
\end{equation*}
Common examples of $v_{\theta}$, all of which are used in the simulation studies of \cite{romano_resurrecting_2017}, include
\begin{equation}\label{eq:scedasticfns}
\begin{aligned}
    v_{{\theta}}(x) &= |x|^{\theta} \\
    v_{{\theta}}(x) &= \{\log|x|\}^{\theta} \\
    v_{{\theta}}(x) &= e^{\theta|x| + \theta|x|^2}.
\end{aligned}
\end{equation}
Conveniently, in each specification above, the ranks of $(\omega_1, \dots, \omega_n)$ are equivalent to the ranks of $(|x_{1,1}|, \dots, |x_{n,1}|)$, which are known. Hence, if one can identify a lower bound $\gamma$ for the minimum plausible value of $\theta$, one can simply take $\tilde{\omega}_i = v_{\gamma}(x_{i,1})$, and the corresponding $b_{\bl{X}}(\tilde{\bs{\Omega}})$ will outperform the OLS estimate. This is due to the fact that, for each of the scedastic functions above, $\gamma \leq \theta$ implies that $v_{\gamma}(x)$ is a fractional power of $v_{\theta}(x)$ and thus satisfies \eqref{eq:grm}.

{Finally, there are simple yet common examples of linear models with non-diagonal covariance for which the subscedastic property can serve as a guide for the design of feasible weights. Suppose that data are recorded on biological specimens, which are processed in $K$ distinct batches. Sources of idiosyncratic variation due to the processing of the specimens may introduce marginal correlation among observations within batches. This may be modeled as follows
\begin{equation}\label{eq:mixmod}
\begin{aligned}
    \bl{y} &= \bl{X}\bs{\beta} + \bl{A}\bl{z} + \bs{\varepsilon}, \\
    \Exp{\bl{z}} &=  \bl{0}, \Cov{\bl{z}} = \theta_1^2 \bl{I}_K \\
    \Exp{\bs{\varepsilon}} &=  \bl{0}, \Cov{\bs{\varepsilon}} = \theta_0^2 \bl{I}_n,
\end{aligned}
\end{equation}
where $\bl{A}$ is an $n \times K$ indicator matrix such that $A_{ik} = 1$ if observation $i$ is in batch $k$, and $A_{ik} = 0$ otherwise. Model \eqref{eq:mixmod} specifies a linear mixed effects model with random intercepts, which induces a non-diagonal marginal error covariance matrix
\begin{equation*}
    \bs{\Omega}(\bs{\theta}) = \theta_1^2 \bl{A}\bl{A}^\top + \theta_0^2 \bl{I}_n.
\end{equation*}
Conveniently, both the eigenvectors of this matrix and the ordering of its eigenvalues are known as long as $\bl{A}$ is known. Assume that the number of observations in batch $k$ is $n_k$, and that $n_1 \geq \dots \geq n_K$. Then the eigenvectors and eigenvalues of $\bs{\Omega}(\bs{\theta})$ are, respectively,
\begin{equation*}
\begin{aligned}
    \bl{U} &= \begin{bmatrix}\bl{A} &\bl{Q}\end{bmatrix}\mathrm{diag}(1/\sqrt{n_1}, \dots, 1/\sqrt{n_K}, 1, \dots, 1), \\
    \bs{\Lambda} &= \mathrm{diag}(n_1 \theta_1^2 + \theta_0^2, \dots, n_K \theta_1^2 + \theta_0^2, \theta_0^2, \dots, \theta_0^2),
\end{aligned}
\end{equation*}
where $\bl{Q} \in \mathcal{V}^{n, n-K}$ is a matrix whose columns form an orthonormal basis for the null space of $\bl{A}^\top$. Here, subscedastic weights may be designed if one has knowledge of, or a conservative upper bound on $\theta_0^2/\theta_1^2$, the ratio of the isotropic variability to the batch variability. Specifically, if $\gamma \geq \theta_0^2/\theta_1^2$, then it can be checked that
\begin{equation*}
\begin{aligned}
    \tilde{\bs{\Omega}} &= \bl{U} \tilde{\bs{\Lambda}} \bl{U}^\top, \\
    \tilde{\bs{\Lambda}} &= \mathrm{diag}(n_1 +\gamma, \dots, n_K +\gamma, \gamma, \dots, \gamma),
\end{aligned}
\end{equation*}
with $\bl{U}$ defined as above satisfies \eqref{eq:main} with respect to $\bs{\Omega}(\bs{\theta})$. Of course this same reasoning applies to the case of a single ($K=1$) ``batch," which for small $\theta_1^2$ would correspond to weak equi-correlation between all errors.}

While the conclusions above have the appeal of being valid for any sample size, they do not yield much insight into the behavior of FLS estimates using random weights. Still, the subscedastic property may be used to assess such estimates when $n$ is large. {One conclusion is that a matrix of random weights $\hat{\bs{\Omega}}_n$ need not be consistent for $\bs{\Omega}$ in order to produce an estimate that eventually outperforms the OLS estimate.} Informally, if $\hat{\bs{\Omega}}_n$ is asymptotically equal to some $\tilde{\bs{\Omega}} \in \mathcal{C}^p_{\bl{\Omega}}$, then $b_{\bl{X}}(\hat{\bs{\Omega}}_n)$ outperforms the OLS estimate if $n$ is large enough. Both \cite{atkinson_robust_2016} and \cite{romano_resurrecting_2017} provide numerical evidence for this claim by evaluating the variance of FLS estimates when they are misspecified with respect to the true form of heteroscedasticity. Here, we give sufficient conditions on the probability limit of feasible weights $\hat{\bs{\Omega}}_n$ so that they yield an estimate that outperforms the OLS estimate as $n \rightarrow \infty$. Since the dimension of $\bs{\Omega}$ grows with $n$, the statement of these conditions requires a modified notation that replaces matrices with infinite sequences. 
\begin{proposition}\label{prop:consistency}
{Let $\{\tilde{\omega}_i\}_{i=1}^{\infty}$, $\{\omega_i\}_{i=1}^{\infty}$ be sequences of positive scalars, let $\{\bl{x}_i\}_{i=1}^{\infty}$ be a sequence of $p$-dimensional vectors, and let $\{y_i\}_{i=1}^\infty$, $\{\hat{\omega}_i\}_{i=1}^{\infty}$ be sequences of random variables. Assuming that
\begin{equation*}
    \lim_{n \rightarrow \infty}\frac{1}{n}\sum_{i=1}^{n} \frac{1}{\tilde{\omega}_i}\bl{x}_i \bl{x}_i^\top~\text{ and }~\lim_{n \rightarrow \infty}\frac{1}{n}\sum_{i=1}^{n} \frac{\omega_i}{\tilde{\omega}^2_i}\bl{x}_i \bl{x}_i^\top
\end{equation*}}
exist, define the coefficient estimates
\begin{equation*}
    b^{n}_{\bl{X}}(\{\tilde{\omega}_i\}_{i=1}^{\infty}) = \left\{\frac{1}{n}\sum_{i=1}^{n} \frac{1}{\tilde{\omega}_i}\bl{x}_i \bl{x}_i^\top \right\}^{-1} \left\{\frac{1}{n}\sum_{i=1}^{n} \frac{y_i}{\tilde{\omega}_i}\bl{x}_i \right\}
\end{equation*}
and
\begin{equation*}
    b^{n}_{\bl{X}}(\{\hat{\omega}_i\}_{i=1}^\infty) = \left\{\frac{1}{n}\sum_{i=1}^{n} \frac{1}{\hat{\omega}_i}\bl{x}_i \bl{x}_i^\top \right\}^{-1} \left\{\frac{1}{n}\sum_{i=1}^{n} \frac{y_i}{\hat{\omega}_i}\bl{x}_i \right\},
\end{equation*}
and suppose that $\{\hat{\omega}_i\}_{i=1}^\infty$ satisfies
\begin{equation*}
\begin{aligned}
    \frac{1}{n} \sum_{i=1}^n (1/\hat{\omega}_i - 1/\tilde{\omega}_i)\bl{x}_i \bl{x}_i^\top &\overset{p}{\longrightarrow} \bl{0}, \\
    \frac{1}{\sqrt{n}} \sum_{i=1}^n (1/\hat{\omega}_i - 1/\tilde{\omega}_i)(y_i - \bl{x}_i^\top\bs{\beta}) \bl{x}_i &\overset{p}{\longrightarrow} \bl{0}.
\end{aligned}
\end{equation*}
Then if $\{\tilde{\omega}_i\}_{i=1}^{\infty}$, $\{\omega_i\}_{i=1}^{\infty}$ are such that
\begin{equation*}
    \mathrm{diag}(\tilde{\omega}_1, \dots, \tilde{\omega}_n) \in \mathcal{C}^p_{\mathrm{diag}({\omega}_1, \dots, \omega_n)}
\end{equation*}
for any positive integer $n \geq 2$, it follows that
\begin{equation*}
    \frac{|\Cov{ b^{n}_{\bl{X}}(\{\hat{\omega}_i\}_{i=1}^{\infty})}|}{|\Cov{ b^n_{\bl{X}}(\{1\})}|} \overset{p}{\longrightarrow} \lim_{n \rightarrow \infty}\frac{|\Cov{ b^{n}_{\bl{X}}(\{\tilde{\omega}_i\}_{i=1}^{\infty})}|}{|\Cov{ b^n_{\bl{X}}(\{1\})}|} \leq 1,
\end{equation*}
where $\{1\}$ denotes the constant sequence.
\end{proposition}

Proposition \ref{prop:consistency} says that a feasible estimate for the error variances need not have the same parametric form as the ground truth $\bs{\Omega} = \bs{\Omega}(\bs{\theta})$ in order to yield coefficient estimates that eventually outperform the OLS estimate. Here again, the benefits of moderating one's goals in estimation become apparent. Any feasible estimate satisfying the consistency properties of Proposition \ref{prop:consistency} will be asymptotically optimal for exactly one sequence of error variances. On the other hand, the same estimate will outperform OLS for a whole family of such sequences.

This observation motivates a general prescription for designing FLS estimates that are conservative with respect to misspecification of the scedastic function. Namely, one can add a small multiple of the identity to the matrix of feasible weights. For simplicity, consider the finite sample case where $\tilde{\bs{\Omega}} \in \mathcal{D}_+^n$ is non-random. Then for any $\bs{\Omega} \in \mathcal{D}_+^n$,
\begin{equation*}
    \tilde{\bs{\Omega}} \in \mathcal{C}^p_{\bs{\Omega}} \implies \tilde{\bs{\Omega}} + s\bl{I}_n \in \mathcal{C}^p_{\bs{\Omega}},~\forall s>0.
\end{equation*}
This is due to the fact that the identity is the unique matrix that is in $\mathcal{C}^p_{\bs{\Omega}}$ for any choice of $\bs{\Omega} \in \mathcal{D}_+^n$, and the fact that $\mathcal{C}^p_{\bs{\Omega}}$ is a convex cone.
If one defines the set
\begin{equation*}
        \mathcal{W}^p_{\tilde{\bs{\Omega}}} = \{\bs{\Omega} \in \mathcal{D}_n^+ | \tilde{\bs{\Omega}} \in \mathcal{C}^p_{\bs{\Omega}}\},
\end{equation*}
then it follows that
\begin{equation*}
    \mathcal{W}^p_{\tilde{\bs{\Omega}}} \subseteq \mathcal{W}^p_{\tilde{\bs{\Omega}} + s\bl{I}_n},
\end{equation*}
with the inclusion above being strict as long as $\tilde{\bs{\Omega}}$ is not proportional to the identity. The subset of model \eqref{eq:smod} under which $b_{\bl{X}}(\tilde{\bs{\Omega}} + s\bl{I}_n)$ outperforms $b_{\bl{X}}(\bl{I}_n)$ is evidently larger than that under which $b_{\bl{X}}(\tilde{\bs{\Omega}})$ outperforms $b_{\bl{X}}(\bl{I}_n)$. This idea can be combined with the ideas of Proposition \ref{prop:consistency} to obtain a similar statement for FLS estimates with random weights that converge in probability to some fixed $\tilde{\bs{\Omega}}$.

\section{Robust estimates of regression coefficients}\label{sec:robust}

Another informal prescription for dealing with heteroscedasticity when the scedastic function is unknown is to use a robust regression estimate rather than the OLS estimate. Here we demonstrate that subscedastic sets provide an explanation for some of the favorable properties of two robust estimates in the context of the heteroscedastic linear model with normally distributed errors.

First, we consider the maximum marginal likelihood estimate of $\bs{\beta}$ under the hierarchical linear model
\begin{equation}\label{eq:tmod}
    \begin{aligned}
    \bl{y} | \bs{\beta}, \bs{\Omega} &\sim N_n(\bl{X} \bs{\beta}, \bs{\Omega}) \\
    \omega_1, \dots, \omega_n &\overset{iid}{\sim} IG(\nu/2, \nu \omega_0 / 2),
    \end{aligned} 
\end{equation}
where $IG$ denotes the inverse gamma distribution. Marginalizing over the $\omega_i$'s, the $y_i$'s in this model are independent realizations of $t$-distributed random variables, each with mean $\bl{x}_i^\top \bs{\beta}$, scale $\omega_0$, and degrees of freedom $\nu$. The independent $t$ model and its maximum likelihood estimate $\hat{\bs{\beta}}_T$ have previously been studied in the context of robust regression, in particular by \cite{lange_robust_1989} who derived several of its properties in the well-specified case. \cite{lange_normalindependent_1993} also discuss how $\hat{\bs{\beta}}_T$ may be computed using an iteratively reweighted least squares algorithm (see also Section \ref{asec:em} of the Appendix). 

Our interest lies in the asymptotic behavior of $\hat{\bs{\beta}}_T$ in the misspecified case, specifically when the true model is the heteroscedastic linear model \eqref{eq:smod} with normally distributed errors. Because $\hat{\bs{\beta}}_T$ is the solution to a maximization problem, its asymptotic distribution under this type of misspecification can be understood through the framework of $M$-estimation \citep{huber_robust_1973, huber_robust_2009}. Applying the $M$-estimation results of \cite{stefanski_calculus_2002}, we provide the asymptotic covariance matrix of $\hat{\bs{\beta}}_T$ in the following theorem.
\begin{theorem}\label{thm:asymvar}
    The asymptotic distribution of $\hat{\bs{\beta}}_{T}$ in the model \eqref{eq:smod} with normally distributed errors is described by
    \begin{equation*}
        \sqrt{n}(\hat{\bs{\beta}}_{T} - \bs{\beta}) \overset{d}{\longrightarrow} N_p(\bl{0}, \bl{V}^{-1} \bl{B} \bl{V}^{-1})
    \end{equation*}
    where 
    \begin{equation*}
    \begin{aligned}
    	\bl{B} = \lim_{n \rightarrow \infty} \frac{1}{n}\sum_{i=1}^n \frac{1}{f_{\omega_0, \nu}(\omega_i)} \bl{x}_i \bl{x}_i^\top \\
    	\bl{V} = \lim_{n \rightarrow \infty} \frac{1}{n} \sum_{i=1}^n \frac{1}{g_{\omega_0, \nu}(\omega_i)} \bl{x}_i \bl{x}_i^\top,
    \end{aligned}
    \end{equation*}
    and
    \begin{equation*}
    \begin{aligned}
    f_{\omega_0, \nu}(\omega_i) &= 2\omega_i\left\{\left(\sqrt{\frac{\nu \omega_0}{\omega_i}} + \sqrt{\frac{\omega_i}{\nu \omega_0}} \right) e^{\frac{\nu \omega_0}{2\omega_i}} \left(\int_{-\infty}^{-\sqrt{\frac{\nu \omega_0}{\omega_i}}} e^{-z^2/2} dz\right) - 1 \right\}^{-1},\\
    g_{\omega_0, \nu}(\omega_i) &= \omega_i\left\{1 - \sqrt{\frac{\nu \omega_0}{\omega_i}} e^{\frac{\nu \omega_0}{2\omega_i}} \left( \int_{-\infty}^{-\sqrt{\frac{\nu \omega_0}{\omega_i}}} e^{-z^2/2} dz \right)\right\}^{-1}. \\
    \end{aligned}
    \end{equation*}
    Furthermore, $g_{\omega_0, \nu}$ satisfies \eqref{eq:grm} for any $\omega_0, \nu > 0$.
\end{theorem}
The last part of the result above pertaining to the \eqref{eq:grm} property allows us to obtain insight into how $\hat{\bs{\beta}}_T$ behaves relative to the OLS estimate. Specifically, if $\{\omega_i\}_{i=1}^\infty$ is a bounded sequence, then the limiting generalized variance of $\sqrt{n}\hat{\bs{\beta}}_T$ can be bounded above by a constant multiple of the limiting generalized variance of $\sqrt{n}\hat{\bs{\beta}}_{\mathrm{OLS}}$. Likewise with the limiting total variance.
\begin{corollary}\label{cor:wcbound}
    Suppose that $\sup\{\omega_i\}_{i=1}^\infty = \omega_{\max}$. Then define
    \begin{equation*}
        C(\omega_0, \nu, \omega_\mathrm{max}) = \frac{g_{\omega_0, \nu}(\omega_{\max})^2}{\omega_{\max} f_{\omega_0, \nu}(\omega_{\max})}.
    \end{equation*}
    Then
    \begin{equation}\label{eq:limbd}
        \lim_{n \rightarrow \infty} \tr[\Cov{\sqrt{n} \hat{\bs{\beta}}_T}] \leq C(\omega_0, \nu, \omega_\mathrm{max}) \lim_{n \rightarrow \infty} \tr[\Cov{\sqrt{n} \hat{\bs{\beta}}_{\mathrm{OLS}}}]
    \end{equation}
    and
    \begin{equation}\label{eq:limbd}
        \lim_{n \rightarrow \infty} |\Cov{\sqrt{n} \hat{\bs{\beta}}_T}| \leq C(\omega_0, \nu, \omega_\mathrm{max})^p \lim_{n \rightarrow \infty} |\Cov{\sqrt{n} \hat{\bs{\beta}}_{\mathrm{OLS}}}|.
    \end{equation}
\end{corollary}

Similar calculations may be done for another well-known robust regression estimate, the Huber estimate, defined as
\begin{equation*}
    \hat{\bs{\beta}}_{H} = \underset{\bs{\beta}}{\mathrm{argmin}}~~ \sum_{i=1}^n \rho_k (y_i - \bl{x}_i^\top \bs{\beta}),
\end{equation*}
where
\begin{equation*}
    \rho_k(z) = \left\{\begin{array}{lc}
         z^2/2 & ~~|z| < k \\
         k|z| - k^2 / 2 & ~~|z| \geq k
    \end{array}\right. .
\end{equation*}
The asymptotic distribution of $\hat{\bs{\beta}}_H$ with respect to \eqref{eq:smod} is easily obtained from the formulae of \cite{huber_robust_1964}. The next result describes this distribution and shows that it too contains weights satisfying the \eqref{eq:grm} property. 
\begin{proposition}\label{prop:huber}
The asymptotic distribution of $\hat{\bs{\beta}}_{H}$ in the model \eqref{eq:smod} with normally distributed errors is described by
\begin{equation*}
    \sqrt{n}(\hat{\bs{\beta}}_{H} - \bs{\beta}) \overset{d}{\longrightarrow} N_p(\bl{0}, \bl{V}^{-1} \bl{B} \bl{V}^{-1})
\end{equation*}
where
\begin{equation*}
    \begin{aligned}
        \bl{B} &= \lim_{n \rightarrow \infty} \frac{1}{n} \sum_{i=1}^n \frac{1}{f_{k}(\omega_i)} \bl{x}_i \bl{x}_i^\top \\
        \bl{V} &= \lim_{n \rightarrow \infty} \frac{1}{n} \sum_{i=1}^n \frac{1}{g_{k}(\omega_i)} \bl{x}_i \bl{x}_i^\top,
    \end{aligned}
\end{equation*}
and
\begin{equation*}
    \begin{aligned}
        f_k(\omega_i) &= \left\{\frac{1}{\sqrt{2\pi\omega_i}}\int_{-k}^k z^2 e^{-z^2/2\omega_i} dz + \frac{k^2}{\sqrt{2\pi\omega_i}} \int_{-\infty}^{-k} e^{-z^2/2\omega_i} dz + \frac{k^2}{\sqrt{2\pi\omega_i}} \int_{k}^\infty e^{-z^2/2\omega_i} dz\right\}^{-1}\\
        g_k(\omega_i) &= \left\{\frac{1}{\sqrt{2\pi\omega_i}}\int_{-k}^k e^{-z^2/2\omega_i} dz\right\}^{-1} .\\
    \end{aligned}
\end{equation*}
Furthermore, $g_{k}$ satisfies \eqref{eq:grm} for all $k > 0$.
\end{proposition}
In analogy to the result in Corollary \ref{cor:wcbound}, Proposition \ref{prop:huber} implies that the asymptotic generalized and total variance of $\sqrt{n} \hat{\bs{\beta}}_{H}$ are bounded above by a constant times those of a FLS estimate using subscedastic weights.
\begin{corollary}\label{cor:wcboundh}
    Suppose that $\sup\{\omega_i\}_{i=1}^\infty = \omega_{\max}$. Then define
    \begin{equation*}
        C(k, \omega_\mathrm{max}) = \frac{g_{k}(\omega_{\max})^2}{\omega_{\max} f_{k}(\omega_{\max})}.
    \end{equation*}
    Then
    \begin{equation}\label{eq:limbd}
        \lim_{n \rightarrow \infty} \tr[\Cov{\sqrt{n} \hat{\bs{\beta}}_H}] \leq C(k, \omega_\mathrm{max}) \lim_{n \rightarrow \infty} \tr[\Cov{\sqrt{n} \hat{\bs{\beta}}_{\mathrm{OLS}}}]
    \end{equation}
    and
    \begin{equation}\label{eq:limbd}
        \lim_{n \rightarrow \infty} |\Cov{\sqrt{n} \hat{\bs{\beta}}_H}| \leq C(k, \omega_\mathrm{max})^p \lim_{n \rightarrow \infty} |\Cov{\sqrt{n} \hat{\bs{\beta}}_{\mathrm{OLS}}}|.
    \end{equation}
\end{corollary}

\begin{table}\centering
\begin{tabular}{@{}cc|cc@{}}\toprule
\multicolumn{2}{c}{$t$ MLE} & \multicolumn{2}{c}{Huber} \\ \midrule
$\omega_{\max}/\nu \omega_0$ & $C(\omega_0, \nu, \omega_{\max})$ & $\omega_{\max}/k^2$ & $C(k, \omega_{\max})$ \\ \midrule
0.1 & 1.024994 & 0.1 & 1.000217 \\
0.2 & 1.061614 & 0.2 & 1.005255 \\
0.5 & 1.16679 & 0.5 & 1.045128\\
1 & 1.313124 & 1 & 1.107267 \\
2 & 2.028271 & 2 & 1.184329 \\
5 & 2.592936 & 5 & 1.286343 \\
\bottomrule
\end{tabular}
\caption{Values for the bounding constants appearing in Corollary \ref{cor:wcbound} and Corollary \ref{cor:wcboundh}.}
\label{tab:consts}
\end{table}

It can be shown that $C(\omega_0, \nu, \omega_\mathrm{max})$ and $C(k, \omega_{\max})$ are monotone increasing functions of $\omega_{\max} / \nu \omega_0$ and $\omega_{\max} / k^2$, respectively. Table \ref{tab:consts} gives evaluations of $C(\omega_0, \nu, \omega_\mathrm{max})$ and $C(k, \omega_{\max})$ as a function of these arguments. As an example, when $\omega_{\max}/\nu \omega_0 = 1$, the total variance of $\hat{\bs{\beta}}_T$ is at worst $1.313$ times that of the OLS estimate, so by this measure it is about $75\%$ efficient in the worst case. When $\omega_{\max} / k^2 = 1$, the total variance of $\hat{\bs{\beta}}_H$ is at worst only $1.11$ times that of $\hat{\bs{\beta}}_{\mathrm{OLS}}$, so its worst-case efficiency is about $90 \%$.

\subsection{Inference for $\hat{\bs{\beta}}_T$}\label{sec:inf}

Often, it is of interest to calculate confidence intervals for the coefficients in a regression to accompany point estimates. In what follows we propose a method for calculating asymptotically valid confidence intervals for the $t$-derived estimates examined above. Our confidence intervals are based on standard errors derived from an empirical approximation to the asymptotic covariance matrix in Theorem \ref{thm:asymvar}.

Let $\ell(y_i ; \bs{\beta})$ refer to the log-likelihood of $\bs{\beta}$ for a single observation in the independent $t$ error model \eqref{eq:tmod}. From the calculations in the proof of Theorem \ref{thm:asymvar}, we have
\begin{equation*}
    \nabla \ell(y_i ; \bs{\beta}) \nabla \ell(y_i ; \bs{\beta})^\top = (\nu+1)^2 \frac{(y_i - \bl{x}_i^\top \bs{\beta})^2}{\{\nu \omega_0 + (y_i - \bl{x}_i^\top \bs{\beta})^2\}^2 } \bl{x}_i \bl{x}_i^\top,
\end{equation*}
and
\begin{equation*}
    \nabla^2 \ell(y_i ; \bs{\beta}) =(\nu+1) \frac{(y_i - \bl{x}_i^\top \bs{\beta})^2 - \nu \omega_0}{\{\nu \omega_0 + (y_i - \bl{x}_i^\top \bs{\beta})^2\}^2} \bl{x}_i \bl{x}_i^\top.
\end{equation*}
Replacing the expectations of these quantities with their empirical counterparts evaluated at $\hat{\bs{\beta}}_T$ yields our estimate of the covariance matrix of $\hat{\bs{\beta}}_T$
\begin{equation*}
    \widehat{\Cov{\hat{\bs{\beta}}_T}} = \left\{ \sum_{i=1}^n \nabla^2 \ell(y_i ; \hat{\bs{\beta}}_T) \right\}^{-1}\left\{\sum_{i=1}^n \nabla \ell(y_i ; \hat{\bs{\beta}}_T) \nabla \ell(y_i ; \hat{\bs{\beta}}_T)^\top \right\} \left\{ \sum_{i=1}^n \nabla^2 \ell(y_i ; \hat{\bs{\beta}}_T) \right\}^{-1}.
\end{equation*}
Theorems from \cite{hoadley_asymptotic_1971}, \cite{white_nonlinear_1980}, and \cite{iverson_effects_1989} suggest that $\widehat{\Cov{\hat{\bs{\beta}}_T}}$ consistently estimates $\Cov{\hat{\bs{\beta}}_T}$, so that
\begin{equation*}
    \left[(\hat{\bs{\beta}}_T)_j - Z_{1-\alpha/2} \widehat{\Cov{\hat{\bs{\beta}}_T}}_{jj}^{1/2}, (\hat{\bs{\beta}}_T)_j + Z_{1-\alpha/2} \widehat{\Cov{\hat{\bs{\beta}}_T}}_{jj}^{1/2}\right]
\end{equation*}
defines a confidence interval for the $j^\mathrm{th}$ regression coefficient with asymptotic coverage at level $1 - \alpha$.

\section{Simulations and real-world examples}\label{sec:num}

While both $\hat{\bs{\beta}}_T$ and $\hat{\bs{\beta}}_H$ are less efficient than OLS in the worst case, numerical results presented in this section suggest that both $\hat{\bs{\beta}}_T$ and $\hat{\bs{\beta}}_H$ are more efficient than the OLS estimate for simulated and real design matrices when there is at least a mild degree of heteroscedasticity. Additional real data examples demonstrate that confidence intervals based on $\hat{\bs{\beta}}_T$ and its asymptotic standard errors can substantially outperform common-use variants of heteroscedasticity-consistent intervals based on the OLS estimate.

In Figure \ref{fig:tbehavior}, we compare the variance of three estimates with respect to the model \eqref{eq:smod} with normal errors and fixed design matrix $\bl{X} \in \mathbb{R}^{n \times 4}$, which has entries drawn independently from a standard normal distribution. In this example, $n = 1000$, and the entries of $\bs{\Omega}$ are set to the $1/(n+1), \dots, n/(n+1)$ quantiles of an inverse gamma distribution with parameters $\nu/2, \nu / 2$. We evaluate the \textit{standardized generalized variance} (SGV,~\citealp{SENGUPTA1987209}) of $\hat{\bs{\beta}}_T$, defined as $|\bl{V}_n^{-1} \bl{B}_n \bl{V}_n^{-1}|^{1/p}$, where
\begin{equation}\label{eq:oracle}
\begin{aligned}
    \bl{V}_n = \sum_{i=1}^n \frac{1}{g_{\omega_0, \nu}(\omega_i)}\bl{x}_i \bl{x}_i^\top,~ \bl{B}_n = \sum_{i=1}^n \frac{1}{f_{\omega_0, \nu}(\omega_i)}\bl{x}_i \bl{x}_i^\top.
\end{aligned}
\end{equation}
We evaluate the SGV both for an oracle version of $\hat{\bs{\beta}}_T$, where the degrees of freedom are set to the true value of $\nu$, and for a version of $\hat{\bs{\beta}}_T$ where the degrees of freedom are fixed at $\nu = 7$. The scale parameter $\omega_0$ is set equal to $1$ for both of these estimates. We also evaluate $|H_{\bl{X}}(\bl{I}_n, \bs{\Omega})|^{1/p}$ and $| H_{\bl{X}}(\bs{\Omega}, \bs{\Omega})|^{1/p}$, corresponding to the standardized generalized variance of the ordinary and weighted least squares estimates, respectively.
\begin{figure}
    \centering
    \includegraphics[scale=0.6]{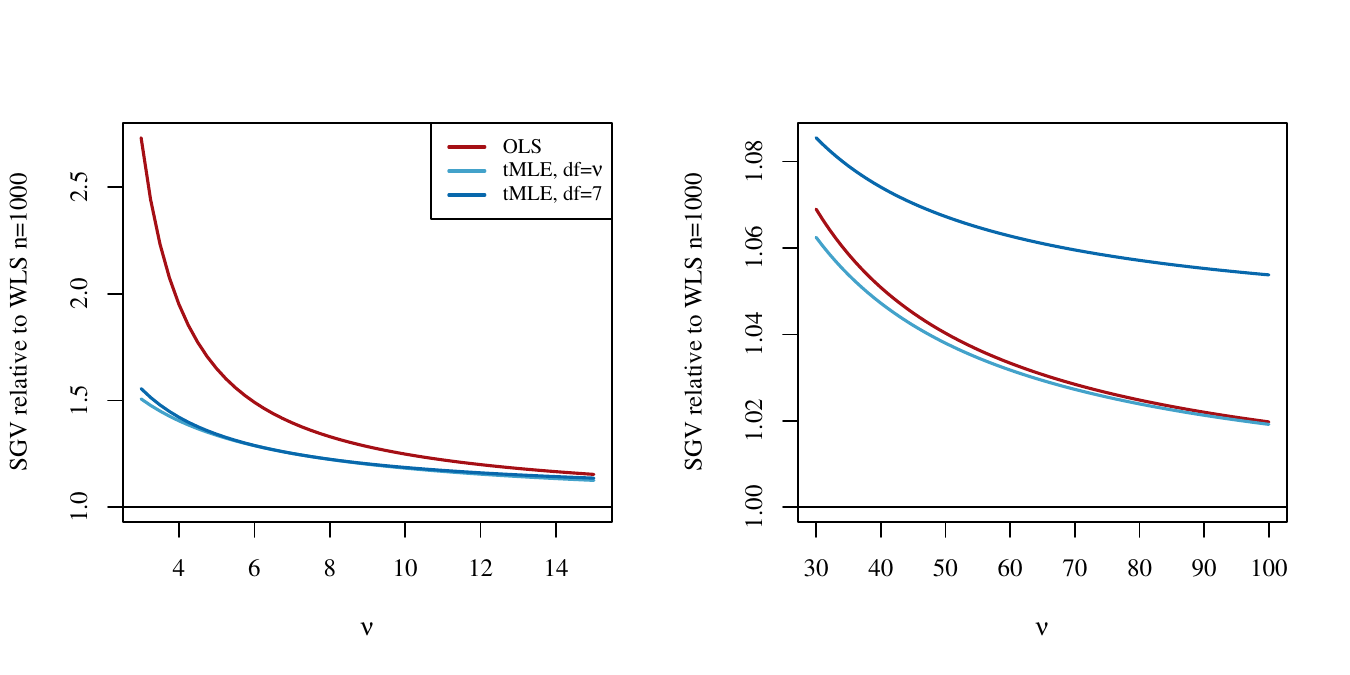}
    \caption{Theoretical behavior of $t$-derived estimates. On the left, a comparison of the standardized generalized variance of three estimates relative to that of the WLS estimate for small degrees of freedom. On the right, the same comparison is made for large degrees of freedom.}
    \label{fig:tbehavior}
\end{figure}

The left panel of Figure \ref{fig:tbehavior} plots the standardized generalized variance of each estimate divided by that of the WLS estimate for values of $\nu$ ranging from $3$ to $15$. The right panel zooms in on the relative standardized generalized variances for the range $\nu = 30$ to $\nu = 100$. For all values of $\nu$ between $3$ and $15$, the oracle $\hat{\bs{\beta}}_T$ has lower standardized generalized variance than the OLS estimate. This is also true of the standardized generalized variance of the non-oracle estimate using the fixed value of $7$ degrees of freedom, and it is interesting to note that these two versions of $\hat{\bs{\beta}}_T$ behave similarly in this range. For values of $\nu$ between $30$ and $100$, the OLS estimate outperforms the non-oracle $\hat{\bs{\beta}}_T$, though the difference between them is small. To summarize, maximum likelihood estimates derived from linear models with independent $t$ errors can be substantially more efficient than the OLS estimate if the dispersion among the elements of $\bs{\Omega}$ is moderate to high. When the dispersion is low, and $\bs{\Omega}$ is close to $\bl{I}_n$, OLS performs better than a non-oracle $t$-derived maximum likelihood estimate, but only by a small amount.

\subsection{Monte Carlo studies of estimation efficiency}

While useful for the purposes of illustration, the previous example is somewhat artificial in terms of the choice of $\bl{X}$ and $\bs{\Omega}$. It also relies only on formulae like those in \eqref{eq:oracle} to calculate the variance of various estimates. The next numerical examples feature a real-data design matrix and approximate SGVs computed using Monte Carlo in addition to those computed using the values of $\bs{\Omega}$ as input. This allows us to compare the theoretical behavior of $t$-derived estimates to their behavior in practice.

Several ground-truth quantities need to be defined for the simulations that follow. First, the design matrix $\bl{X}$ is chosen to be a $2370 \times 17$ matrix, corresponding to a subset of the data collected during a study of the association between the concentration of pesticide byproducts in maternal serum and preterm births \citep{longnecker_association_2001}. Each row of $\bl{X}$ corresponds to a birth occurring between 1959 and 1966. In addition to an intercept term, the columns of $\bl{X}$ are comprised of maternal serum concentrations of 12 environmental contaminants, as well as maternal triglyceride level, age, smoking status, and cholesterol. We scale all non-intercept columns of $\bl{X}$ to have variance equal to $1$.

Next, we set ground truth parameters $\nu, \omega_0, \bs{\beta}$ equal to the maximum marginal likelihood estimates of the parameters in the independent $t$ model \eqref{eq:tmod}, where the dependent variable $\bl{y}$ is the gestational age---also recorded as part of the \cite{longnecker_association_2001} study---of each of the births in $\bl{X}$. These parameters are computed using an EM-algorithm, which we describe in the Appendix (see also \cite{lange_normalindependent_1993, liu_ml_1995}). Finally, the error covariance matrix $\bs{\Omega}$ is set equal to a $2370 \times 2370$ diagonal matrix, whose diagonal entries are independent draws from an inverse gamma distribution with parameters $\nu/2, \nu \omega_0 / 2$. In preparation for the simulation study, we also preallocate the submatrices $\bl{X}_{1:n}$ consisting of the first $n$ rows of $\bl{X}$ for $n \in \{50, 100, 200, 500, 1000, 2370\}$. Similarly, we form the error covariance matrices $\bs{\Omega}_{1:n,1:n}$ consisting of the first $n$ rows and columns of $\bs{\Omega}$.

Using these quantities as our ground-truth, we evaluate the SGV of six estimates of $\bs{\beta}$ with respect to the heteroscedastic normal linear model \eqref{eq:smod}. The estimates are: the OLS estimate, the $t$ estimate with estimated scale parameter and estimated degrees of freedom, the $t$ estimate with estimated scale parameter and $7$ degrees of freedom, the ``oracle" $t$ estimate with scale parameter and degrees of freedom equal to the ground-truth $\omega_0, \nu$, the Huber estimate, and the WLS estimate. For each $n$, the SGV of the non-oracle $t$-derived estimates are computed using Monte Carlo; that is, we simulate $1000$ instances of $\bl{y} \in \mathbb{R}^n$ according to \eqref{eq:smod}, compute a $\bs{\beta}_T^*$ for each instance in order to form an $1000 \times 17$ matrix of estimates, compute the sample covariance matrix of the estimates, and then take the geometric mean of the eigenvalues of this matrix. For the ordinary and weighted least squares estimates, we use the formulae $|H_{\bl{X}_{1:n}}(\bl{I}_n, \bs{\Omega}_{1:n, 1:n})|^{1/17}$ and $|H_{\bl{X}_{1:n}}(\bs{\Omega}_{1:n, 1:n}, \bs{\Omega}_{1:n, 1:n})|^{1/17}$, respectively. For the oracle $t$ estimate, we use $|\bl{V}_n^{-1} \bl{B}_n \bl{V}_n^{-1}|^{1/17}$ with $\bl{V}_n, \bl{B}_n$ defined as in $\eqref{eq:oracle}$.
\begin{figure}
    \centering
    \includegraphics[scale=0.6]{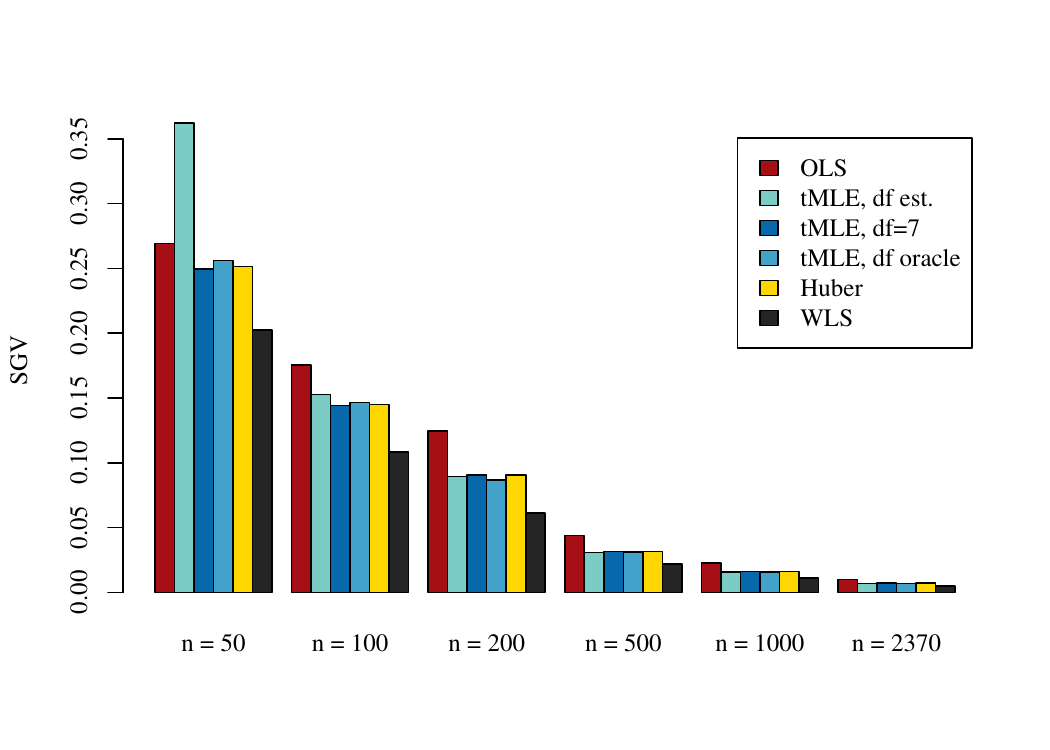}
    \caption{Results of the first simulation study using the \cite{longnecker_association_2001} dataset, which uses a sample of inverse gamma distributed random variables to specify the heteroscedasticity.}
    \label{fig:longnecker}
\end{figure}

Figure \ref{fig:longnecker} displays the SGV for each of the estimates described above. When interpreting these results, it should be kept in mind that only the OLS estimate, the Huber estimate, and the two non-oracle $t$-derived estimates can be computed in practice. For $n \geq 200$, the $t$-derived estimate with estimated degrees of freedom slightly outperforms the estimate with fixed degrees of freedom, though the opposite is true for $n < 200$. The Huber estimate performs comparably to all three $t$-derived estimates. We also note that the behavior of both non-oracle $t$ estimates closely matches that of the oracle $t$ estimate for $n \geq 200$, which provides some assurance that the asymptotic formulae derived in Theorem \ref{thm:asymvar} hold, and that the rate of convergence to this limit is not too slow.

Next, we conduct a simulation similar to the one above, but using a different specification of heteroscedasticity to set up a comparison to parametric FLS estimates. Specifically, we set
\begin{equation*}
    \omega_i = 1.1|x_{i, 15}|^3 |x_{i, 16}|^2,~i=1, \dots, n,
\end{equation*}
which is a particular instance of the flexible parametric model of heteroscedasticity
\begin{equation*}
    v_{\bs{\theta}}(\bl{x}_i) = e^{\theta_1} \prod_{j=2}^p |x_{i, j}|^{\theta_j} 
\end{equation*}
suggested by \cite{romano_resurrecting_2017}. Columns 15 and 16 of $\bl{X}$ correspond to maternal age and smoking status, respectively, and, as before, the first column of $\bl{X}$ is $\bl{1}_n / \sqrt{n}$. All other aspects of this simulation are then the same as above, except we substitute two parametric FLS estimates for the oracle $t$ estimate. The first estimate takes the form $b_{\bl{X}_{1:n}}(\hat{\bs{\Omega}}_n)$ for
\begin{equation*}
    \hat{\bs{\Omega}}_n = \mathrm{diag}(v_{\hat{\bs{\theta}}_n}(\bl{x}_1), \dots, v_{\hat{\bs{\theta}}_n}(\bl{x}_n)),
\end{equation*}
where $\hat{\bs{\theta}}_n$ is the OLS solution to the regression implied by
\begin{equation*}
    \log\{\max(0.01^2, \hat{\varepsilon}_i^2)\} = \theta_1 + \sum_{j=2}^{17} \theta_j \log|x_{i,j}| + z_i,~ z_i \sim N(0,1),~i=1, \dots, n,
\end{equation*}
and $\hat{\varepsilon}_i$ is the $i$th residual from the OLS fit of $\bs{\beta}$. The second estimate is based off of a model of heteroscedasticity suggested by \cite{wooldridge2016introductory}, which is misspecified in this case. It takes the form $b_{\bl{X}_{1:n}}(\hat{\bs{\Omega}}_n)$ for
\begin{equation*}
    \hat{\bs{\Omega}}_n = \mathrm{diag}(u_{\hat{\bs{\theta}}_n}(\bl{x}_1), \dots, u_{\hat{\bs{\theta}}_n}(\bl{x}_n)),
\end{equation*}
for $u_{\bs{\theta}}(\bl{x}_i) = e^{\theta_1 + \sum_{j=2}^p \theta_j x_{ij}}$, where $\hat{\bs{\theta}}_n$ is the OLS solution to the regression implied by
\begin{equation*}
    \log\{\max(0.01^2, \hat{\varepsilon}_i^2)\} = \theta_1 + \sum_{j=2}^{17} \theta_j x_{i,j} + z_i,~ z_i \sim N(0,1),~i=1, \dots, n.
\end{equation*}
We evaluate the SGV of both the correctly specified and the misspecified parametric FLS estimates using Monte Carlo and display them along with the SGV of the other estimates in Figure \ref{fig:longnecker_parametric}.

\begin{figure}
    \centering
    \includegraphics[scale=0.6]{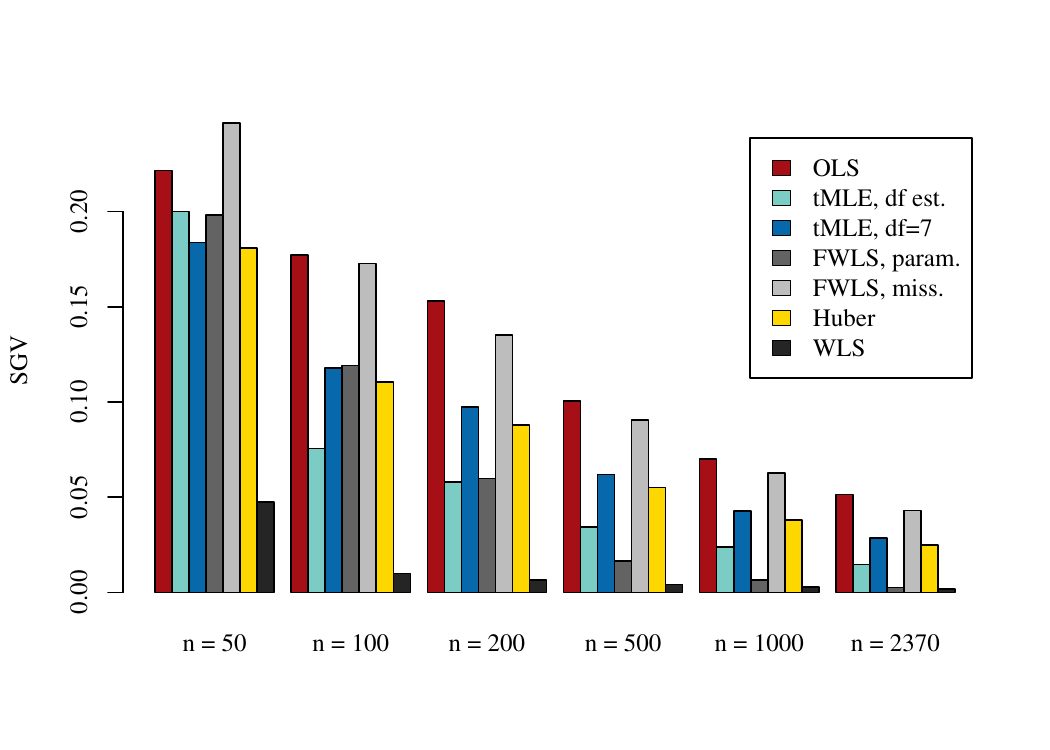}
    \caption{Results of the second simulation study using the \cite{longnecker_association_2001} dataset, which uses a parametric model to specify the heteroscedasticity.}
    \label{fig:longnecker_parametric}
\end{figure}

Here we see that the SGV of both $t$-derived estimates is less than that of the OLS estimate for all values of $n$. Notably, the $t$ estimate with adaptive degrees of freedom is substantially more precise relative to the $t$ estimate with $\nu = 7$ for $n \geq 100$. The $t$-estimate with adaptive degrees of freedom also outperforms the Huber estimate for $n \geq 100$.

These results also suggest that the $t$-derived estimate with adaptive degrees of freedom performs quite favorably relative to the FLS estimate for $n \leq 200$, and this is when the parametric form of heteroscedasticity is correctly specified. Of course, when $n$ is large, the correctly specified parametric FLS estimate is nearly optimal, while the $t$-derived estimates lag behind. However, the misspecified parametric FLS estimate is outperformed by both $t$ estimates and the Huber estimate. All of the characterizations above also apply to the mean squared error of the estimates (see the Appendix, Section \ref{asec:figs}), which, for $n$ large, is proportional to the total variance.


\subsection{Interval coverage for realistic data generating processes}

To evaluate the soundness of the confidence intervals for $\hat{\bs{\beta}}_T$ proposed in Section \ref{sec:inf}, we consider four additional regression problems based on four real-world datasets. The \texttt{food}, \texttt{cps2}, and \texttt{andy} datasets are available in the R package that accompanies the 5th edition of Principles of Econometrics \citep{hill_principles_2018}. The \texttt{Boston} dataset is available in the MASS R package. The regression problems we consider are described by the dependent and independent variables listed below. Dependent variables are marked with an asterisk.

\begin{itemize}[noitemsep]
    \item \texttt{food}
    \begin{itemize}[noitemsep]
        \item \textit{food\_exp}* : weekly food expenditure in \$1s
        \item \textit{income} : weekly income in \$100s
    \end{itemize}
    \item \texttt{cps2}
    \begin{itemize}[noitemsep]
        \item $\log(\textit{wage})$* : log of earnings per hour
        \item \textit{educ} : years of education
        \item \textit{exper} : years of workforce experience
    \end{itemize}
    \item \texttt{andy}
    \begin{itemize}[noitemsep]
        \item \textit{sales}* : monthly sales revenue in \$1000
        \item \textit{price} : price index for all products sold in a given month
        \item \textit{advert} : expenditure on advertising in \$1000s
    \end{itemize}
    \item \texttt{Boston}
    \begin{itemize}[noitemsep]
        \item $\log(\textit{medv})$* : log of median value of owner-occupied homes in \$1000s.
        \item \textit{nox} : nitrogen oxides concentration (parts per 10 million)
        \item \textit{dis} : weighted mean of distances to five Boston employment centers
        \item \textit{rm} : average number of rooms per dwelling
        \item \textit{ptratio} : pupil-teacher ratio by town
    \end{itemize}
\end{itemize}

For each dataset, we use a wild bootstrap method to generate artificial replicate datasets. Specifically, we generate new datasets according to
\begin{equation*}
    y^*_i = \bl{x}_i^\top\hat{\bs{\beta}}_{\mathrm{OLS}} + \frac{\hat{\epsilon}_i}{\sqrt{1 - h_i}} z^*_i,~~z_i^* \sim N(0, 1),~~~i = 1, \dots, n,
\end{equation*}
where $\hat{\bs{\beta}}_{\mathrm{OLS}}$ is the OLS estimate of $\bs{\beta}$ computed on the real dependent variable $\bs{y}$, $\hat{\epsilon}_i$ is the $i$th residual from this OLS fit, and $h_i$ is the $i$th diagonal element of $\bl{X}(\bl{X}^\top \bl{X})^{-1} \bl{X}^\top$. For each replicate, we compute six different 95\% confidence intervals for each entry of $\bs{\beta}$, and we evaluate the width of each interval as well as whether it contains the ground truth (in this case, the appropriate entry of $\hat{\bs{\beta}}_{\mathrm{OLS}}$), or not. The results of this study are displayed in Table \ref{tab:coverage}.

\begin{table}\centering
\begin{tabular}{@{}llcccccc@{}}\toprule
& & Hom. & HC0 & HC1 & HC2 & HC3 & $t$ \\ \midrule
\texttt{food} & \textit{income} & 97.9 & 93.5 & 94.2 (1.03) & 94.6 (1.04) & 95.4 (1.07) & 94.1 (\textbf{0.74}) \\ \midrule
& \textit{educ} & 94.6 & 95.0 & 95.0 (1.00) & 95.1 (1.00) & 95.1 (1.01) & 95.0 (\textbf{0.53}) \\
\texttt{cps2} & \textit{exper} & 95.1 & 95.6 & 95.6 (1.00) & 95.7 (1.01) & 95.8 (1.01) & 95.5 (\textbf{0.54}) \\
& \textit{exper}$^2$ & 93.0 & 95.2 & 95.3 (1.00) & 95.4 (1.00) & 95.7 (1.02) & 95.0 (\textbf{0.53})  \\\midrule
\multirow{2}{*}{\texttt{andy}}& \textit{price} & 96.5 & 93.9 & 94.5 (1.02) & 94.6 (1.02) & 95.0 (1.05) & 94.2 (\textbf{0.77}) \\
& \textit{advert} & 95.5 & 94.3 & 94.8 (1.02) & 94.8 (1.02) & 95.2 (1.05) & 94.4 (\textbf{0.75}) \\
\midrule
& $\log(\textit{nox})$ & 93.6 & 95.1 & 95.3 (1.00) & 95.3 (1.00) & 95.4 (1.02) & 95.1 (\textbf{0.36})\\
\multirow{2}{*}{\texttt{Boston}}& $\log(\textit{dis})$ & 89.7 & 95.3 & 95.4 (1.00) & 95.4 (1.01) & 95.5 (1.02) & 94.6 (\textbf{0.33}) \\
& \textit{rm} & 86.1 & 95.2 & 95.3 (1.00) & 95.6 (1.01) & 95.7 (1.02) & 94.6 (\textbf{0.34}) \\
& \textit{ptratio} & 98.9 & 95.2 & 95.3 (1.00) & 95.3 (1.01) & 95.4 (1.02) & 95.0 (\textbf{0.49}) \\
\bottomrule
\end{tabular}
\caption{Coverage percentage of confidence intervals for non-intercept regression coefficients in four regression problems. The ratio of the average interval width to the average width of the HC0 interval are included in parentheses for the HC1-3 and $t$ intervals.}
\label{tab:coverage}
\end{table}

The column of Table \ref{tab:coverage} labeled ``Hom."  corresponds to the usual confidence interval for a regression coefficient under the assumption of homoscedasticity, which is not guaranteed to have the correct coverage under heteroscedasticity. Indeed there is evidence that under- and over- coverage occurs for certain coefficients. The HC columns of Table \ref{tab:coverage} correspond to the \cite{white_heteroskedasticity-consistent_1980} heteroscedasticity-consistent intervals (HC0) and variants thereof (HC1-3, \citealp{mackinnon_heteroskedasticity-consistent_1985}), which all come with coverage guarantees under heteroscedasticity. All HC intervals appear to achieve the correct coverage, and their average widths relative to that of the HC0 interval are included in parentheses next to their coverage level. Finally, the coverage levels of the confidence intervals based on the asymptotic covariance of $\hat{\bs{\beta}}_T$ are presented in the right-most column. The coverage of these $t$-derived intervals is comparable to that of the HC intervals, while the average width is significantly less than that of the HC intervals for all coefficients and all regression problems.


\section{Discussion}\label{sec:disc}

The experiments of the previous section suggest that $t$-derived estimates can be substantially more efficient than the OLS estimate in the heteroscedastic linear model with normally distributed errors. The theoretical results in this article suggest that this improvement in efficiency can be attributed to a quasi-oracle property of the $t$-derived estimates: in the limit, these estimates are sub-optimal, but are still preferable to OLS because they are similar to a FLS estimate using subscedastic weights. From the perspective of point estimation in the heteroscedastic linear model, we contend that $t$-derived estimates---perhaps with fixed degrees of freedom for $n < 100$---may be used as a default estimate preferable to OLS. For inference purposes, we present evidence that asymptotically valid confidence intervals for $t$-derived estimates may be obtained from an empirical approximation to the asymptotic covariance matrix. A more careful theoretical justification for the use of these intervals may be warranted in future work. However, the empirical results presented here are promising in that the average interval width tends to be much smaller than that of popular heteroscedasticity-consistent alternatives.

{In other future theoretical work, we hope to establish more results that apply to non-diagonal covariance matrices. Part of the challenge with doing so is that such results must depend on simultaneous conditions on the eigenvectors and eigenvalues of $\bs{\Omega}$ and $\tilde{\bs{\Omega}}$, which may be more opaque than the relatively simple subscedastic criterion in \eqref{eq:main}.} {It is also possible that subscedastic sets described by different criteria than \eqref{eq:main} might arise by using other measures of relative multivariate variance besides the generalized and total variances considered in this article. However, we conjecture that \eqref{eq:main} is the necessary and sufficient subscedastic condition for at least one additional measure of multivariate relative variance, described by
\begin{equation}\label{eq:rtrace}
    \frac{1}{p}\tr \left\{ H_{\bl{X}}(\tilde{\bs{\Omega}}, \bs{\Omega})H_{\bl{X}}(\bl{I}_n, \bs{\Omega})^{-1}\right\}.
\end{equation}
While it is relatively easy to verify the necessity of \eqref{eq:main} in this case, we have not yet found a proof of its sufficiency.}

Finally, we have also shown that subscedastic sets can be a lens through which to analyze regression $M$-estimates other than those derived from the $t$-distribution. It would be interesting to investigate more connections between subscedasticity, the influence functions of various robust regression estimates, and the distribution of the error terms.

\bibliographystyle{apalike}
\bibliography{paper-ref.bib}

\appendix

\section{Proofs}\label{asec:proofs}

\subsection{{Proof of Proposition \ref{prop:mono}}}

\begin{proof}
Let $p=1$, and for each $i \in \{1, \dots, n\}$, let $\tilde{\omega}_i = g(\omega_i)$ where $g$ is a monotone non-decreasing function such that the function $f(\omega) = g(\omega)/\omega$ is monotone non-increasing. Set $\tilde{\bs{\Phi}} = \tilde{\bs{\Omega}}^{-1}$, and let $\mathcal{V}^n$ denote the set of all $n$-dimensional unit vectors. 

As discussed in the main text, for any unit vector $\bl{u}$ the functions $\mathrm{e}_{\bl{u}}$ and $ \mathrm{c}_{\bl{u}}$ behave, respectively, like the expectation and covariance functions of discrete random variables with supports determined by the diagonal entries of $\tilde{\bs{\Omega}}, \bs{\Omega}$ and probability mass functions determined by $\bl{u}$. Therefore, by \cite{schmidt_inequalities_2014} Corollary 3.1,
\begin{equation*}
    \mathrm{c}_{\bl{u}}(\tilde{\bs{\Phi}}, \bs{\Omega}) \leq 0,~\forall \bl{u} \in \mathcal{V}^n,
\end{equation*}
and
\begin{equation*}
    \mathrm{c}_{\bl{u}}(\tilde{\bs{\Phi}}, \tilde{\bs{\Phi}} \bs{\Omega}) \leq 0,~\forall \bl{u} \in \mathcal{V}^n.
\end{equation*}
Thus, 
\begin{equation*}
    \mathrm{c}_{\bl{u}}(\tilde{\bs{\Phi}}, \tilde{\bs{\Phi}} \bs{\Omega}) + \mathrm{e}_{\bl{u}}(\tilde{\bs{\Phi}})\mathrm{c}_{\bl{u}}(\tilde{\bs{\Phi}}, \bs{\Omega})\leq 0,~\forall \bl{u} \in \mathcal{V}^n,
\end{equation*}
so $\tilde{\bs{\Omega}} \in \mathcal{C}^1_{\bs{\Omega}}$.
\end{proof}

\subsection{{Proof of Proposition \ref{prop:bd}}}

\begin{proof}
    Let $g$ be a function satisfying \eqref{eq:grm}. Define $f : \mathbb{R}_+ \rightarrow \mathbb{R}_+$ as
    \begin{equation*}
        f(\omega) = \frac{1}{1/g(\omega) + \lambda} + \gamma.
    \end{equation*}
    Because $g$ is increasing in $\omega$, $f$ must be increasing in $\omega$. Looking at
    \begin{equation*}
        f(\omega)/w = \frac{1}{w/g(\omega) + w\lambda} + \gamma/w
    \end{equation*}
    one sees that $\gamma/w$ is decreasing in $\omega$, $\omega \lambda$ is increasing in $\omega$, and because $g$ satisfies \eqref{eq:grm} $\omega / g(\omega)$ is increasing in $\omega$. This implies that $f(\omega)/\omega$ is decreasing in $\omega$; hence, $f$ satisfies \eqref{eq:grm}.

    Since $g$ is a positive-valued function, and since both $f$ and $g$ are increasing in $\omega$, $f$ will be bounded below by the limit as $g(\omega) \rightarrow 0$, which is $\gamma$. It will be bounded above by the limit as $g(\omega) \rightarrow \infty$, which is $\gamma + 1/\lambda$.
\end{proof}

\subsection{{Proof of Proposition \ref{prop:majorization}}}

\begin{proof}
    Let $(\omega_{(1)}, \dots, \omega_{(n)})$ be the arrangement of $\{\omega_i\}_{i=1}^n$ in decreasing order. Since $g$ satisfies \eqref{eq:grm}, $\omega_{(i)} / g(\omega_{(i)})$ is decreasing in $i = 1, \dots, n$. Apply \cite{marshall_inequalities_2011} Proposition B.1.b. to complete the proof.
\end{proof}

\subsection{{Proof of Theorem \ref{thm:main}}}
To prove Theorem \ref{thm:main}, we first present the following lemma:
\begin{lemma}\label{lem:zeros}
Let $\tilde{\bs{\Phi}}, \bs{\Omega} \in \mathcal{D}_+^n$. Let $\mathcal{V}^n$ denote the set of all $n$-dimensional unit vectors. Define the function $k : \mathbb{R}^n \rightarrow \mathbb{R}$ by
\begin{equation*}
    k(\bl{u}) = \bl{u}^\top \tilde{\bs{\Phi}} \bs{\Omega} \tilde{\bs{\Phi}} \bl{u} - (\bl{u}^\top \tilde{\bs{\Phi}} \bl{u})^2 (\bl{u}^\top \bs{\Omega} \bl{u}).
\end{equation*}
Then
\begin{equation*}
    \sup_{\bl{u} \in \mathcal{V}^n}~k(\bl{u}) = \sup_{\bl{u} \in \mathcal{V}^n, \|\bl{u}\|_0 \leq 2}~k(\bl{u}).
\end{equation*}
\end{lemma}

\begin{proof}
We will prove the statement by providing, for any $\bl{u} \in \mathcal{V}^n$, a corresponding $\bl{v} \in \mathcal{V}^n$ with $\|\bl{v}\|_{0} \leq 2$ such that
\begin{equation*}
    k(\bl{u}) \leq k(\bl{v}).
\end{equation*}
Given $\tilde{\bs{\Phi}}, \bs{\Omega} \in \mathcal{D}_+^n$, let $\bl{u} \in \mathcal{V}^n$ and let $\bs{\alpha} = \bl{u} \odot \bl{u}$, where $\odot$ denotes the Hadamard product. Further, let $\bs{\alpha}^{1/2}$ denote the entrywise positive square root of $\bs{\alpha}$ so that $(\bs{\alpha}^{1/2})_i = |u_i|$ for each $i \in \{1, \dots, n\}$. Note that in terms of $\bs{\alpha}$, $k(\bl{u})$ may be written as
\begin{equation}\label{eq:simplex}
    \bs{\alpha}^\top (\tilde{\bs{\phi}} \odot \tilde{\bs{\phi}} \odot \bs{\omega}) - (\bs{\alpha}^\top \tilde{\bs{\phi}})^2 (\bs{\alpha}^\top \bs{\omega}),
\end{equation}
where $\tilde{\phi}, \bs{\omega}$ are $n$-dimensional vectors containing the diagonal elements of $\tilde{\bs{\Phi}}, \bs{\Omega}$, respectively.
    
If $\|\bl{u}\|_0 \leq 2$, then setting $\bl{v} := \bl{u}$ yields a trivial bound satisfying the norm constraint. Suppose instead that $\|\bl{u}\|_0 = d > 4$, and let $\mathcal{I}_{\bl{u}}$ be the index set of the non-zero entries of $\bl{u}$. Then there exists a vector $\bs{\eta} \in \mathbb{R}^{n}$ such that
\begin{equation}\label{eq:restrict}
    \bs{\eta}^\top (\tilde{\bs{\phi}} \odot \tilde{\bs{\phi}} \odot \bs{\omega}) = 0,~\bs{\eta}^\top \bs{\omega}= 0,~ \bs{\eta}^\top \tilde{\bs{\phi}}=0, ~ \bs{\eta}^\top \bl{1}_n =0
\end{equation}
and $\eta_i = 0$ for $i \notin \mathcal{I}_{\bl{u}}$. Such an $\bs{\eta}$ exists because these restrictions define a system of at most $n-1$ independent linear equations in $n$ variables. To see this, note that there are four linear equations in \eqref{eq:restrict}, and there are $n-d$ linear equations that enforce $\eta_i = 0$ for $i \notin \mathcal{I}_{\bl{u}}$. As $d > 4$, the number of linear equations is $n - d + 4 < n$. If $\bs{\omega}, \tilde{\bs{\phi}}$ and $\bl{1}_n$ are linearly independent, then the number of independent linear equations is exactly $n-d+4$. If they are not, then the effective number of independent linear equations is less than $n-d+4$. 
    
An $\bs{\eta}$ satisfying the restrictions above must also have at least one negative entry and one positive entry among its non-zero entries. This is due to the fact that $\bs{\omega}, \tilde{\bs{\phi}}$ and $\bl{1}_n$ are all vectors with strictly positive entries, so \eqref{eq:restrict} implies that $\bs{\eta}$ cannot lie in either the positive orthant or the negative orthant. Consequently, there exists an $\epsilon > 0$ such that $\| \bs{\alpha} + \epsilon \bs{\eta} \|_0 = d-1$ and the entries of $\bs{\alpha} + \epsilon \bs{\eta}$ are all non-negative. Specifically, if
\begin{equation*}
	\epsilon = \min_{\{i : \eta_i < 0\}} |\alpha_i / \eta_i|,
\end{equation*}
then no entry of $\bs{\alpha} + \epsilon \bs{\eta}$ will fall below zero, and $\bs{\alpha} + \epsilon \bs{\eta}$ will have one additional entry equal to zero (the entry corresponding to the minimum above) relative to $\bs{\alpha}$. 

Setting $\bl{v} := (\bs{\alpha} + \epsilon \bs{\eta})^{1/2}$ produces a unit vector with $L_0$ norm equal to $d-1$ such that $k(\bl{u}) = k(\bl{v})$. The fact that $\bl{v}$ is a unit vector follows from
\begin{equation*}
\begin{aligned}
    \bl{v}^\top \bl{v} &= \sum_{i=1}^n |\alpha_i + \epsilon \eta_i| \\
                       &= \sum_{i=1}^n (\alpha_i + \epsilon \eta_i) \\
                       &= \bl{1}^\top \bs{\alpha} + \epsilon \bl{1}^\top \bs{\eta} \\
                       &= 1,
\end{aligned}
\end{equation*}
where we used the last linear equation in \eqref{eq:restrict} to obtain the last equality. The fact that $\bl{v}$ has $L_0$ norm equal to $d-1$ follows from the fact that $\|\bs{\alpha} + \epsilon \bs{\eta}\|_0 = d-1$. Finally, $k(\bl{u}) = k(\bl{v})$ because the first three linear equations in \eqref{eq:restrict} ensure that none of the terms in \eqref{eq:simplex} change when $\bs{\alpha} + \epsilon \bs{\eta}$ is substituted for $\bs{\alpha}$. Each of the steps above can be repeated until one begins the process with $d = 5$ and obtains a valid $\bl{v}$ with $L_0$ norm equal to $4$. This demonstrates that
\begin{equation*}
    \sup_{\bl{u} \in \mathcal{V}^n}~k(\bl{u}) = \sup_{\bl{u} \in \mathcal{V}^n, \|\bl{u}\|_0 \leq 4}~k(\bl{u}).
\end{equation*}
It remains to address the case that $2 < \|\bl{u}\|_0 \leq 4$.
    
If $\|\bl{u}\|_0 = 4$, then there exists a non-zero vector $\bs{\eta} \in \mathbb{R}^n$ such that
\begin{equation}\label{eq:restrict2}
    \bs{\eta}^\top (\tilde{\bs{\phi}} \odot \tilde{\bs{\phi}} \odot \bs{\omega}) > 0,~\bs{\eta}^\top \bs{\omega}= 0,~ \bs{\eta}^\top \tilde{\bs{\phi}}=0, ~ \bs{\eta}^\top \bl{1}_n =0
\end{equation}
and $\eta_i = 0$ for each $i \notin \mathcal{I}_{\bl{u}}$. Reasoning as before, such an $\bs{\eta}$ exists because there is at least a $1$-dimensional subspace of $\mathbb{R}^n$ where the stated equalities are satisfied. If we choose an arbitrary vector in this subspace, it will either have positive or negative dot product with $(\tilde{\bs{\phi}} \odot \tilde{\bs{\phi}} \odot \bs{\omega})$. If the dot product is positive, we can choose $\bs{\eta}$ to be this vector. If it is negative then we can choose $\bs{\eta}$ to be the negation of this vector. Having found such an $\bs{\eta}$, we may again choose
\begin{equation*}
	\epsilon = \min_{\{i : \eta_i < 0\}} |\alpha_i / \eta_i|,
\end{equation*}
and note, as before, that $(\bs{\alpha} + \epsilon \bs{\eta})^{1/2}$ is a unit vector with three non-zero, positive entries. Here, setting $\bl{v} := (\bs{\alpha} + \epsilon \bs{\eta})^{1/2}$ yields $k(\bl{u}) < k(\bl{v})$ due to the inequality in \eqref{eq:restrict2}. This implies that it suffices to consider the $\|\bl{u}\|_0 = 3$ case.

If $\|\bl{u}\|_0 = 3$, then there exists a non-zero vector $\bs{\eta} \in \mathbb{R}^n$ such that
\begin{equation}\label{eq:restrict3}
    \bs{\eta}^\top \left\{(\tilde{\bs{\phi}} \odot \tilde{\bs{\phi}} \odot \bs{\omega}) - (\bs{\alpha}^\top \tilde{\bs{\phi}})^2 \bs{\omega} \right\} > 0,~ \bs{\eta}^\top \tilde{\bs{\phi}}=0, ~ \bs{\eta}^\top \bl{1}_n =0
\end{equation}
and $\eta_i = 0$ for each $i \notin \mathcal{I}_{\bl{u}}$. Such an $\bs{\eta}$ exists because there is at least a $1$-dimensional subspace of $\mathbb{R}^n$ where the stated equalities are satisfied. If we choose an arbitrary vector in this subspace, it will either have positive or negative dot product with the vector $(\tilde{\bs{\phi}} \odot \tilde{\bs{\phi}} \odot \bs{\omega}) - (\bs{\alpha}^\top \tilde{\bs{\phi}})^2 \bs{\omega}$. If the dot product is positive, we can choose $\bs{\eta}$ to be this vector. If it is negative then we can choose $\bs{\eta}$ to be the negation of this vector. Having found such an $\bs{\eta}$, we may again choose
\begin{equation*}
	\epsilon = \min_{\{i : \eta_i < 0\}} |\alpha_i / \eta_i|,
\end{equation*}
and note as before that $(\bs{\alpha} + \epsilon \bs{\eta})^{1/2}$ is a unit vector with two non-zero, positive entries. Setting $\bl{v} := (\bs{\alpha} + \epsilon \bs{\eta})^{1/2}$ yields $k(\bl{u}) < k(\bl{v})$ due to the inequality in \eqref{eq:restrict3}.

Hence, for any $\bl{u} \in \mathcal{V}^n$ there exists a $\bl{v} \in \mathcal{V}^n$ such that $\|\bl{v}\|_0 \leq 2$ and $k(\bl{u}) < k(\bl{v})$, so
\begin{equation*}
    \sup_{\bl{u} \in \mathcal{V}^n}~k(\bl{u}) = \sup_{\bl{u} \in \mathcal{V}^n, \|\bl{u}\|_0 \leq 2}~k(\bl{u}), 
\end{equation*}
which completes the proof. 
\end{proof}

Now we prove Theorem \ref{thm:main}:
\begin{proof}[Proof of Theorem \ref{thm:main}]
For the entirety of the proof we assume that the diagonal entries of $\bs{\Omega}, \tilde{\bs{\Omega}} \in \mathcal{D}_+^n$ are distinct. The result may be generalized by a continuity argument to the case of non-distinct diagonal entries.

($\impliedby$) We will prove the necessity of \eqref{eq:main} by proving that if it does not hold for some $\tilde{\bs{\Omega}}$, then $\tilde{\bs{\Omega}} \notin \mathcal{C}^1_{\bs{\Omega}}$. Suppose that the diagonal entries of $\tilde{\bs{\Omega}}$ do not satisfy
\begin{equation*}
    1 \leq \frac{\tilde{\omega}_i}{\tilde{\omega}_j} \leq 2\frac{\omega_i}{\omega_j} - 1
\end{equation*}
for all $\omega_i > \omega_j$, $i \in \{1, \dots, n\}, j \in \{1, \dots, n\}$. Then there exists at least one pair of indices, $(i,j)$, for which $\omega_i > \omega_j$, and either
\begin{equation}\label{eq:opt1}
    \frac{\tilde{\omega}_i}{\tilde{\omega}_j} > 2\frac{\omega_i}{\omega_j} - 1
\end{equation}
or
\begin{equation}\label{eq:opt2}
    \frac{\tilde{\omega}_i}{\tilde{\omega}_j} < 1.
\end{equation}
First, suppose that \eqref{eq:opt1} holds for a pair of indices $(i,j)$, and let $\bl{u}$ be a unit vector with entries equal to zero everywhere except at the indices $i$ and $j$. Then we may write $t := u_i^2$ and $1-t := u_j^2$, and after expanding and collecting terms, find that
\begin{equation*}
\begin{aligned}
    \mathrm{c}_{\bl{u}}(\tilde{\boldsymbol{\Phi}}, \tilde{\boldsymbol{\Phi}} \bs{\Omega}) + \mathrm{e}_{\bl{u}}(\tilde{\boldsymbol{\Phi}}) \mathrm{c}_{\bl{u}}(\tilde{\boldsymbol{\Phi}}, \bs{\Omega}) &= t(1-t)(\tilde{\phi}_i - \tilde{\phi}_j)^2 (\omega_i - \omega_j)\left[ t - \frac{2 \tilde{\phi}_j \omega_j - (\tilde{\phi}_i + \tilde{\phi}_j)\omega_i}{(\tilde{\phi}_i - \tilde{\phi}_j)(\omega_i - \omega_j)}\right],
\end{aligned}
\end{equation*}
where $\tilde{\bs{\Phi}} = \tilde{\bs{\Omega}}^{-1}$. This is a third-degree polynomial in $t$, with roots at $r_1=0$, $r_2=1$, and 
\begin{equation}\label{eq:thirdroot}
    r_3 = \frac{2 \tilde{\phi}_j \omega_j - (\tilde{\phi}_i + \tilde{\phi}_j)\omega_i}{(\tilde{\phi}_i - \tilde{\phi}_j)(\omega_i - \omega_j)},
\end{equation}
respectively. Since we assumed \eqref{eq:opt1}, it holds that $\tilde{\phi}_j > \tilde{\phi}_i$. Therefore, the denominator above is negative, and the sign of $r_3$ depends on whether the numerator above is positive or negative. 
    
Supposing first that the numerator is negative implies that $r_3$ must be positive. Additionally, $r_3$ must be strictly less than 1 because \eqref{eq:opt1} implies
\begin{equation*}
    2 \tilde{\phi}_i \omega_i < (\tilde{\phi}_i + \tilde{\phi}_j) \omega_j .
\end{equation*}
Subtracting $\tilde{\phi}_i \omega_i + 2 \tilde{\phi}_j \omega_j$ from both sides, and adding $\tilde{\phi}_j \omega_i$ to both sides yields
\begin{equation*}
    \tilde{\phi}_i \omega_i + \tilde{\phi}_j \omega_i - 2\tilde{\phi}_j \omega_j < \tilde{\phi}_j \omega_i - \tilde{\phi}_i \omega_i - \tilde{\phi}_j \omega_j + \tilde{\phi}_i \omega_j,
\end{equation*}
which implies that
\begin{equation*}
    (\tilde{\phi}_i + \tilde{\phi}_j)\omega_i - 2 \tilde{\phi}_j\omega_j < (\tilde{\phi}_j - \tilde{\phi}_i)(\omega_i - \omega_j),
\end{equation*}
so $0 < r_3 < 1$. Since there are then three real roots of the cubic equation in the interval $t \in [0, 1]$, we conclude that there exists a $t \in (0, 1)$ so that the cubic takes on a positive value. Therefore, there exists a $\bl{u}$ such that $H_{\bl{u}}(\tilde{\bs{\Omega}}, \bs{\Omega}) > H_{\bl{u}}(\bl{I}_n, \bs{\Omega})$.

If instead $r_3$ is negative, we conclude that the cubic polynomial is positive in the entire interval $(0, 1)$. This is because the sign of the leading coefficient of the cubic polynomial is negative, so the polynomial must take on positive values on $(-\infty, r_3)$, negative values on $(r_3, 0)$, positive values on $(0, 1)$, and negative values on $(1, \infty)$. Choosing any $t$ in $(0, 1)$, we conclude that there exists a $\bl{u}$ such that $H_{\bl{u}}(\tilde{\bs{\Omega}}, \bs{\Omega}) > H_{\bl{u}}(\bl{I}_n, \bs{\Omega})$. Since we came to this conclusion both when $r_3$ was assumed positive and when it was assumed negative, \eqref{eq:opt1} implies that there exists a $\bl{u}$ for which $H_{\bl{u}}(\tilde{\bs{\Omega}}, \bs{\Omega}) > H_{\bl{u}}(\bl{I}_n, \bs{\Omega})$. Hence \eqref{eq:opt1} implies $\tilde{\bs{\Omega}} \notin \mathcal{C}_{\bs{\Omega}}$.

If \eqref{eq:opt2} holds for a pair of indices $(i,j)$, then let $\bl{u}$ be a unit vector with entries equal to zero everywhere except at the indices $i$ and $j$ and set $t := u_i^2$ and $1-t := u_j^2$ as before. Here, since \eqref{eq:opt2} implies $\tilde{\phi}_j < \tilde{\phi}_i$, we see that the denominator in \eqref{eq:thirdroot} is positive. The numerator, on the other hand, must be negative, since
\begin{equation*}
    2 \tilde{\phi}_j \omega_j < (\tilde{\phi}_i + \tilde{\phi}_j) \omega_j < (\tilde{\phi}_i + \tilde{\phi}_j) \omega_i.
\end{equation*}
Hence, when \eqref{eq:opt2} holds, $r_3 < 0$. As before, this leads to the conclusion that the cubic polynomial is positive on the entire interval $(0, 1)$, so we may find a $\bl{u}$ such that $H_{\bl{u}}(\tilde{\bs{\Omega}}, \bs{\Omega}) > H_{\bl{u}}(\bl{I}_n, \bs{\Omega})$. Hence \eqref{eq:opt2} implies $\tilde{\bs{\Omega}} \notin \mathcal{C}_{\bs{\Omega}}$.

($\implies$) Assume that the diagonal entries of $\tilde{\bs{\Omega}}$ satisfy
\begin{equation*}
    1 \leq \frac{\tilde{\omega}_i}{\tilde{\omega}_j} \leq 2\frac{\omega_i}{\omega_j} - 1
\end{equation*}
for all $\omega_i > \omega_j$, $i \in \{1, \dots, n\}, j \in \{1, \dots, n\}$. Note that $\tilde{\bs{\Omega}} \in \mathcal{C}^1_{\bs{\Omega}}$ if and only if
\begin{equation*}
    \underset{\bl{u} \in \mathcal{V}^n}{\sup}~ H_{\bl{u}}(\tilde{\bs{\Omega}}, \bs{\Omega}) - H_{\bl{u}}(\tilde{\bs{\Omega}}, \bl{I}_n) \leq 0 .
\end{equation*}
By rearrangement of terms, this can be expressed alternately as
\begin{equation*}
    \underset{\bl{u} \in \mathcal{V}^n}{\sup} ~\mathrm{c}_{\bl{u}}(\tilde{\boldsymbol{\Phi}}, \tilde{\boldsymbol{\Phi}} \bs{\Omega}) + \mathrm{e}_{\bl{u}}(\tilde{\boldsymbol{\Phi}}) \mathrm{c}_{\bl{u}}(\tilde{\boldsymbol{\Phi}}, \bs{\Omega}) \leq 0,
\end{equation*}
or as
\begin{equation}\label{eq:kbd}
    \underset{\bl{u} \in \mathcal{V}^n}{\sup} ~ k(\bl{u}) \leq 0,
\end{equation}
where $k$ is defined as in Lemma \ref{lem:zeros}, and as before $\tilde{\bs{\Phi}} = \tilde{\bs{\Omega}}^{-1}$. Hence, it suffices to show that \eqref{eq:kbd} holds to prove that $\tilde{\bs{\Omega}} \in \mathcal{C}^1_{\bs{\Omega}}$. Applying Lemma \ref{lem:zeros}, the maxima of $k$ are attained for $\|\bl{u}\|_0 \leq 2$. Therefore,
\begin{equation}\label{eq:polys}
    \underset{\bl{u} \in \mathcal{V}^n}{\sup} ~ k(\bl{u}) = \underset{i \neq j}{\max}~\underset{0 \leq t \leq 1}{\sup}~t(1-t)(\tilde{\phi}_i - \tilde{\phi}_j)^2 (\omega_i - \omega_j)\left[ t - \frac{2 \tilde{\phi}_j \omega_j - (\tilde{\phi}_i + \tilde{\phi}_j)\omega_i}{(\tilde{\phi}_i - \tilde{\phi}_j)(\omega_i - \omega_j)}\right]
\end{equation}
If the diagonal entries of $\tilde{\bs{\Omega}}$ satisfy
\begin{equation*}
        1 \leq \frac{\tilde{\omega}_i}{\tilde{\omega}_j} \leq 2\frac{\omega_i}{\omega_j} - 1
\end{equation*}
for all $\omega_i > \omega_j$, $i \in \{1, \dots, n\}, j \in \{1, \dots, n\}$, then none of the above polynomials can attain a positive value. To see this, re-write the condition above in terms of $\tilde{\phi}_i, \tilde{\phi}_j$ and conclude from the first inequality that $\tilde{\phi}_i < \tilde{\phi}_j$. From the second inequality, derive
\begin{equation*}
\begin{aligned}
    \frac{\tilde{\phi}_j}{\tilde{\phi}_i} \leq 2 \frac{\omega_i}{\omega_j} - 1 &\implies \tilde{\phi}_j \omega_j \leq 2 \phi_i \omega_i - \phi_i \omega_j \\
    &\implies \tilde{\phi}_j \omega_i - \tilde{\phi}_j \omega_j - \tilde{\phi}_i \omega_i + \tilde{\phi}_i \omega_j \leq \tilde{\phi}_i \omega_i + \tilde{\phi}_j \omega_i - 2\tilde{\phi}_j \omega_j \\
    &\implies (\tilde{\phi}_j - \tilde{\phi}_i)(\omega_i - \omega_j) \leq (\tilde{\phi}_i + \tilde{\phi}_j) \omega_i - 2 \tilde{\phi}_j \omega_j
\end{aligned}
\end{equation*}
Since $\tilde{\phi}_i < \tilde{\phi}_j$, we conclude that
\begin{equation*}
    \frac{2 \tilde{\phi}_j \omega_j - (\tilde{\phi}_i + \tilde{\phi}_j) \omega_i}{(\tilde{\phi}_i - \tilde{\phi}_j)(\omega_i - \omega_j)} \geq 1,
\end{equation*}
so the third root of all polynomials in \eqref{eq:polys} is greater than or equal to 1. Denote this root by $r_3$ as before, and first assume $r_3 > 1$. In this case, since the leading coefficient in each of the cubic polynomials in \eqref{eq:polys} is negative, each polynomial must be positive on $(-\infty, 0)$, negative on $(0, 1)$, positive on $(1, r_3)$, and negative on $(r_3, \infty)$. If $r_3=1$, then there is a repeated root at $1$, and the polynomial is non-positive on $(0, \infty)$. Thus, all polynomials in \eqref{eq:polys} are non-positive on the interval $[0, 1]$.
\end{proof}

\subsection{{Proof of Corollary \ref{cor:main1}}}

\begin{proof}

For the entirety of the proof we assume that the diagonal entries of $\bs{\Omega}, \tilde{\bs{\Omega}} \in \mathcal{D}_+^n$ are distinct. The result may be generalized by a continuity argument to the case of non-distinct diagonal entries.

($\implies$) Assume that the diagonal entries of $\tilde{\bs{\Omega}}$ satisfy
\begin{equation*}
    1 \leq \frac{\tilde{\omega}_i}{\tilde{\omega}_j} \leq 2\frac{\omega_i}{\omega_j} - 1
\end{equation*}
for all $\omega_i > \omega_j$, $i \in \{1, \dots, n\}, j \in \{1, \dots, n\}$. Letting $\tilde{\bs{\Phi}} = \tilde{\bs{\Omega}}^{-1}$ and letting $\mathcal{V}^{n,p}$ denote the set of all $n\times p$ orthogonal matrices, see that 
\begin{equation}\label{eq:sups}
    \underset{\bl{U} \in \mathcal{V}^{n,p}}{\sup}~ \frac{|H_{\bl{U}}(\tilde{\bs{\Omega}}, \bs{\Omega})|}{|H_{\bl{U}}(\bl{I}_n, \bs{\Omega})|} = \underset{\bl{U} \in \mathcal{V}^{n,p}}{\sup}~\frac{|\bl{U}^\top \tilde{\bs{\Phi}}^2 \bs{\Omega}\bl{U}|}{|\bl{U}^\top\bs{\Omega}\bl{U}| |\bl{U}^\top \tilde{\bs{\Phi}}\bl{U}|^2}
\end{equation}
Now set $\bl{Z} := \bs{\Omega}^{1/2} \bl{U} (\bl{U}^\top \bs{\Omega} \bl{U})^{-1/2}$, where we have used the inverse of the symmetric matrix square root, so that $\bl{Z} \in \mathcal{V}^{n,p}$. Reparameterizing in terms of $\bl{Z}$, we have
\begin{equation*}
    \begin{aligned}
        \frac{|\bl{U}^\top \tilde{\bs{\Phi}}^2 \bs{\Omega}\bl{U}|}{|\bl{U}^\top\bs{\Omega}\bl{U}|} = |\bl{Z}^\top \tilde{\bs{\Phi}}^2 \bl{Z}|~\text{ and }~ \frac{|\bl{U}^\top \bl{U}|^2}{|\bl{U}^\top \tilde{\bs{\Phi}}\bl{U}|^2} = \frac{|\bl{Z}^\top \bs{\Omega}^{-1} \bl{Z}|^2}{|\bl{Z}^\top \tilde{\bs{\Phi}} \bs{\Omega}^{-1} \bl{Z}|^2}
    \end{aligned}
\end{equation*}
so \eqref{eq:sups} becomes
\begin{equation*}
    \underset{\bl{Z} \in \mathcal{V}^{n,p}}{\sup}~ \frac{|\bl{Z}^\top \tilde{\bs{\Phi}}^2 \bl{Z}||\bl{Z}^\top \bs{\Omega}^{-1} \bl{Z}|^2}{|\bl{Z}^\top \tilde{\bs{\Phi}} \bs{\Omega}^{-1}\bl{Z}|^2} .
\end{equation*}
Since the determinant is invariant to multiplication of its matrix argument by any square orthogonal matrix, we may seek a convenient orthogonal basis in $\mathbb{R}^p$ and write the above in terms of this basis. So let $\bl{Q}$ be the $p \times p$ orthogonal matrix whose columns are the eigenvectors of $\bl{Z}^\top \tilde{\bs{\Phi}} \bs{\Omega}^{-1}\bl{Z}$, and set $\bl{V} := \bl{Z} \bl{Q}$. Then
\begin{equation*}
    \begin{aligned}
        \frac{|\bl{Z}^\top \tilde{\bs{\Phi}}^2 \bl{Z}||\bl{Z}^\top \bs{\Omega}^{-1} \bl{Z}|^2}{|\bl{Z}^\top \tilde{\bs{\Phi}} \bs{\Omega}^{-1}\bl{Z}|^2} &= \frac{|\bl{Q}^\top \bl{Z}^\top \tilde{\bs{\Phi}}^2 \bl{Z} \bl{Q}||\bl{Q}^\top \bl{Z}^\top \bs{\Omega}^{-1} \bl{Z} \bl{Q}|^2}{|\bl{Q}^\top \bl{Z}^\top \tilde{\bs{\Phi}} \bs{\Omega}^{-1}\bl{Z}\bl{Q}|^2} \\
        &= \frac{|\bl{V}^\top \tilde{\bs{\Phi}}^2 \bl{V}||\bl{V}^\top \bs{\Omega}^{-1} \bl{V}|^2}{|\bl{V}^\top \tilde{\bs{\Phi}} \bs{\Omega}^{-1}\bl{V}|^2} \\
        &= \frac{|\bl{V}^\top \tilde{\bs{\Phi}}^2 \bl{V}||\bl{V}^\top \bs{\Omega}^{-1} \bl{V}|^2}{\prod_{j=1}^p (\bl{v}_j^\top \tilde{\bs{\Phi}} \bs{\Omega}^{-1} \bl{v}_j )^2}.
    \end{aligned}
\end{equation*}
In the last line, $\bl{v}_j$ denotes the $j$th column of $\bl{V}$. The product in the denominator of the last line above is therefore taken over the diagonal elements of $\bl{V}^\top \tilde{\bs{\Phi}} \bs{\Omega}^{-1}\bl{V}$, which, by construction of $\bl{Q}$, are also the eigenvalues of $\bl{Z}^\top \tilde{\bs{\Phi}} \bs{\Omega}^{-1}\bl{Z}$. Hadamard's inequality \citep{marshall_inequalities_2011} states that the determinant of a symmetric positive definite matrix is less than or equal to the product of its diagonal entries. Apply this inequality to the two determinants in the numerator of the last line above, and find that
\begin{equation*}
    \begin{aligned}
        \frac{|\bl{Z}^\top \tilde{\bs{\Phi}}^2 \bl{Z}||\bl{Z}^\top \bs{\Omega}^{-1} \bl{Z}|^2}{|\bl{Z}^\top \tilde{\bs{\Phi}} \bs{\Omega}^{-1}\bl{Z}|^2} &= \frac{|\bl{V}^\top \tilde{\bs{\Phi}}^2 \bl{V}||\bl{V}^\top \bs{\Omega}^{-1} \bl{V}|^2}{\prod_{j=1}^p (\bl{v}_j^\top \tilde{\bs{\Phi}} \bs{\Omega}^{-1} \bl{v}_j )^2} \\
        &\leq \frac{\left\{ \prod_{j=1}^p \bl{v}_j^\top \tilde{\bs{\Phi}}^2 \bl{v}_j  \right\}\left\{ \prod_{j=1}^p (\bl{v}_j^\top \bs{\Omega}^{-1} \bl{v}_j )^2 \right\}}{\prod_{j=1}^p (\bl{v}_j^\top \tilde{\bs{\Phi}} \bs{\Omega}^{-1} \bl{v}_j )^2}.
    \end{aligned}
\end{equation*}
Returning to our notation from the main text, this upper bound may be written as a product of ratios of $\mathrm{e}_{\bl{v}_j}$ functions as follows
\begin{equation*}
    \frac{\left\{ \prod_{j=1}^p \bl{v}_j^\top \tilde{\bs{\Phi}}^2 \bl{v}_j  \right\}\left\{ \prod_{j=1}^p (\bl{v}_j^\top \bs{\Omega}^{-1} \bl{v}_j )^2 \right\}}{\prod_{j=1}^p (\bl{v}_j^\top \tilde{\bs{\Phi}} \bs{\Omega}^{-1} \bl{v}_j )^2} = \prod_{j=1}^p \frac{\mathrm{e}_{\bl{v}_j}(\tilde{\bs{\Phi}}^2) \mathrm{e}_{\bl{v}_j}(\bs{\Omega}^{-1})^2}{\mathrm{e}_{\bl{v}_j}(\tilde{\bs{\Phi}} \bs{\Omega}^{-1})^2}.
\end{equation*}
So we conclude that
\begin{equation*}
\begin{aligned}
    \underset{\bl{Z} \in \mathcal{V}^{n,p}}{\sup}~\frac{|\bl{Z}^\top \tilde{\bs{\Phi}}^2 \bl{Z}||\bl{Z}^\top \bs{\Omega}^{-1} \bl{Z}|^2}{|\bl{Z}^\top \tilde{\bs{\Phi}} \bs{\Omega}^{-1}\bl{Z}|^2} &\leq \underset{\bl{V} \in \mathcal{V}^{n,p}}{\sup}~ \prod_{j=1}^p \frac{\mathrm{e}_{\bl{v}_j}(\tilde{\bs{\Phi}}^2) \mathrm{e}_{\bl{v}_j}(\bs{\Omega}^{-1})^2}{\mathrm{e}_{\bl{v}_j}(\tilde{\bs{\Phi}} \bs{\Omega}^{-1})^2} \\
    &\leq \underset{\bl{v}_j \in \mathcal{V}^n,~j=1,\dots,p}{\sup}~ \prod_{j=1}^p \frac{\mathrm{e}_{\bl{v}_j}(\tilde{\bs{\Phi}}^2) \mathrm{e}_{\bl{v}_j}(\bs{\Omega}^{-1})^2}{\mathrm{e}_{\bl{v}_j}(\tilde{\bs{\Phi}} \bs{\Omega}^{-1})^2} \\
    &= \prod_{j=1}^p \underset{\bl{v}_j \in \mathcal{V}^n}{\sup}~ \frac{\mathrm{e}_{\bl{v}_j}(\tilde{\bs{\Phi}}^2) \mathrm{e}_{\bl{v}_j}(\bs{\Omega}^{-1})^2}{\mathrm{e}_{\bl{v}_j}(\tilde{\bs{\Phi}} \bs{\Omega}^{-1})^2},
\end{aligned}
\end{equation*}
where the second inequality is due to the fact that 
\begin{equation*}
    \{ \bl{V} \in \mathbb{R}^{n \times p} : \bl{V}^\top \bl{V} = \bl{I}_p\} \subseteq \{ \bl{V} \in \mathbb{R}^{n \times p} : \mathrm{diag}(\bl{V}^\top \bl{V}) = \bl{1}_p\}.
\end{equation*}
Finally, by the reparameterization $\bl{w}_j = \bs{\Omega}^{-1/2} \bl{v}_j / \|\bs{\Omega}^{-1/2} \bl{v}_j\| = (\bs{\omega}^{-1/2} \odot \bl{v}_j) / \sqrt{\mathrm{e}_{\bl{v}_j}(\bs{\Omega}^{-1})}$ see that
\begin{equation*}
    \underset{\bl{v}_j \in \mathcal{V}^n}{\sup}~\frac{\mathrm{e}_{\bl{v}_j}(\tilde{\bs{\Phi}}^2) \mathrm{e}_{\bl{v}_j}(\bs{\Omega}^{-1})^2}{\mathrm{e}_{\bl{v}_j}(\tilde{\bs{\Phi}} \bs{\Omega}^{-1})^2} = \underset{\bl{w}_j \in \mathcal{V}^n}{\sup}~ \frac{\mathrm{e}_{\bl{w}_j}(\tilde{\bs{\Phi}}^2 \bs{\Omega})}{\mathrm{e}_{\bl{w}_j}(\bs{\Omega}) \mathrm{e}_{\bl{w}_j}(\tilde{\bs{\Phi}})^2} = \underset{\bl{w}_j \in \mathcal{V}^n}{\sup}~\frac{H_{\bl{w}_j}(\tilde{\bs{\Omega}}, \bs{\Omega})}{H_{\bl{w}_j}(\bl{I}_n, \bs{\Omega})}.
\end{equation*}
By Theorem \ref{thm:main}, the right-most term is less than or equal to 1 for all unit vectors $\bl{w}_j$. Hence, 
\begin{equation*}
    \underset{\bl{U} \in \mathcal{V}^{n,p}}{\sup}~ \frac{|H_{\bl{U}}(\tilde{\bs{\Omega}}, \bs{\Omega})|}{|H_{\bl{U}}(\bl{I}_n, \bs{\Omega})|} \leq \prod_{j=1}^p \underset{\bl{w}_j \in \mathcal{V}^n}{\sup}~ \frac{H_{\bl{w}_j}(\tilde{\bs{\Omega}}, \bs{\Omega})}{H_{\bl{w}_j}(\bl{I}_n, \bs{\Omega})} \leq 1,
\end{equation*}
which demonstrates that $\tilde{\bs{\Omega}} \in \mathcal{C}^p_{\bs{\Omega}}$.

($\impliedby$) As in the proof of Theorem \ref{thm:main}, suppose that the diagonal entries of $\tilde{\bs{\Omega}}$ do not satisfy
\begin{equation*}
    1 \leq \frac{\tilde{\omega}_i}{\tilde{\omega}_j} \leq 2\frac{\omega_i}{\omega_j} - 1
\end{equation*}
for all $\omega_i > \omega_j$, $i \in \{1, \dots, n\}, j \in \{1, \dots, n\}$. Then there exists at least one pair of indices, $(i,j)$, for which $\omega_i > \omega_j$, and either
\begin{equation*}
    \frac{\tilde{\omega}_i}{\tilde{\omega}_j} > 2\frac{\omega_i}{\omega_j} - 1
\end{equation*}
or
\begin{equation*}
    \frac{\tilde{\omega}_i}{\tilde{\omega}_j} < 1.
\end{equation*}
In the proof of Theorem \ref{thm:main}, we showed that either of the conditions above imply that it is possible to find a unit vector $\bl{u}$ with entries equal to zero everywhere except at indices $i$ and $j$ such that 
\begin{equation*}
    \frac{H_{\bl{u}}(\tilde{\bs{\Omega}}, \bs{\Omega})}{H_{\bl{u}}(\bl{I}_n, \bs{\Omega})} > 1.
\end{equation*}
Without loss of generality, suppose that $(i,j) = (1,2)$. Then the non-zero entries of $\bl{u}$ occur at indices $1$ and $2$. Again without loss of generality, let $\bl{U} \in \mathbb{R}^{n \times p}$ be a matrix which has its first column equal to $\bl{u}$, and all other $p-1$ columns equal to the $n$-dimensional standard basis vectors $\bs{e}_3, \dots, \bs{e}_{p+1}$. Then $\bl{U} \in \mathcal{V}^{n,p}$ and both $H_{\bl{U}}(\tilde{\bs{\Omega}}, \bs{\Omega})$ and $H_{\bl{U}}(\bl{I}_n, \bs{\Omega})$ are diagonal matrices. Thus,
\begin{equation*}
    \frac{|H_{\bl{U}}(\tilde{\bs{\Omega}}, \bs{\Omega})|}{|H_{\bl{U}}(\bl{I}_n, \bs{\Omega})|} = \frac{H_{\bl{u}}(\tilde{\bs{\Omega}}, \bs{\Omega}) \prod_{i=3}^{p+1} \tilde{\phi}_i^2 \omega_i }{H_{\bl{u}}(\bl{I}_n, \bs{\Omega}) (\prod_{i=3}^{p+1} \omega_i) (\prod_{i=3}^{p+1} \tilde{\phi}_i)^2} = \frac{H_{\bl{u}}(\tilde{\bs{\Omega}}, \bs{\Omega})}{H_{\bl{u}}(\bl{I}_n, \bs{\Omega})} > 1.
\end{equation*}
Since we were able to construct an orthogonal $\bl{U}$ for which $|H_{\bl{U}}(\tilde{\bs{\Omega}}, \bs{\Omega})| > |H_{\bl{U}}(\bl{I}_n, \bs{\Omega})|$, we conclude that $\tilde{\bs{\Omega}} \notin \mathcal{C}^p_{\bs{\Omega}}$.
\end{proof}

\subsection{{Proof of Corollary \ref{cor:main2}}}

\begin{proof}

For the entirety of the proof we assume that the diagonal entries of $\bs{\Omega}, \tilde{\bs{\Omega}} \in \mathcal{D}_+^n$ are distinct. The result may be generalized by a continuity argument to the case of non-distinct diagonal entries.

($\implies$) Assume that the diagonal entries of $\tilde{\bs{\Omega}}$ satisfy
\begin{equation*}
    1 \leq \frac{\tilde{\omega}_i}{\tilde{\omega}_j} \leq 2\frac{\omega_i}{\omega_j} - 1
\end{equation*}
for all $\omega_i > \omega_j$, $i \in \{1, \dots, n\}, j \in \{1, \dots, n\}$. Letting $\tilde{\bs{\Phi}} = \tilde{\bs{\Omega}}^{-1}$ and letting $\mathcal{V}^{n,p}$ denote the set of all $n\times p$ orthogonal matrices, we want to show that
\begin{equation}\label{eq:sups2}
    \underset{\bl{U} \in \mathcal{V}^{n,p}}{\sup}~\tr \left\{(\bl{U}^\top \tilde{\bs{\Phi}} \bl{U})^{-1} \bl{U}^\top \tilde{\bs{\Phi}} \bs{\Omega} \tilde{\bs{\Phi}} \bl{U} (\bl{U}^\top \tilde{\bs{\Phi}} \bl{U})^{-1} - \bl{U}^\top \bs{\Omega} \bl{U} \right\} \leq 0.
\end{equation}
By the cyclic property of the trace, we can pre- and post-multiply the matrices above by any $p \times p$ orthogonal matrix without changing the value of the trace. Let $\bl{Q} \in \mathcal{V}^{p,p}$ be the matrix whose columns are the eigenvectors of $\bl{U}^\top \tilde{\bs{\Phi}} \bl{U}$, and let $\bl{Z} = \bl{U}\bl{Q}$. Then
\begin{equation*}
    \tr \left\{(\bl{U}^\top \tilde{\bs{\Phi}} \bl{U})^{-1} \bl{U}^\top \tilde{\bs{\Phi}} \bs{\Omega} \tilde{\bs{\Phi}} \bl{U} (\bl{U}^\top \tilde{\bs{\Phi}} \bl{U})^{-1}\right\} = \tr \left\{(\bl{Z}^\top \tilde{\bs{\Phi}} \bl{Z})^{-1} \bl{Z}^\top \tilde{\bs{\Phi}} \bs{\Omega} \tilde{\bs{\Phi}} \bl{Z} (\bl{Z}^\top \tilde{\bs{\Phi}} \bl{Z})^{-1} \right\} 
\end{equation*}
and
\begin{equation*}
    \tr\left\{ \bl{U}^\top \bs{\Omega} \bl{U} \right\} = \tr\left\{ \bl{Z}^\top \bs{\Omega} \bl{Z} \right\}.
\end{equation*}
Moreover, $\bl{Z}^\top \tilde{\bs{\Phi}} \bl{Z}$ is a diagonal matrix. So
\begin{equation*}
\begin{aligned}
    \tr \left\{(\bl{Z}^\top \tilde{\bs{\Phi}} \bl{Z})^{-1} \bl{Z}^\top \tilde{\bs{\Phi}} \bs{\Omega} \tilde{\bs{\Phi}} \bl{Z} (\bl{Z}^\top \tilde{\bs{\Phi}} \bl{Z})^{-1} - \bl{Z}^\top \bs{\Omega} \bl{Z} \right\} &= \sum_{j=1}^p \frac{\bl{z}_j^\top \tilde{\bs{\Phi}} \bs{\Omega} \tilde{\bs{\Phi}}\bl{z}_j}{(\bl{z}_j^\top \tilde{\bs{\Phi}} \bl{z}_j)^2} - \sum_{j=1}^p \bl{z}_j^\top \bs{\Omega} \bl{z}_j \\
    &= \sum_{j=1}^p \left( \frac{\bl{z}_j^\top \tilde{\bs{\Phi}} \bs{\Omega} \tilde{\bs{\Phi}}\bl{z}_j}{(\bl{z}_j^\top \tilde{\bs{\Phi}} \bl{z}_j)^2} - \bl{z}_j^\top \bs{\Omega} \bl{z}_j \right).
\end{aligned}
\end{equation*}
By Theorem \ref{thm:main}, each of the summands in the last line above is less than or equal to zero. Therefore,
\begin{equation*}
\begin{aligned}
    \underset{\bl{U} \in \mathcal{V}^{n,p}}{\sup}~\tr \left\{(\bl{U}^\top \tilde{\bs{\Phi}} \bl{U})^{-1} \bl{U}^\top \tilde{\bs{\Phi}} \bs{\Omega} \tilde{\bs{\Phi}} \bl{U} (\bl{U}^\top \tilde{\bs{\Phi}} \bl{U})^{-1} - \bl{U}^\top \bs{\Omega} \bl{U} \right\} &= \underset{\bl{Z} \in \mathcal{V}^{n,p}}{\sup}~\sum_{j=1}^p \left( \frac{\bl{z}_j^\top \tilde{\bs{\Phi}} \bs{\Omega} \tilde{\bs{\Phi}}\bl{z}_j}{(\bl{z}_j^\top \tilde{\bs{\Phi}} \bl{z}_j)^2} - \bl{z}_j^\top \bs{\Omega} \bl{z}_j \right) \\
    &\leq \sum_{j=1}^p \underset{\bl{z}_j \in \mathcal{V}^{n}}{\sup}\left( \frac{\bl{z}_j^\top \tilde{\bs{\Phi}} \bs{\Omega} \tilde{\bs{\Phi}}\bl{z}_j}{(\bl{z}_j^\top \tilde{\bs{\Phi}} \bl{z}_j)^2} - \bl{z}_j^\top \bs{\Omega} \bl{z}_j \right) \\
    &\leq 0,
\end{aligned}    
\end{equation*}
which shows that $\tilde{\bs{\Omega}} \in \mathcal{K}_{\bs{\Omega}}^p$.

($\impliedby$) As in the proof of Theorem \ref{thm:main}, suppose that the diagonal entries of $\tilde{\bs{\Omega}}$ do not satisfy
\begin{equation*}
    1 \leq \frac{\tilde{\omega}_i}{\tilde{\omega}_j} \leq 2\frac{\omega_i}{\omega_j} - 1
\end{equation*}
for all $\omega_i > \omega_j$, $i \in \{1, \dots, n\}, j \in \{1, \dots, n\}$. Then there exists at least one pair of indices, $(i,j)$, for which $\omega_i > \omega_j$, and either
\begin{equation*}
    \frac{\tilde{\omega}_i}{\tilde{\omega}_j} > 2\frac{\omega_i}{\omega_j} - 1
\end{equation*}
or
\begin{equation*}
    \frac{\tilde{\omega}_i}{\tilde{\omega}_j} < 1.
\end{equation*}
In the proof of Theorem \ref{thm:main}, we showed that either of the conditions above imply that it is possible to find a unit vector $\bl{u}$ with entries equal to zero everywhere except at indices $i$ and $j$ such that 
\begin{equation*}
    \frac{H_{\bl{u}}(\tilde{\bs{\Omega}}, \bs{\Omega})}{H_{\bl{u}}(\bl{I}_n, \bs{\Omega})} > 1.
\end{equation*}
Without loss of generality, suppose that $(i,j) = (1,2)$. Then the non-zero entries of $\bl{u}$ occur at indices $1$ and $2$. Again without loss of generality, let $\bl{U} \in \mathbb{R}^{n \times p}$ be a matrix which has its first column equal to $\bl{u}$, and all other $p-1$ columns equal to the $n$-dimensional standard basis vectors $\bs{e}_3, \dots, \bs{e}_{p+1}$. Then $\bl{U} \in \mathcal{V}^{n,p}$ and both $H_{\bl{U}}(\tilde{\bs{\Omega}}, \bs{\Omega})$ and $H_{\bl{U}}(\bl{I}_n, \bs{\Omega})$ are diagonal matrices. Thus,
\begin{equation*}
    \tr \left\{H_{\bl{U}}(\tilde{\bs{\Omega}}, \bs{\Omega}) - H_{\bl{U}}(\bl{I}_n, \bs{\Omega}) \right\} = H_{\bl{u}}(\tilde{\bs{\Omega}}, \bs{\Omega}) - H_{\bl{u}}(\bl{I}_n, \bs{\Omega}) +  0  > 0.
\end{equation*}
Since we were able to construct an orthogonal $\bl{U}$ for which $\tr \left\{H_{\bl{U}}(\tilde{\bs{\Omega}}, \bs{\Omega}) - H_{\bl{U}}(\bl{I}_n, \bs{\Omega}) \right\} > 0$, we conclude that $\tilde{\bs{\Omega}} \notin \mathcal{K}^p_{\bs{\Omega}}$.
\end{proof}

\subsection{{Proof of Corollary \ref{cor:main3}}}

\begin{proof}
    Let $\bs{\Omega}, \tilde{\bs{\Omega}} \in \mathcal{S}_+^n$ be simultaneously diagonalizable with common eigenvectors $\bl{V} \in \mathcal{V}^{n,n}$, and let $\{\lambda_1, \dots, \lambda_n\}$ and $\{\tilde{\lambda}_1, \dots, \tilde{\lambda}_n\}$ be the eigenvalues of $\bs{\Omega}$ and $\tilde{\bs{\Omega}}$, respectively. Then
\begin{equation*}
    H_{\bl{U}}(\tilde{\bs{\Omega}}, \bs{\Omega}) = (\bl{U}^\top \bl{V} \tilde{\bs{\Lambda}}^{-1} \bl{V}^\top \bl{U})^{-1} \bl{U}^\top \bl{V} \tilde{\bs{\Lambda}}^{-1} \bs{\Lambda} \tilde{\bs{\Lambda}}^{-1} \bl{V}^\top \bl{U}(\bl{U}^\top \bl{V} \tilde{\bs{\Lambda}}^{-1} \bl{V}^\top \bl{U})^{-1},
\end{equation*}
and
\begin{equation*}
    H_{\bl{U}}(\bl{I}_n, \bs{\Omega}) = \bl{U}^\top \bl{V} \bs{\Lambda} \bl{V}^\top \bl{U}.
\end{equation*}
Since $\bl{V}^\top \bl{U} \in \mathcal{V}^{n, p}$, it follows that
\begin{equation*}
\underset{\bl{U} \in \mathcal{V}^{n,p}}{\sup}~\frac{|H_{\bl{U}}(\tilde{\bs{\Omega}}, \bs{\Omega})|}{|H_{\bl{U}}(\bl{I}_n, \bs{\Omega})|} = \underset{\bl{U} \in \mathcal{V}^{n,p}}{\sup}~\frac{|H_{\bl{U}}(\tilde{\bs{\Lambda}}, \bs{\Lambda})|}{|H_{\bl{U}}(\bl{I}_n, \bs{\Lambda})|}.
\end{equation*}
Hence, by Corollary \ref{cor:main1}, $\tilde{\bs{\Omega}} \in \mathcal{C}^p_{\bs{\Omega}}$ if and only if $\tilde{\bs{\Lambda}} \in \mathcal{C}^p_{\bs{\Lambda}}$ if and only if
\begin{equation}
        1 \leq \frac{\tilde{\lambda}_i}{\tilde{\lambda}_j} \leq 2\frac{\lambda_i}{\lambda_j} - 1
    \end{equation}
for all $\lambda_i \geq \lambda_j$, $i \in \{1, \dots, n\}, j \in \{1, \dots, n\}$.
\end{proof}

\subsection{{Proof of Proposition \ref{prop:consistency}}}

\begin{proof}
By assumption,
\begin{equation*}
\begin{aligned}
    \frac{1}{n} \sum_{i=1}^n (1/\hat{\omega}_i - 1/\tilde{\omega}_i)\bl{x}_i \bl{x}_i^\top &\overset{p}{\longrightarrow} \bl{0}, \\
    \frac{1}{\sqrt{n}} \sum_{i=1}^n (1/\hat{\omega}_i - 1/\tilde{\omega}_i)(y_i - \bl{x}_i^\top\bs{\beta}) \bl{x}_i &\overset{p}{\longrightarrow} \bl{0}.
\end{aligned}
\end{equation*}
Therefore, 
\begin{equation*}
    \sqrt{n}(b^{n}_{\bl{X}}(\{\hat{\omega}_i\}_{i=1}^\infty) - \bs{\beta}) \overset{d}{\longrightarrow} N_p(\bl{0}, \bl{V}^{-1} \bl{B} \bl{V}^{-1}),
\end{equation*}
where
\begin{equation*}
    \bl{V} = \lim_{n \rightarrow \infty} \frac{1}{n}\sum_{i=1}^{n} \frac{1}{\tilde{\omega}_i}\bl{x}_i \bl{x}_i^\top,~\text{ and }~\bl{B} = \lim_{n \rightarrow \infty} \frac{1}{n}\sum_{i=1}^{n} \frac{\omega_i}{\tilde{\omega}^2_i}\bl{x}_i \bl{x}_i^\top.
\end{equation*}
This implies that
\begin{equation*}
    |\Cov{ b^{n}_{\bl{X}}(\{\hat{\omega}_i\}_{i=1}^{\infty})}| \overset{p}{\longrightarrow} \lim_{n \rightarrow \infty}\left| \left( \sum_{i=1}^n \frac{1}{\tilde{\omega}_i} \bl{x}_i \bl{x}_i^\top\right)^{-1} \left( \sum_{i=1}^n \frac{\omega_i}{\tilde{\omega}^2_i} \bl{x}_i \bl{x}_i^\top\right) \left( \sum_{i=1}^n \frac{1}{\tilde{\omega}_i} \bl{x}_i \bl{x}_i^\top\right)^{-1} \right|.
\end{equation*}
If additionally
\begin{equation*}
    \mathrm{diag}(\tilde{\omega}_1, \dots, \tilde{\omega}_n) \in \mathcal{C}^p_{\mathrm{diag}({\omega}_1, \dots, \omega_n)}
\end{equation*}
for each positive integer $n \geq 2$, then by Corollary \ref{cor:main1},
\begin{equation*}
    \frac{\left| \left( \sum_{i=1}^n \frac{1}{\tilde{\omega}_i} \bl{x}_i \bl{x}_i^\top\right)^{-1} \left( \sum_{i=1}^n \frac{\omega_i}{\tilde{\omega}^2_i} \bl{x}_i \bl{x}_i^\top\right) \left( \sum_{i=1}^n \frac{1}{\tilde{\omega}_i} \bl{x}_i \bl{x}_i^\top\right)^{-1} \right|}{\left| \left( \sum_{i=1}^n \bl{x}_i \bl{x}_i^\top\right)^{-1}\left( \sum_{i=1}^n \omega_i \bl{x}_i \bl{x}_i^\top\right) \left( \sum_{i=1}^n \bl{x}_i \bl{x}_i^\top\right)^{-1}\right|} \leq 1,
\end{equation*}
for each positive integer $n \geq 2$. Therefore,
\begin{equation*}
    \frac{|\Cov{ b^{n}_{\bl{X}}(\{\hat{\omega}_i\}_{i=1}^{\infty})}|}{|\Cov{ b^n_{\bl{X}}(\{1\})}|} \overset{p}{\longrightarrow} \lim_{n \rightarrow \infty}\frac{|\Cov{ b^{n}_{\bl{X}}(\{\tilde{\omega}_i\}_{i=1}^{\infty})}|}{|\Cov{ b^n_{\bl{X}}(\{1\})}|} \leq 1.
\end{equation*}
This concludes the proof.
\end{proof}

\subsection{{Proof of Theorem \ref{thm:asymvar}}}
To prove Theorem \ref{thm:asymvar}, we first present and prove the following four technical lemmas A2 to A5.

\begin{lemma}\label{lem:intid1}
Let $c > 0$. Then
\begin{equation*}
    \int_{-\infty}^\infty \frac{1}{(c + z^2)^{2}}e^{-z^2/2} dz = \frac{\sqrt{\pi/2}}{c} + \frac{\pi(1-c)}{2c^{3/2}}e^{c/2}\mathrm{erfc}(\sqrt{c/2}),
\end{equation*}
where $\mathrm{erfc}(c) = (2/\sqrt{\pi})\int_{c}^\infty e^{-r^2} dr$ is the complementary error function.
\end{lemma}
\begin{proof}
The expectation and variance of an $\mathrm{Exp}(\lambda)$ random variable are $1/\lambda$ and $1/\lambda^2$, respectively. So
\begin{equation*}
    \int_{0}^\infty x \lambda e^{-\lambda x} dx = \frac{1}{\lambda}~\text{ and }~\int_{0}^\infty (x-1/\lambda)^2 \lambda e^{-\lambda x} dx = \frac{1}{\lambda^2}.
\end{equation*}
From this, we obtain the identity
\begin{equation*}
    \int_{0}^\infty t^2 \lambda e^{-\lambda t} dt = \frac{2}{\lambda^2} .
\end{equation*}
Apply this identity to see that 
\begin{equation*}
    \begin{aligned}
        \int_{-\infty}^\infty \frac{1}{(c + z^2)^{2}}e^{-z^2/2} dz &= \frac{1}{2} \int_{-\infty}^\infty \int_0^\infty t^2 (c+z^2)e^{-(c+z^2)t} e^{-z^2/2} dt dz \\ 
        &= \frac{1}{2} \int_0^\infty t^2 e^{-ct} \int_{-\infty}^\infty (c+z^2)e^{-t z^2} e^{-z^2/2} dz dt \\
        &= \frac{\sqrt{2\pi}}{2} \int_0^\infty (2t+1)^{-1/2}t^2 e^{-ct} \int_{-\infty}^\infty (c+z^2)\frac{\sqrt{2t+1}}{\sqrt{2\pi}} e^{-(2t+1)z^2/2} dz dt .\\
    \end{aligned}
\end{equation*}
Recognizing the interior integral as an expectation with respect to a normal density with mean $0$ and variance $1/(2t+1)$, we obtain
\begin{equation*}
    \begin{aligned}
        \int_{-\infty}^\infty \frac{1}{(c + z^2)^{2}}e^{-z^2/2} dz &= \frac{\sqrt{2\pi}}{2} \int_0^\infty (2t+1)^{-1/2}t^2 e^{-ct} \{c + (2t+1)^{-1}\} dt \\
        &= \frac{\sqrt{2\pi}}{2} \left\{ c\int_0^\infty (2t+1)^{-1/2}t^2 e^{-ct} dt + \int_0^\infty (2t+1)^{-3/2}t^2 e^{-ct} dt\right\}.
    \end{aligned}
\end{equation*}
For both integrals above, substitute $r = \sqrt{2t+1}$ so that $dr = (2t+1)^{-1/2} dt$. Then
\begin{equation*}
\begin{aligned}
    \int_0^\infty (2t+1)^{-1/2}t^2 e^{-ct} dt &= \int_1^\infty \{(r^2-1)/2\}^2 e^{-c (r^2-1)/2} dr, 
\end{aligned}
\end{equation*}
and
\begin{equation*}
    \int_0^\infty (2t+1)^{-3/2}t^2 e^{-ct} dt = \int_1^\infty r^{-2} \{(r^2-1)/2\}^2 e^{-c(r^2-1)/2} dr.
\end{equation*}
This yields
\begin{equation*}
    \begin{aligned}
        \int_{-\infty}^\infty \frac{1}{(c + z^2)^{2}}e^{-z^2/2} dz &= \frac{\sqrt{2\pi} e^{c/2}}{8}\int_1^\infty (r^{-2}+c) (r^2-1)^2 e^{-cr^2/2} dr \\
        &= \frac{\sqrt{2\pi} e^{c/2}}{8 c^{3/2}}\int_{\sqrt{c}}^\infty (r^{-2}+1) (r^2-c)^2 e^{-r^2/2} dr \\
        &= \frac{\sqrt{\pi} e^{c/2}}{4 c^{3/2}}\int_{\sqrt{c/2}}^\infty \{4 r^4 + 2(1- 2c)r^2+ (c^2 - 2c) + c^2/(2r^2)\} e^{-r^2} dr ,
    \end{aligned}
\end{equation*}
where in the second and third lines we rescaled the variable of integration without changing its symbol. Evaluating each of the summands above using integration by parts yields
\begin{equation*}
\begin{aligned}
        &\int_{\sqrt{c/2}}^\infty 4 r^4 e^{-r^2} dr = \frac{3}{2}\sqrt{\pi} \mathrm{erfc}(\sqrt{c/2}) + \frac{\sqrt{c}e^{-c/2}(c+3)}{\sqrt{2}},\\
        &\int_{\sqrt{c/2}}^\infty 2(1-2c) r^2 e^{-r^2} dr = \frac{1-2c}{2}\left\{\sqrt{\pi} \mathrm{erfc}(\sqrt{c/2}) + \sqrt{2c} e^{-c/2} \right\},\\
        &\int_{\sqrt{c/2}}^\infty (c^2 - 2c) e^{-r^2} dr = \frac{\sqrt{\pi}(c^2-2c)}{2} \mathrm{erfc}(\sqrt{c/2}),\\
        &\int_{\sqrt{c/2}}^\infty \{c^2/(2r^2)\} e^{-r^2} dr = c^2 \left\{-\frac{\sqrt{\pi}}{2} \mathrm{erfc}(\sqrt{c/2}) + \frac{e^{-c/2}}{\sqrt{2c}}\right\}.
\end{aligned}
\end{equation*}
    
From the display above, the sum of the constants involving the complementary error function is
\begin{equation*}
    \begin{aligned}
        \frac{3}{2}\sqrt{\pi} +\frac{1-2c}{2}\sqrt{\pi} + \frac{c^2-2c}{2}\sqrt{\pi} - (c^2/2)\sqrt{\pi} = 2\sqrt{\pi}\left\{1-c\right\}.
    \end{aligned}
\end{equation*}
The sum of the constants in front of $e^{-c/2}$ is
\begin{equation*}
    \frac{\sqrt{c}(c+3)}{\sqrt{2}} + \frac{\sqrt{c}(1-2c)}{\sqrt{2}} + \frac{c^{3/2}}{\sqrt{2}} = 4\sqrt{\frac{c}{2}}.
\end{equation*}
Putting everything together, we have
\begin{equation*}
\begin{aligned}
    \int_{-\infty}^\infty \frac{1}{(c + z^2)^{2}}e^{-z^2/2} dz &= \frac{\sqrt{\pi/2}}{c} + \frac{\pi(1-c)}{2c^{3/2}}e^{c/2}\mathrm{erfc}(\sqrt{c/2}),
\end{aligned}
\end{equation*}
which concludes the proof.
\end{proof}

\begin{lemma}\label{lem:intid2}
Let $c > 0$. Then
\begin{equation*}
\int_{-\infty}^\infty \frac{z^2}{(c + z^2)^{2}}e^{-z^2/2} dz = \frac{1}{2}\left\{\pi(\sqrt{c} + 1/\sqrt{c}) e^{c/2} \mathrm{erfc}(\sqrt{c/2}) -\sqrt{2\pi} \right\}.
\end{equation*}
\end{lemma}
\begin{proof}
The expectation and variance of an $\mathrm{Exp}(\lambda)$ random variable are $1/\lambda$ and $1/\lambda^2$, respectively. So
\begin{equation*}
    \int_{0}^\infty x \lambda e^{-\lambda x} dx = \frac{1}{\lambda}~\text{ and }~\int_{0}^\infty (x-1/\lambda)^2 \lambda e^{-\lambda x} dx = \frac{1}{\lambda^2}.
\end{equation*}
From this, we obtain the identity
\begin{equation*}
    \int_{0}^\infty t^2 \lambda e^{-\lambda t} dt = \frac{2}{\lambda^2} .
\end{equation*}
Apply this identity to see that 
\begin{equation*}
    \begin{aligned}
        \int_{-\infty}^\infty \frac{z^2}{(c + z^2)^{2}}e^{-z^2/2} dz &= \frac{1}{2} \int_{-\infty}^\infty \int_0^\infty t^2 (c+z^2)e^{-(c+z^2)t} z^2 e^{-z^2/2} dt dz \\ 
        &= \frac{1}{2} \int_0^\infty t^2 e^{-ct} \int_{-\infty}^\infty z^2(c+z^2)e^{-t z^2} e^{-z^2/2} dz dt \\
        &= \frac{\sqrt{2\pi}}{2} \int_0^\infty (2t+1)^{-1/2}t^2 e^{-ct} \int_{-\infty}^\infty (cz^2+z^4)\frac{\sqrt{2t+1}}{\sqrt{2\pi}} e^{-(2t+1)z^2/2} dz dt .\\
    \end{aligned}
\end{equation*}
Recognizing the interior integral as an expectation with respect to a normal density with mean $0$ and variance $1/(2t+1)$, we obtain
\begin{equation*}
    \begin{aligned}
        \int_{-\infty}^\infty \frac{1}{(c + z^2)^{2}}e^{-z^2/2} dz &= \frac{\sqrt{2\pi}}{2} \int_0^\infty (2t+1)^{-1/2}t^2 e^{-ct} \{c(2t+1)^{-1} + 3(2t+1)^{-2}\} dt \\
        &= \frac{\sqrt{2\pi}}{2} \left\{ c\int_0^\infty (2t+1)^{-3/2}t^2 e^{-ct} dt + 3\int_0^\infty (2t+1)^{-5/2}t^2 e^{-ct} dt\right\}.
    \end{aligned}
\end{equation*}
For both integrals above, substitute $r = \sqrt{2t+1}$ so that $dr = (2t+1)^{-1/2} dt$. Then
\begin{equation*}
    \int_0^\infty (2t+1)^{-3/2}t^2 e^{-ct} dt = \int_1^\infty r^{-2} \{(r^2-1)/2\}^2 e^{-c(r^2-1)/2} dr.
\end{equation*}
and
\begin{equation*}
    \int_0^\infty (2t+1)^{-5/2}t^2 e^{-ct} dt = \int_1^\infty r^{-4} \{(r^2-1)/2\}^2 e^{-c(r^2-1)/2} dr.
\end{equation*}
This yields
\begin{equation*}
    \begin{aligned}
        \int_{-\infty}^\infty \frac{1}{(c + z^2)^{2}}e^{-z^2/2} dz &= \frac{\sqrt{2\pi} e^{c/2}}{8}\int_1^\infty r^{-2} (3r^{-2}+c) (r^2-1)^2 e^{-cr^2/2} dr \\
        &= \frac{\sqrt{2\pi} e^{c/2}}{8 \sqrt{c}}\int_{\sqrt{c}}^\infty r^{-2}(3r^{-2}+1) (r^2-c)^2 e^{-r^2/2} dr \\
        &= \frac{\sqrt{\pi} e^{c/2}}{4 \sqrt{c}}\int_{\sqrt{c/2}}^\infty \{2r^2 + (c^2-6c)/(2r^2) + (3-2c) + 3c^2/(4r^4)\} e^{-r^2} dr \\
    \end{aligned}
\end{equation*}
where in the second and third lines we rescaled the variable of integration without changing its symbol. Evaluating each of the summands above using integration by parts yields
\begin{equation*}
    \begin{aligned}
        &\int_{\sqrt{c/2}}^\infty 2 r^2 e^{-r^2} dr = \frac{1}{2}\left\{\sqrt{\pi} \mathrm{erfc}(\sqrt{c/2}) + \sqrt{2c} e^{-c/2} \right\},\\
        &\int_{\sqrt{c/2}}^\infty (c^2-6c)/(2r^2) e^{-r^2} dr = (c^2-6c)\left\{-\frac{\sqrt{\pi}}{2} \mathrm{erfc}(\sqrt{c/2}) + \frac{e^{-c/2}}{\sqrt{2c}}\right\},\\
        &\int_{\sqrt{c/2}}^\infty (3 - 2c) e^{-r^2} dr = \frac{\sqrt{\pi}(3-2c)}{2} \mathrm{erfc}(\sqrt{c/2}),\\
        &\int_{\sqrt{c/2}}^\infty \{3c^2/(4r^4)\} e^{-r^2} dr = \frac{c^2}{2} \left\{\sqrt{\pi} \mathrm{erfc}(\sqrt{c/2}) - \frac{\sqrt{2}(c-1)e^{-c/2}}{c^{3/2}}\right\}.
    \end{aligned}
\end{equation*} 

From the display above, the sum of the constants involving the complementary error function is
\begin{equation*}
    \begin{aligned}
        \frac{1}{2}\sqrt{\pi} - \frac{c^2-6c}{2}\sqrt{\pi} +\frac{3-2c}{2}\sqrt{\pi} + (c^2/2)\sqrt{\pi} = 2\sqrt{\pi}\left\{1+c\right\}.
    \end{aligned}
\end{equation*}
The sum of the constants in front of $e^{-c/2}$ is
\begin{equation*}
    \frac{\sqrt{2c}}{2} + \frac{c^2-6c}{\sqrt{2c}} - \frac{\sqrt{2c}(c-1)}{2} = -\frac{4\sqrt{c}}{\sqrt{2}}.
\end{equation*}
Putting everything together, we have
\begin{equation*}
\begin{aligned}
    \int_{-\infty}^\infty \frac{z^2}{(c + z^2)^{2}}e^{-z^2/2} dz &=\frac{\pi(1+c) e^{c/2}}{2\sqrt{c}} \mathrm{erfc}(\sqrt{c/2}) -\sqrt{\pi/2},
\end{aligned}
\end{equation*}
which concludes the proof.
\end{proof}

\begin{lemma}\label{lem:tnorm}
    Let $z$ be a truncated standard normal random variable taking values on $(a, \infty)$ with probability density function
    \begin{equation*}
        p(z) = \frac{1}{\sqrt{2 \pi}} e^{-z^2/2} \bigg/ \left( 1 - \int_{-\infty}^a \frac{1}{\sqrt{2 \pi}} e^{-z^2/2}dz\right),
    \end{equation*}
    for some $a > 0$.
    Then the expectation of $z$ is bounded below and above as follows
    \begin{equation*}
        \frac{a(a^2+3)}{a^2+2} \leq \Exp{z} \leq a + 1/a.
    \end{equation*}
\end{lemma}

\begin{proof}
From the definition 
\begin{equation*}
    \Exp{z} = \int_a^\infty p(z) dz = \int_a^\infty \frac{1}{\sqrt{2 \pi}} z e^{-z^2/2}dz \bigg/ \left( 1 - \int_{-\infty}^a \frac{1}{\sqrt{2 \pi}} e^{-z^2/2}dz\right).
\end{equation*}
Due to the fact that $z/a \geq 1$ on the domain of integration, it follows that
\begin{equation*}
\begin{aligned}
    \Exp{z} &= \int_a^\infty \frac{1}{\sqrt{2 \pi}} z e^{-z^2/2}dz \bigg/ \left( 1 - \int_{-\infty}^a \frac{1}{\sqrt{2 \pi}} e^{-z^2/2}dz\right) \\
    &= a + \frac{\int_a^\infty \frac{1}{\sqrt{2 \pi}} (z-a) e^{-z^2/2}dz}{\int_{a}^\infty \frac{1}{\sqrt{2 \pi}} e^{-z^2/2}dz} \\
    &\leq a+ \frac{\int_a^\infty \frac{1}{\sqrt{2 \pi}} (z^2/a-z) e^{-z^2/2}dz}{\int_{a}^\infty \frac{1}{\sqrt{2 \pi}} e^{-z^2/2}dz} \\
    &= a + \frac{1}{a}\frac{\int_a^\infty \frac{1}{\sqrt{2 \pi}} z^2 e^{-z^2/2}dz}{\int_{a}^\infty \frac{1}{\sqrt{2 \pi}} e^{-z^2/2}dz} - \frac{\int_a^\infty \frac{1}{\sqrt{2 \pi}} z e^{-z^2/2}dz}{\int_{a}^\infty \frac{1}{\sqrt{2 \pi}} e^{-z^2/2}dz} \\
    &= a + \frac{1}{a}\left(1 + a \frac{e^{-a^2/2} / \sqrt{2\pi}}{\int_{a}^\infty \frac{1}{\sqrt{2 \pi}} e^{-z^2/2}dz}\right) - \frac{e^{-a^2/2} / \sqrt{2\pi}}{\int_{a}^\infty \frac{1}{\sqrt{2 \pi}} e^{-z^2/2}dz}\\
    &= a + 1/a,
\end{aligned}
\end{equation*}
where in the fifth line we applied the formulae for the first and second moments of a truncated standard normal random variable \citep{johnson_continuous_1994}.

For the lower bound, we apply an upper bound on the so-called $Q$-function, defined as
\begin{equation*}
    Q(x) = \frac{1}{\sqrt{2}} \mathrm{erfc}(x / \sqrt{2}) = \frac{1}{\sqrt{2\pi}} \int_{x}^\infty e^{-z^2/2} dz.
\end{equation*}
This bound, developed by \cite{peric_class_2019}, states that
\begin{equation}\label{eq:qfun}
    Q(x) \leq \frac{1}{\sqrt{2\pi}} \frac{x^2+2}{x(x^2+3)} e^{-x^2/2},~x >0.
\end{equation}

By rearranging terms, see that
\begin{equation*}
    \Exp{z} = \frac{e^{-a^2/2} / \sqrt{2\pi}}{Q(a)},
\end{equation*}
so that applying \eqref{eq:qfun} yields
\begin{equation*}
    \Exp{z} \geq \frac{a(a^2+3)}{a^2+2},
\end{equation*}
and the proof is done.
\end{proof}

\begin{lemma}\label{lem:hfun}
Define the function
\begin{equation*}
    h(x) := \sqrt{2\pi} x e^{x^2/2} \int_{-\infty}^{-x} \frac{1}{\sqrt{2 \pi}} e^{-z^2/2}dz.
\end{equation*}
The following properties hold
\begin{enumerate}
    \item $h(x)$ is non-decreasing for $x>0$.
    \item $x^2 / (x^2 +1) \leq h(x) \leq (x^2 + 2) / (x^2 + 3)$ for $x > 0$.
    \item $h(x)$ is concave for $x > 0$.
    \item $x^2\{1 - h(x)\}$ is non-decreasing for $x>0$.
\end{enumerate}
\end{lemma}

\begin{proof}
To prove the first property, note that
\begin{equation*}
\begin{aligned}
    \frac{d}{dx} \log h(x) &= x + 1/x - \frac{e^{-x^2/2}/\sqrt{2\pi}}{\int_{-\infty}^{-x} \frac{1}{\sqrt{2 \pi}} e^{-z^2/2}dz} \\
    &= x + 1/x - \frac{e^{-x^2/2}/\sqrt{2\pi}}{\int_{x}^{\infty} \frac{1}{\sqrt{2 \pi}} e^{-z^2/2}dz} \\
    &= x + 1/x - x/h(x),
 \end{aligned}   
\end{equation*}
where we used the symmetry of the standard normal probability density function to arrive at the second line. By Lemma \ref{lem:tnorm}, the expression in the second line is non-negative for all $x > 0$. Hence, $\log h(x)$ is non-decreasing for $x >0$, which implies that $h(x)$ is non-decreasing for $x > 0$.

For the second property, both the inequality $\frac{d}{dx} \log h(x) \geq 0,~\forall x>0$, and the inequality
\begin{equation*}
\begin{aligned}
    \frac{d}{dx} h(x) &= \left\{\frac{d}{dx} \log h(x)\right\} h(x) \\
    &= (x+1/x)h(x) - x \\
    &\geq 0
\end{aligned}
\end{equation*}
for $x > 0$ yield the lower bound
\begin{equation*}
    \begin{aligned}
        h(x) \geq \frac{x^2}{x^2+1},~\forall x> 0
    \end{aligned}
\end{equation*}
upon rearrangement. A direct application of the \cite{peric_class_2019} inequality (see the proof of Lemma \ref{lem:tnorm}) then yields the upper bound
\begin{equation*}
    h(x) \leq \frac{x^2+2}{x^2+3},~x>0.
\end{equation*}

To see that $h(x)$ is concave, take its second derivative to obtain
\begin{equation*}
    \begin{aligned}
        \frac{d^2}{dx^2} h(x) &= \frac{d}{dx} \left\{ (x+1/x)h(x) - x \right\} \\
        &= h(x) + x\{(x + 1/x)h(x) - x\} - h(x)/x^2 + (1 + 1/x^2)h(x) - 2 \\
        &= (x^2+3) h(x) - (x^2  + 2). \\
    \end{aligned}
\end{equation*}
By the previously established upper bound on $h(x)$, we conclude that
\begin{equation*}
    \frac{d^2}{dx^2} h(x) \leq 0,~\forall x>0,
\end{equation*}
so $h(x)$ is concave for $x > 0$.

Finally, note that
\begin{equation*}
    \begin{aligned}
        \frac{d}{dx} [x^2\{1 - h(x)\}] &= 2x - 2x h(x) - x^2\{(x + 1/x)h(x) - x\} \\
        &= x\left\{(x^2 + 2) - (x^2 + 3)h(x)\right\},
    \end{aligned}
\end{equation*}
For all $x > 0$ the sign of the expression above is equal to the sign of $-\frac{d^2}{dx^2} h(x)$. By the concavity of $h(x)$, we conclude that $x^2\{1 - h(x)\}$ is non-decreasing for $x > 0$.
\end{proof}

Now we prove Theorem \ref{thm:asymvar}.
\begin{proof}[Proof of Theorem \ref{thm:asymvar}]
The estimate $\bs{\beta}^*$ maximizes the marginal likelihood of \eqref{eq:tmod}. The logarithm of the marginal likelihood for $y_i$ under \eqref{eq:tmod} is
\begin{equation*}
    \ell(y_i ; \bs{\beta}) = \log \frac{\Gamma\left(\frac{\nu +1 }{2}\right)}{\Gamma(\nu/2)\sqrt{\pi \nu}} - \frac{1}{2}\log \omega_0 - \frac{\nu+1}{2} \log\left(1 + \frac{(y_i - \bl{x}_i^\top \bs{\beta})^2}{\nu \omega_0} \right).
\end{equation*}
Therefore
\begin{equation*}
    \nabla \ell(y_i ; \bs{\beta}) = (\nu+1) \frac{(y_i - \bl{x}_i^\top \bs{\beta})}{\nu \omega_0 + (y_i - \bl{x}_i^\top \bs{\beta})^2 } \bl{x}_i,
\end{equation*}
and
\begin{equation*}
    \nabla \ell(y_i ; \bs{\beta}) \nabla \ell(y_i ; \bs{\beta})^\top = (\nu+1)^2 \frac{(y_i - \bl{x}_i^\top \bs{\beta})^2}{\{\nu \omega_0 + (y_i - \bl{x}_i^\top \bs{\beta})^2\}^2 } \bl{x}_i \bl{x}_i^\top.
\end{equation*}
For the second-order derivatives, we obtain
\begin{equation*}
\begin{aligned}
    \nabla^2 \ell(y_i ; \bs{\beta}) &= 2 (\nu+1) \frac{(y_i - \bl{x}_i^\top \bs{\beta})^2}{\{\nu \omega_0 + (y_i - \bl{x}_i^\top \bs{\beta})^2\}^2 } \bl{x}_i \bl{x}_i^\top - (\nu+1) \frac{1}{\nu \omega_0 + (y_i - \bl{x}_i^\top \bs{\beta})^2 } \bl{x}_i \bl{x}_i^\top\\
    &=(\nu+1) \frac{(y_i - \bl{x}_i^\top \bs{\beta})^2 - \nu \omega_0}{\{\nu \omega_0 + (y_i - \bl{x}_i^\top \bs{\beta})^2\}^2} \bl{x}_i \bl{x}_i^\top.
\end{aligned}
\end{equation*}
Following \cite{stefanski_calculus_2002}, we evaluate the expectations
\begin{equation*}
\begin{aligned}
    \Exp{  \nabla \ell(y_i ; \bs{\beta}) \nabla \ell(y_i ; \bs{\beta})^\top}~\text{ and }~\Exp{ \nabla^2 \ell(y_i ; \bs{\beta}) }
\end{aligned}
\end{equation*}
with respect to \eqref{eq:smod}. This yields
\begin{equation}
    \begin{aligned}
        \Exp{  \nabla \ell(y_i ; \bs{\beta}) \nabla \ell(y_i ; \bs{\beta})^\top} = \frac{(\nu+1)^2}{\sqrt{2 \pi \omega_i}} \left(\int_{-\infty}^\infty \frac{(y_i - \bl{x}_i^\top \bs{\beta})^2}{\{\nu \omega_0 + (y_i - \bl{x}_i^\top \bs{\beta})^2\}^2 } e^{-\frac{(y_i - \bl{x}_i \bs{\beta})^2}{2\omega_i}} dy_i \right)\bl{x}_i \bl{x}_i^\top,
    \end{aligned}
\end{equation}
and
\begin{equation}\label{eq:int0}
    \begin{aligned}
        \Exp{ \nabla^2 \ell(y_i ; \bs{\beta}) } = \frac{\nu+1}{\sqrt{2\pi \omega_i}} \left(\int_{-\infty}^\infty \frac{(y_i - \bl{x}_i^\top \bs{\beta})^2 - \nu \omega_0}{\{\nu \omega_0 + (y_i - \bl{x}_i^\top \bs{\beta})^2\}^2} e^{-\frac{(y_i - \bl{x}_i \bs{\beta})^2}{2\omega_i}} dy_i \right) \bl{x}_i \bl{x}_i^\top .
    \end{aligned}
\end{equation}

Making the substitution $z_i = (y_i - \bl{x}_i \bs{\beta})/\sqrt{\omega_i}$ in each of the integrals above, we obtain the more compact expressions
\begin{equation}\label{eq:int1}
    \begin{aligned}
        \Exp{  \nabla \ell(y_i ; \bs{\beta}) \nabla \ell(y_i ; \bs{\beta})^\top} = \frac{(\nu+1)^2}{\sqrt{2 \pi}\omega_i} \left(\int_{-\infty}^\infty \frac{z_i^2}{\{\nu \omega_0/\omega_i + z_i^2\}^2 } e^{-z_i^2 / 2} dz_i \right)\bl{x}_i \bl{x}_i^\top,
    \end{aligned}
\end{equation}
and
\begin{equation}\label{eq:int2}
    \begin{aligned}
        \Exp{ \nabla^2 \ell(y_i ; \bs{\beta}) } = \frac{\nu+1}{\sqrt{2\pi} \omega_i} \left(\int_{-\infty}^\infty \frac{z_i^2 - \nu \omega_0 /\omega_i}{\{\nu \omega_0 / \omega_i + z_i^2\}^2} e^{-z_i^2/2} dz_i \right) \bl{x}_i \bl{x}_i^\top .
    \end{aligned}
\end{equation}

Starting with \eqref{eq:int1}, apply Lemma \ref{lem:intid2} to find
\begin{equation*}
\begin{aligned}
    \int_{-\infty}^\infty \frac{z_i^2}{\{\nu \omega_0/\omega_i + z_i^2\}^2 } e^{-z_i^2 / 2} dz_i &= \frac{1}{2}\left\{\pi \left(\sqrt{\frac{\nu \omega_0}{\omega_i}} + \sqrt{\frac{\omega_i}{\nu \omega_0}} \right) e^{\frac{\nu \omega_0}{2\omega_i}} \mathrm{erfc}\left(\sqrt{\frac{\nu \omega_0}{2\omega_i}} \right) -\sqrt{2\pi} \right\}.
\end{aligned}
\end{equation*}
Therefore
\begin{equation*}
\begin{aligned}
    \Exp{  \nabla \ell(y_i ; \bs{\beta}) \nabla \ell(y_i ; \bs{\beta})^\top} &= \frac{(\nu+1)^2}{2\omega_i}\left\{\sqrt{\frac{\pi}{2}} \left(\sqrt{\frac{\nu \omega_0}{\omega_i}} + \sqrt{\frac{\omega_i}{\nu \omega_0}} \right) e^{\frac{\nu \omega_0}{2\omega_i}} \mathrm{erfc}\left(\sqrt{\frac{\nu \omega_0}{2\omega_i}} \right) - 1 \right\} \\
    &= \frac{(\nu+1)^2}{2\omega_i}\left\{ \left(\sqrt{\frac{\nu \omega_0}{\omega_i}} + \sqrt{\frac{\omega_i}{\nu \omega_0}} \right) e^{\frac{\nu \omega_0}{2\omega_i}} \left(\int_{-\infty}^{-\sqrt{\frac{\nu \omega_0}{\omega_i}}} e^{-z^2/2} dz\right) - 1 \right\},
\end{aligned}
\end{equation*}
where in the second line we used the identity $\mathrm{erfc}(-x/\sqrt{2}) / 2 = (2\pi)^{-1/2} \int_{-\infty}^x e^{-z^2/2} dz$.

Next, evaluating \eqref{eq:int2} in light of Lemmas \ref{lem:intid1} and \ref{lem:intid2} yields
\begin{equation*}
    \begin{aligned}
        \int_{-\infty}^\infty \frac{z_i^2 - \nu \omega_0 /\omega_i}{\{\nu \omega_0 / \omega_i + z_i^2\}^2} e^{-z_i^2/2} dz_i &= \frac{1}{2}\left\{\pi \left(\sqrt{\frac{\nu \omega_0}{\omega_i}} + \sqrt{\frac{\omega_i}{\nu \omega_0}} \right) e^{\frac{\nu \omega_0}{2\omega_i}} \mathrm{erfc}\left(\sqrt{\frac{\nu \omega_0}{2\omega_i}} \right) -\sqrt{2\pi} \right\} - \\
        &\quad~\frac{1}{2}\left\{\sqrt{2\pi} + \pi\left(\sqrt{\frac{\omega_i}{\nu \omega_0}}-\sqrt{\frac{\nu \omega_0}{\omega_i}}\right) e^{\frac{\nu \omega_0}{2\omega_i}} \mathrm{erfc}\left(\sqrt{\frac{\nu \omega_0}{2\omega_i}} \right)\right\} \\
        &= \pi \sqrt{\frac{\nu \omega_0}{\omega_i}} e^{\frac{\nu \omega_0}{2\omega_i}} \mathrm{erfc}\left(\sqrt{\frac{\nu \omega_0}{2\omega_i}} \right) - \sqrt{2\pi}.
    \end{aligned}
\end{equation*}
Therefore,
\begin{equation*}
    \begin{aligned}
        -\Exp{ \nabla^2 \ell(y_i ; \bs{\beta}) } &= -\frac{\nu+1}{\sqrt{2\pi} \omega_i} \left\{ \pi \sqrt{\frac{\nu \omega_0}{\omega_i}} e^{\frac{\nu \omega_0}{2\omega_i}} \mathrm{erfc}\left(\sqrt{\frac{\nu \omega_0}{2\omega_i}} \right) - \sqrt{2\pi} \right\} \\
        &= -\frac{\nu+1}{\omega_i} \left\{ \sqrt{\frac{\pi}{2}} \sqrt{\frac{\nu \omega_0}{\omega_i}} e^{\frac{\nu \omega_0}{2\omega_i}} \mathrm{erfc}\left(\sqrt{\frac{\nu \omega_0}{2\omega_i}} \right) - 1 \right\} \\
        &= \frac{\nu+1}{\omega_i} \left\{1-\sqrt{\frac{\nu \omega_0}{\omega_i}} e^{\frac{\nu \omega_0}{2\omega_i}} \left(\int_{-\infty}^{-\sqrt{\frac{\nu \omega_0}{\omega_i}}} e^{-z^2/2} dz \right)\right\}.
    \end{aligned}
\end{equation*}

We have established the form of the asymptotic variance of $\bs{\beta}^*$. It remains to show that the function
\begin{equation*}
    g_{\omega_0, \nu}(\omega) := \omega \left\{1-\sqrt{\frac{\nu \omega_0}{\omega}} e^{\frac{\nu \omega_0}{2\omega}} \left( \int_{-\infty}^{-\sqrt{\frac{\nu \omega_0}{\omega}}} e^{-z^2/2} dz\right) \right\}^{-1}
\end{equation*}
is non-decreasing in $\omega$, and that $g_{\omega_0, \nu}(\omega)/\omega$ is non-increasing in $\omega$. To see that the latter is true, use the function $h$ defined in Lemma \ref{lem:hfun}. Since $\sqrt{\nu \omega_0 / \omega}$ is non-increasing in $\omega$, and since $h$ is non-decreasing in its argument, this implies that 
\begin{equation*}
    h(\sqrt{\nu \omega_0 / \omega}) = \sqrt{\frac{\nu \omega_0}{\omega}} e^{\frac{\nu \omega_0}{2\omega}} \left( \int_{-\infty}^{-\sqrt{\frac{\nu \omega_0}{\omega}}} e^{-z^2/2} dz\right)
\end{equation*}
is non-increasing in $\omega$, which in turn implies that
\begin{equation*}
    g_{\omega_0, \nu}(\omega)/\omega = \{1 - h(\sqrt{\nu \omega_0 / \omega})\}^{-1}
\end{equation*}
is non-increasing as well. Next, from the last property of Lemma \ref{lem:hfun}, $x^2\{1 - h(x)\}$ is non-decreasing in $x$ for $x > 0$. Hence,
\begin{equation*}
    g_{\omega_0, \nu}(\omega) = \nu \omega_0 (\sqrt{\nu \omega_0/\omega})^{-2} \left\{1 - h(\sqrt{\nu \omega_0/\omega})\right\}^{-1},
\end{equation*}
is non-decreasing in $\omega$ for $\omega > 0$.

\end{proof}

\subsection{{Proof of Corollary \ref{cor:wcbound}}}

\begin{proof}
    As a function of $\omega$, the function
    \begin{equation*}
        C(\nu, \omega_0, \omega) =\frac{g_{\omega_0, \nu}(\omega)^2}{\omega f_{\omega_0, \nu}(\omega)}
    \end{equation*}
    is monotone increasing. Therefore, the limiting covariance matrix of $\sqrt{n}\hat{\bs{\beta}}_T$
    \begin{equation*}
        \left\{ \lim_{n \rightarrow \infty} \frac{1}{n} \sum_{i=1}^n \frac{1}{g_{\omega_0, \nu}(\omega_i)} \bl{x}_i \bl{x}_i^\top \right\}^{-1} \left\{ \lim_{n \rightarrow \infty} \frac{1}{n} \sum_{i=1}^n \frac{1}{f_{\omega_0, \nu}(\omega_i)} \bl{x}_i \bl{x}_i^\top \right\} \left\{ \lim_{n \rightarrow \infty} \frac{1}{n} \sum_{i=1}^n \frac{1}{g_{\omega_0, \nu}(\omega_i)} \bl{x}_i \bl{x}_i^\top \right\}^{-1}
    \end{equation*}
    is bounded above in the Loewner order by
    \small{\begin{equation*}
        C(\nu, \omega_0, \omega_{\max})\left\{ \lim_{n \rightarrow \infty} \frac{1}{n} \sum_{i=1}^n \frac{1}{g_{\omega_0, \nu}(\omega_i)} \bl{x}_i \bl{x}_i^\top \right\}^{-1} \left\{ \lim_{n \rightarrow \infty} \frac{1}{n} \sum_{i=1}^n \frac{\omega_i}{g_{\omega_0, \nu}(\omega_i)^2} \bl{x}_i \bl{x}_i^\top \right\} \left\{ \lim_{n \rightarrow \infty} \frac{1}{n} \sum_{i=1}^n \frac{1}{g_{\omega_0, \nu}(\omega_i)} \bl{x}_i \bl{x}_i^\top \right\}^{-1}.
    \end{equation*}}
    Since $g_{\omega_0, \nu}$ satisfies GRM, the weights $\{g_{\omega_0, \nu}(\omega_i)\}_{i=1}^\infty$ are subscedastic, and the trace and determinant of the matrix above are bounded by the trace and determinant of
    \begin{equation*}
        C(\nu, \omega_0, \omega_{\max})\left\{ \lim_{n \rightarrow \infty} \frac{1}{n} \sum_{i=1}^n \bl{x}_i \bl{x}_i^\top \right\}^{-1} \left\{ \lim_{n \rightarrow \infty} \frac{1}{n} \sum_{i=1}^n \omega_i \bl{x}_i \bl{x}_i^\top \right\} \left\{ \lim_{n \rightarrow \infty} \frac{1}{n} \sum_{i=1}^n \bl{x}_i \bl{x}_i^\top \right\}^{-1},
    \end{equation*}
    which is proportional to the limiting covariance of $\sqrt{n}\hat{\bs{\beta}}_{\mathrm{OLS}}$.

\end{proof}

\subsection{{Proof of Proposition \ref{prop:huber}}}

\begin{proof}
Using the formulae in \cite{huber_robust_1964}, the asymptotic variance of $\sqrt{n} \bs{\beta}_{H}^*$ in the model \eqref{eq:smod} with normally distributed errors is $\bl{V}^{-1} \bl{B} \bl{V}^{-1}$, where
\begin{equation*}
    \begin{aligned}
        \bl{B} &= \lim_{n \rightarrow \infty} \frac{1}{n} \sum_{i=1}^n \frac{1}{f_{k}(\omega_i)} \bl{x}_i \bl{x}_i^\top \\
        \bl{V} &= \lim_{n \rightarrow \infty} \frac{1}{n} \sum_{i=1}^n \frac{1}{g_{k}(\omega_i)} \bl{x}_i \bl{x}_i^\top,
    \end{aligned}
\end{equation*}
and
\begin{equation*}
    \begin{aligned}
        f_k(\omega) &= \left\{\frac{1}{\sqrt{2\pi\omega}}\int_{-k}^k z^2 e^{-z^2/2\omega} dz + k^2 \int_{-\infty}^{-k} e^{-z^2/2\omega} dz + k^2 \int_{k}^\infty z^2 e^{-z^2/2\omega} dz\right\}^{-1}\\
        g_k(\omega) &= \left\{\frac{1}{\sqrt{2\pi\omega}}\int_{-k}^k e^{-z^2/2\omega} dz\right\}^{-1} .\\
    \end{aligned}
\end{equation*}

To see that $g_k(\omega)$ is monotone increasing, while $g_k(\omega) / \omega$ is monotone decreasing, first note that by a transformation of variables
\begin{equation*}
    g_k(\omega) = \left\{ \frac{1}{\sqrt{2 \pi}} \int_{-k/\sqrt{\omega}}^{k/\sqrt{\omega}} e^{-z^2/2} dz \right\}^{-1}.
\end{equation*}
Since the integrand above is positive, the integral within the braces above must be decreasing in $\omega$. Hence $g_k(\omega)$ is increasing in $\omega$.

Next, see that
\begin{equation*}
\begin{aligned}
    \omega / g_k(\omega) = \sqrt{\frac{\omega}{2 \pi}} \int_{-k}^k e^{-z^2/2\omega} dz
\end{aligned}
\end{equation*}
Differentiating with respect to $\omega$, we obtain
\begin{equation*}
    \frac{1}{2\sqrt{2\pi \omega}} \int_{-k}^k e^{-z^2/2\omega} dz + \sqrt{\frac{\omega}{2 \pi}} \int_{-k}^k \{z^2/(2\omega^2)\}e^{-z^2/2\omega} dz.
\end{equation*}
Since this is positive for all $\omega > 0$, we conclude that $w / g_k(\omega)$ is increasing in $\omega$.
Hence, $g_k(\omega)/w$ is decreasing in $\omega$, and $g_k(\omega)$ satisfies \eqref{eq:grm}.
\end{proof}

\subsection{{Proof of Corollary \ref{cor:wcboundh}}}

\begin{proof}
    As a function of $\omega$, the function
    \begin{equation*}
        C(k, \omega) =\frac{g_{k}(\omega)^2}{\omega f_{k}(\omega)}
    \end{equation*}
    is monotone increasing. Therefore, the limiting covariance matrix of $\sqrt{n}\hat{\bs{\beta}}_H$
    \begin{equation*}
        \left\{ \lim_{n \rightarrow \infty} \frac{1}{n} \sum_{i=1}^n \frac{1}{g_{k}(\omega_i)} \bl{x}_i \bl{x}_i^\top \right\}^{-1} \left\{ \lim_{n \rightarrow \infty} \frac{1}{n} \sum_{i=1}^n \frac{1}{f_{k}(\omega_i)} \bl{x}_i \bl{x}_i^\top \right\} \left\{ \lim_{n \rightarrow \infty} \frac{1}{n} \sum_{i=1}^n \frac{1}{g_{k}(\omega_i)} \bl{x}_i \bl{x}_i^\top \right\}^{-1}
    \end{equation*}
    is bounded above in the Loewner order by
    \small{\begin{equation*}
        C(k, \omega_{\max})\left\{ \lim_{n \rightarrow \infty} \frac{1}{n} \sum_{i=1}^n \frac{1}{g_{k}(\omega_i)} \bl{x}_i \bl{x}_i^\top \right\}^{-1} \left\{ \lim_{n \rightarrow \infty} \frac{1}{n} \sum_{i=1}^n \frac{\omega_i}{g_{k}(\omega_i)^2} \bl{x}_i \bl{x}_i^\top \right\} \left\{ \lim_{n \rightarrow \infty} \frac{1}{n} \sum_{i=1}^n \frac{1}{g_{k}(\omega_i)} \bl{x}_i \bl{x}_i^\top \right\}^{-1}.
    \end{equation*}}
    Since $g_{k}$ satisfies GRM, the weights $\{g_{k}(\omega_i)\}_{i=1}^\infty$ are subscedastic, and the trace and determinant of the matrix above are bounded by the trace and determinant of
    \begin{equation*}
        C(k, \omega_{\max})\left\{ \lim_{n \rightarrow \infty} \frac{1}{n} \sum_{i=1}^n \bl{x}_i \bl{x}_i^\top \right\}^{-1} \left\{ \lim_{n \rightarrow \infty} \frac{1}{n} \sum_{i=1}^n \omega_i \bl{x}_i \bl{x}_i^\top \right\} \left\{ \lim_{n \rightarrow \infty} \frac{1}{n} \sum_{i=1}^n \bl{x}_i \bl{x}_i^\top \right\}^{-1},
    \end{equation*}
    which is proportional to the limiting covariance of $\sqrt{n}\hat{\bs{\beta}}_{\mathrm{OLS}}$.
\end{proof}

\section{EM algorithm for $t$ regression}\label{asec:em}

The algorithm described here has been developed in more detail in \cite{liu_ml_1995}. We include its derivation here for completeness and for consistency with our notation. Consider the hierarchical linear model
\begin{equation*}
    \begin{aligned}
    \bl{y} | \bs{\beta}, \bs{\Omega} &\sim N(\bl{X} \bs{\beta}, \bs{\Omega}),\\
    \omega_1, \dots, \omega_n &\overset{iid}{\sim} IG(\nu/2, \nu \omega_0 / 2).
    \end{aligned} 
\end{equation*}
Given $\omega_1, \dots, \omega_n$, one can find the maximizers $\bs{\beta}^*, \nu^*, \omega_0^{*}$ of the likelihood function
\begin{equation}\label{eq:jlik}
    L(\bs{\beta}, \nu, \omega_0; \bs{\Omega}, \bl{y}) = p(\bl{y} | \bs{\Omega}, \bs{\beta}) p(\bs{\Omega} | \nu, \omega_0).
\end{equation}
We also have a simple form for the conditional distribution of $\omega_i$ given the unknown parameters $\bs{\beta}, \nu, \omega_0$ and the data $\bl{y}$
\begin{equation}\label{eq:conddist}
    \omega_i | \bl{y}, \bs{\beta}, \nu, \omega_0 \sim IG((\nu+1)/2, ((y_i - \bl{x}_i^\top \bs{\beta})^2 + \nu \omega_0)/2).
\end{equation}
This suggests an EM algorithm as a means to obtain maximum marginal likelihood estimates of $\bs{\beta}, \nu, \omega_0$.

The E-step computes the expectation of the log of \eqref{eq:jlik} with respect to the distribution in \eqref{eq:conddist} given a current set of iterates $\bs{\beta}^{(t)}, \nu^{(t)}, \omega_0^{(t)}$. The log likelihood is
\begin{equation*}
-\frac{1}{2}\sum_{i=1}^n [\log(2 \pi \omega_i) + (y_i - \bl{x}_i^\top \bs{\beta})^2/\omega_i] +\frac{1}{2}\sum_{i=1}^n [\nu\log(\nu \omega_0 / 2) - 2\log \Gamma (\nu / 2) - (\nu + 2)\log \omega_i - \nu \omega_0 / \omega_i].
\end{equation*}
Up to addition of constants, this can be written more compactly as
\begin{equation*}
\frac{1}{2} \sum_{i=1}^n [-(\nu+3) \log \omega_i - ((y_i - \bl{x}_i^\top \bs{\beta})^2 + \nu \omega_0)/\omega_i + \nu \log(\nu \omega_0/2) - 2\log \Gamma(\nu/2)].
\end{equation*}
With respect to \eqref{eq:conddist}, this has expectation
\begin{equation}\label{eq:estep}
\begin{aligned}
   \frac{1}{2} \sum_{i=1}^n \left[-(\nu+3) (\log(((y_i - \bl{x}_i^\top \bs{\beta}^{(t)})^2 + \nu^{(t)} \omega_0^{(t)})/2) - \psi((\nu^{(t)}+1)/2)) -\right. \\\left.(\nu^{(t)}+1)\frac{(y_i - \bl{x}_i^\top \bs{\beta})^2 +
    \nu \omega_0}{(y_i - \bl{x}_i^\top \bs{\beta}^{(t)})^2 + \nu^{(t)} \omega_0^{(t)}} + \nu \log(\nu \omega_0/2) - 2\log \Gamma(\nu/2)\right].
\end{aligned}
\end{equation}
The M-step maximizes \eqref{eq:estep} with respect to $\bs \beta, \nu, \omega_0$. The maximizer in $\bs{\beta}$ of \eqref{eq:estep} can be obtained as the solution to a WLS regression
\begin{equation*}
\begin{aligned}
    \bs{\beta}^{(t+1)} = \underset{\bs{\beta}}{\mathrm{argmin}} ~ \sum_{i=1}^n \alpha_i (y_i - \bl{x}_i^\top \bs{\beta})^2, \\
    \alpha_i = [(y_i - \bl{x}_i^\top \bs{\beta}^{(t)})^2 + \nu^{(t)} \omega_0^{(t)}]^{-1}.
\end{aligned}
\end{equation*}
In relation to the discussion at the beginning of Section \ref{sec:robust}, it is apparent that at the final step of the algorithm, $T$, $\bs{\beta}^{(T)}$ is precisely a regularized FLS estimate, with weights equal to $(y_i - \bl{x}_i^\top \bs{\beta}^{(T-1)})^2$ for $i \in \{1, \dots, n\}$ and regularization term equal to $\nu^{(T-1)} \omega_0^{(T-1)}$.

A closed form expression for the maximizer of \eqref{eq:estep} in $\omega_0$ also exists
\begin{equation*}
    \omega_0^{(t+1)} = \left[\sum_{i=1}^n\frac{(\nu^{(t)}+1)/n}{(y_i - \bl{x}_i^\top \bs{\beta}^{(t)})^2 + \nu^{(t)} \omega_0^{(t)}}\right]^{-1}.
\end{equation*}
The maximizer of \eqref{eq:estep} in $\nu$ has no such closed form expression, but it may be obtained numerically with standard software.

\pagebreak

\section{Mixed effects model}\label{asec:mixed}

For a last numerical example, we return to the linear mixed effects model discussed in Section \ref{sec:impl}. Instead of simulating from the heteroscedastic linear model as in the previous examples, we simulate data according to \eqref{eq:mixmod}, where, in this case, the batch indicator matrix $\bl{A}$ has 12 columns, each of which corresponds to a different center from which the \cite{longnecker_association_2001} data were collected. To add some additional complexity, we modify \eqref{eq:mixmod} slightly by letting each center's random intercept have its own variance, so that
\begin{equation*}
    \Cov{\bl{z}} = \mathrm{diag}(\theta_1^2, \dots, \theta_{12}^2),
\end{equation*}
where we draw each $\theta_k^2$ independently from an $IG(1, 3)$ distribution prior to the simulation. This has the effect of changing the order of the eigenvalues of the marginal error covariance relative to the center-specific sample sizes. In such a circumstance, it could be preferable to use an adaptive estimate like the $t$ estimate as opposed to a fixed weight estimate like the one discussed in Section \ref{sec:impl} because the order of the $\theta_k$'s may be unknown.

\begin{figure}
    \centering
    \includegraphics[scale=0.6]{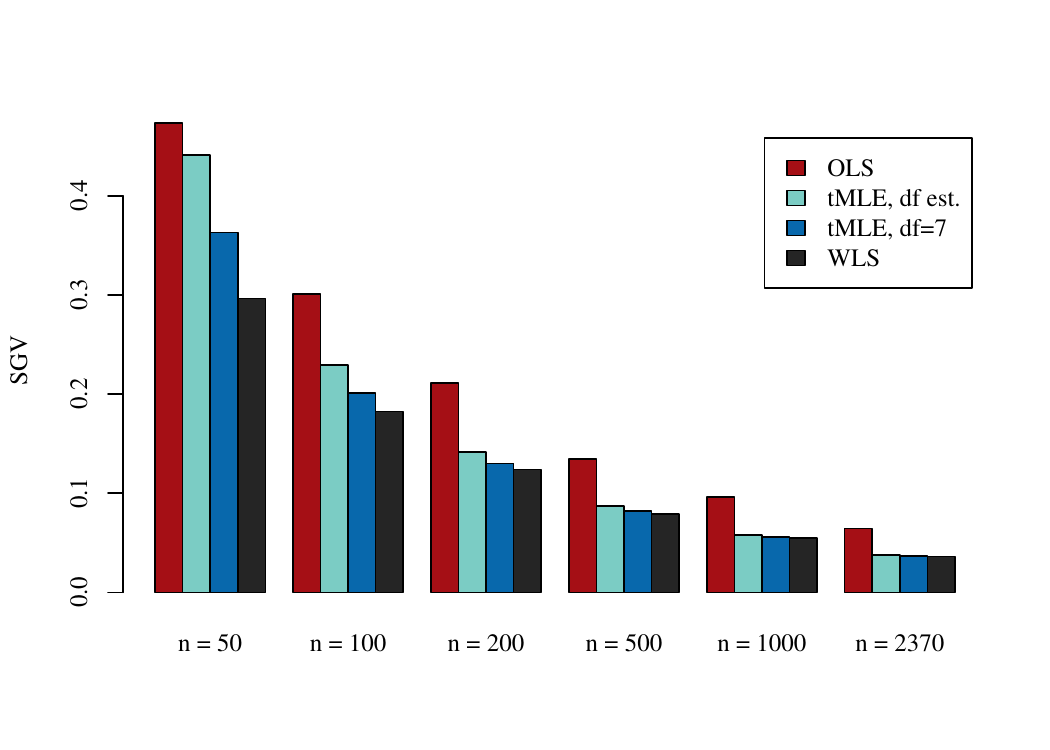}
    \caption{Results of the third simulation study using the \cite{longnecker_association_2001} dataset, which simulates from a modified version of the mixed effects model in \eqref{eq:mixmod}.}
    \label{fig:longnecker_mixed}
\end{figure}

Setting $\theta_0^2 = 1$, we simulate from \eqref{eq:mixmod} with the modifications described above and evaluate the SGVs of the $t$-derived estimates, the OLS estimate, and WLS estimate. Each estimate is computed on data that has been pre-multiplied by the matrix whose columns are the eigenvectors of the marginal error covariance matrix, which in this case is known. As we observed in the previous simulations, the $t$ estimate with $7$ degrees of freedom outperforms OLS for all values of $n$ considered. The $t$ estimate with estimated degrees of freedom performs less favorably for small $n$. See the Appendix, Section \ref{asec:figs} for a complementary plot displaying the mean squared errors for all estimates.

\section{Additional Figures} \label{asec:figs}

We include two supplementary figures showing the mean squared error of the estimates from the second and third simulation studies using the \cite{longnecker_association_2001} dataset. The mean squared error criterion is equal to the total variance of the estimates plus the mean squared bias, which vanishes as $n$ gets large for all estimates. The relative ordering of the mean squared error values mirrors that of the SGV values from the main text.

\begin{figure}[!htb]
    \centering
    \includegraphics[scale=0.55]{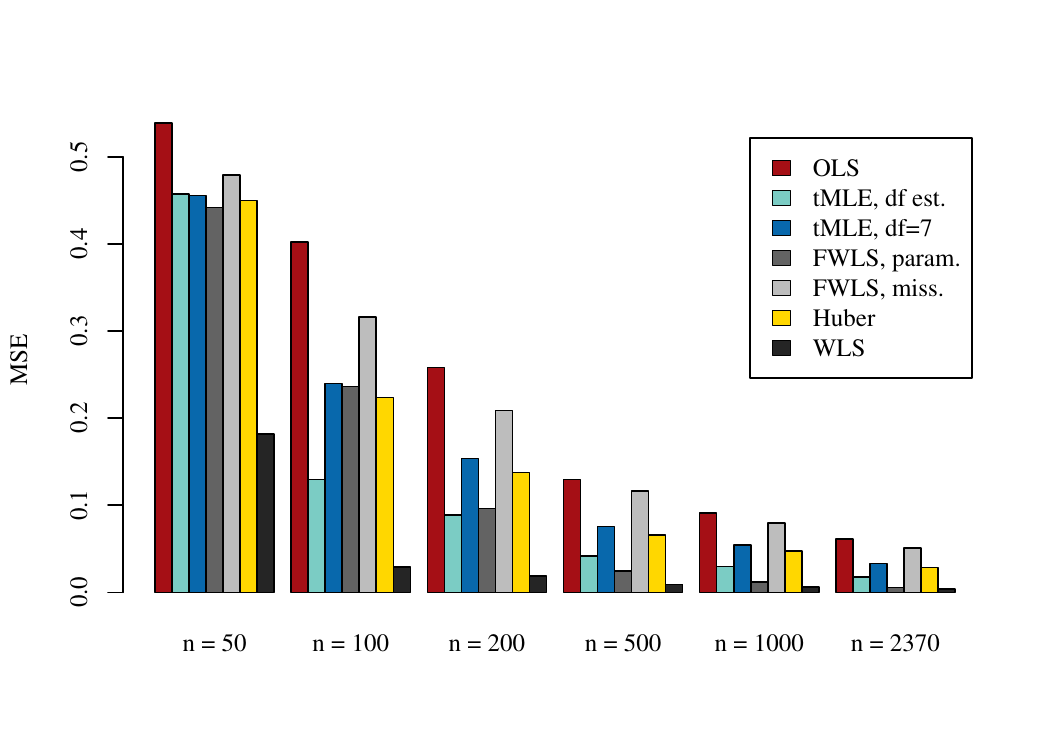}
    \caption{Mean squared error results of the second simulation study using the \cite{longnecker_association_2001} dataset, which uses a parametric model to specify the heteroscedasticity.}
\end{figure}

\begin{figure}[!htb]
    \centering
    \includegraphics[scale=0.55]{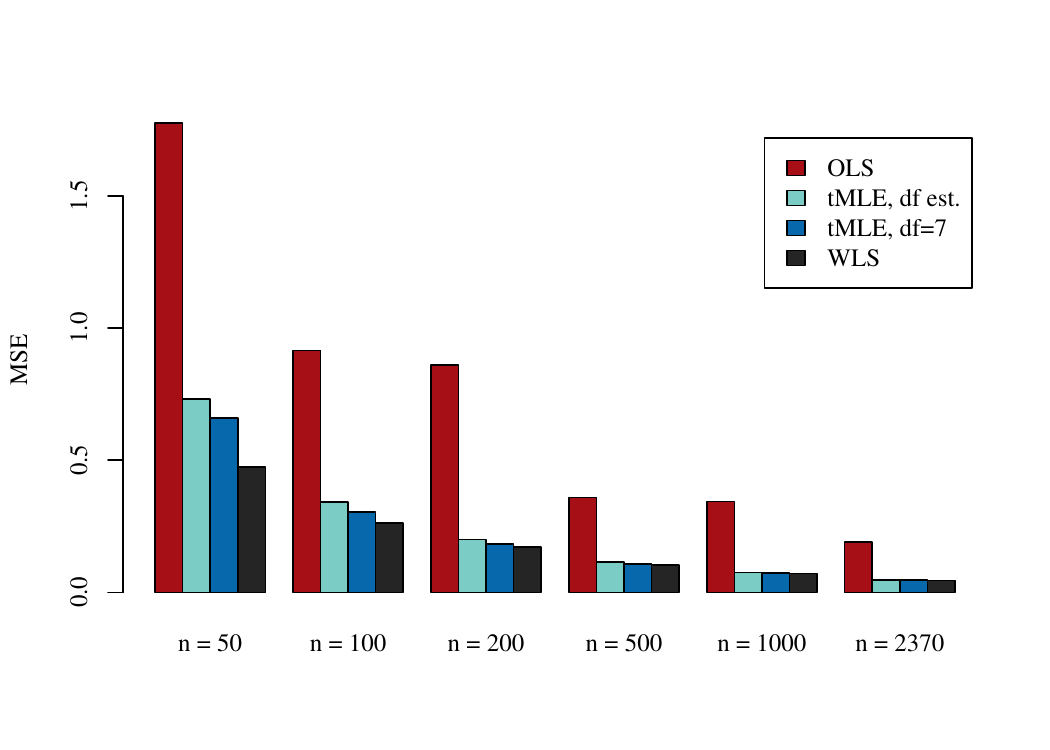}
    \caption{Mean squared error results of the third simulation study using the \cite{longnecker_association_2001} dataset, which simulates from a modified version of the mixed effects model in \eqref{eq:mixmod}.}
\end{figure}

\end{document}